\RequirePackage[l2tabu, orthodox]{nag}

\documentclass[
oneside
,PhDB
]{sapthesis}
\IDnumber{1146957}
\course[Mathematics]{Matematica}
\courseorganizer{Scuola Dottorale in Scienze Astronomiche, Chimiche, Fisiche, Matematiche e della Terra ``Vito Volterra''}
\cycle{XXII}
\submitdate{December 2011}
\copyyear{2011}
\advisor{Prof.~Giuseppe~Rosolini}
\authoremail{caminati@mat.uniroma1.it}
\website{\site}
\examdate{\printdate{20.1.2012}}
\examiner{Prof.~Claudio~Bernardi}
\examiner{Prof.~Carlo~Toffalori}
\examiner{Prof.~Lorenzo~Tortora~de~Falco}

\usepackage{amsmath}
\usepackage{amsthm}
\usepackage{amssymb} 
\usepackage[bibstyle=alphabetic,citestyle=alphabetic,firstinits=true]{biblatex}
\usepackage{extarrows}
\usepackage{graphicx}
\usepackage[]{hyperref}
\usepackage{marginnote}
\usepackage{microtype} 
\usepackage{tikz}
\usepackage{turnstile}
\usepackage{xspace}

\usepackage{varwidth}
\usepackage{tabularx}
\usepackage{verbatim} 
\usepackage[british,inputamerican]{isodate}

\DeclareMathOperator{\ari}{\#}
\DeclareMathOperator{\dom}{dom}
\DeclareMathOperator{\ran}{ran}
\DeclareMathOperator{\rng}{\ran}
\DeclareMathOperator{\con}{Con}
\DeclareMathOperator{\inc}{Inc}

\newcommand{\vv}{\overline{v}}
\renewcommand{\models}[1]{\sdtstile{#1}{}}
\newcommand{\ra}[1]{\xLongrightarrow{\text{\tiny{{#1}}}}}
\newcommand{\A}{automation\xspace}

\newcommand{\C}{\subseteq}

\newcommand{\G}{G\"odel}
\renewcommand{\H}{Henkin}
\renewcommand{\L}{L\"owenheim}
\newcommand{\M}{Mizar\xspace}
\newcommand{\MML}{MML\xspace}
\newcommand{\N}{\mathbb{N}}
\newcommand{\Op}{\op \kern -1.9pt \op}

\newcommand{\Title}{A simplified framework for first-order languages and its formalization in \M{}}
\newcommand{\U}{\bigcup}
\newcommand{\Z}{\mathbb{Z}}
\newcommand{\addF}[2]{\mathcal{E}_{#1}^{#2}}
\newcommand{\addW}[2]{\mathcal{W}_{#1}^{#2}}
\newcommand{\Adj}{Attribute}
\newcommand{\adj}{attribute}
\newcommand{\abs}[1]{\left| #1 \right|}
\newcommand{\asslikeS}[1]{strongly assumptive}
\newcommand{\asslikeW}[1]{weakly assumptive}
\newcommand{\card}[1]{\left| #1 \right|}
\newcommand{\cartprod}{\times}
\newcommand{\Classes}[2]{{#1}/{#2}}
\newcommand{\closed}[1]{{#1}-closed}
\newcommand{\comp}[2]{$\left( #1, #2 \right)$-compatible\xspace}
\newcommand{\Comp}{compatible}
\newcommand{\Con}[2]{\con_{#1} \left( #2 \right)}
\newcommand{\conc}{*}
\newcommand{\cover}{cover}
\newcommand{\curry}[2]{{\vphantom{#1}}_{#2}{#1}}

\newcommand{\cutLike}{cut-like}
\newcommand{\Derivables}[1]{\derivables{\infty}{#1}} 
\newcommand{\derivables}[2]{\iter{\onestep{#2}}{#1}}
\newcommand{\depth}[1]{\left| #1 \right|}

\newcommand{\witnessed}{witnessed}
\newcommand{\witnessRel}[1]{W_{#1}}
\newcommand{\Eq}[2]{\equiv #1 #2}
\newcommand{\eq}[1]{\Eq{#1}{#1}}
\newcommand{\eval}[1]{\overline{#1}}
\newcommand{\termeval}[1]{\overline{#1}}

\newcommand{\entails}[1]{\models{#1}}

\newcommand{\emp}{\emptyset}
\newcommand{\emul}[1]{\geq_{#1}}
\newcommand{\fconc}{\conc \kern -1.9pt \conc}
\newcommand{\Finsets}[2]{\powset{#2}_{#1}}
\newcommand{\finsets}[1]{\mathcal{F}\left( #1 \right)}
\newcommand{\free}{free\xspace}
\newcommand{\freeInt}[1]{\Phi_{#1}}
\newcommand{\funccomp}{\circ}
\newcommand{\funcs}[2]{\mapsFromTo{#1}{#2}}
\newcommand{\henk}[2]{\mathcal{H}_{#1,#2}}
\newcommand{\If}{if and only if\xspace}
\newcommand{\id}[1]{\mathcal{I}_{#1}}

\newcommand{\iQuotient}[2]{\frac{#1}{#2}}
\newcommand{\im}[2]{#1 \left[ #2 \right]}
\newcommand{\Inc}[2]{\inc_{#1} \left( #2 \right)}
\newcommand{\indicator}{\mathrm{1}}
\newcommand{\iter}[2]{{#1}^{\left({#2} \right)}}

\newcommand{\leaves}[1]{\Gamma_{#1}}
\newcommand{\mapsFromTo}[2]{{#2}^{#1}}
\newcommand{\me}{Marco~Caminati}
\newcommand{\mincov}{mincover}

\newcommand{\monotone}{monotone}

\newcommand{\nnot}[1]{\xnot{#1}}

\newcommand{\nor}{\downarrow}
\newcommand{\ol}[1]{\overline{#1}}
\newcommand{\onestep}[1]{\overline{#1}}
\newcommand{\op}{\square}
\newcommand{\pair}[2]{\left( {#1}, {#2} \right) }
\newcommand{\paste}{\lhd}
\newcommand{\peel}[1]{\}\{_{#1}}
\newcommand{\placesof}[2]{{#1}^{\left[ #2 \right]} }
\newcommand{\powset}[1]{\ensuremath{2^{#1}}}

\newcommand{\projR}{\rng}
\newcommand{\provables}[1]{#1}
\newcommand{\proves}[1]{\sststile{#1}{}}
\newcommand{\quotient}[3]{\frac{\hphantom{#2}#1\hphantom{#3}}{#2\hphantom{#1}#3}}
\newcommand{\relcomp}{\bullet}

\newcommand{\restrict}[2]{{\left. #1 \right|}_{#2}}
\newcommand{\rAssumption}{R_0}

\newcommand{\rCut}{R_c}
\newcommand{\rEqRefl}{R_{=}}

\newcommand{\rEqSymm}{R_{\overset{\leftrightarrow}{=}}}
\newcommand{\rEqTrans}{R_{\overset{\Rightarrow}{=}}}
\newcommand{\rEx}{R_{\overset{\rightarrow}{\exists}}}

\newcommand{\rIntuitionisticNightmare}{R_{\not \neg}}
\newcommand{\rFunc}{R_{+}}

\newcommand{\rNor}{R_{\nor}}

\newcommand{\rRel}{R_{\mathcal{R}}}

\newcommand{\rThin}{R_{\cup}}

\newcommand{\rVCut}{R_{\vv}}

\newcommand{\rVThinFork}{R_{< \vv}}

\newcommand{\rWitnessA}{R_{\overset{\leftarrow}{\exists}}}

\newcommand{\reassign}[3]{\frac{#2}{#1}#3}
\newcommand{\rem}[1]{}
\newcommand{\Root}[1]{r_{#1}}
\newcommand{\sdiff}{\backslash}
\newcommand{\seqs}[1]{\ensuremath{G \left( #1 \right)}}
\newcommand{\sound}{sound}
\newcommand{\st}{|}
\newcommand{\site}{http://www.mat.uniroma1.it/people/caminati}
\newcommand{\symbof}[1]{\left\lfloor #1 \right\rfloor}
\newcommand{\symbsubst}[3]{ \frac{#2}{#1}#3 }
\newcommand{\subst}[3]{{#3} \left[ {#1} / {#2} \right]}
\newcommand{\substf}[2]{\left[ #1 / #2 \right]}
\newcommand{\substringsf}{\subtermsf}
\newcommand{\subtermsf}{\odot}
\newcommand{\subterms}[1]{\overrightarrow{#1}}
\newcommand{\substrings}{\subterms}
\newcommand{\terms}[1]{T_{#1}}
\newcommand{\TTermeq}{Equability}
\newcommand{\Termeq}{equability}
\newcommand{\termeq}[2]{\ensuremath{\underset{\scriptscriptstyle #2}{\overset{\scriptscriptstyle #1}{\sim}}}}
\newcommand{\toClass}[1]{\pi_{#1}}
\newcommand{\tuple}{tuple}
\newcommand{\tupleToClass}[2]{\eta_{#1, #2}}
\newcommand{\vcontr}[1]{\nor{\veq}{#1}}

\renewcommand{\vec}[1]{\mathbf{#1}}
\newcommand{\veq}{\eq{\vv}}

\newcommand{\xnot}[1]{\nor{#1}{#1}}
\newcommand{\wffs}[1]{F_{#1}}
\newcommand{\xxnot}[1]{\neg #1}
\newcommand{\notf}{\neg}

\newcommand{\query}[1]{}

\newcommand{\lnote}[1]{\reversemarginpar\query{#1}\normalmarginpar}

\theoremstyle{plain}
\newtheorem{Cor}{Corollary}[subsection]
\newtheorem{Lm}[Cor]{Lemma}
\newtheorem{Prop}[Cor]{Proposition}
\newtheorem{Thm}[Cor]{Theorem}
\theoremstyle{definition}
\newtheorem{Def}[Cor]{Definition}
\newtheorem{Rem}[Cor]{Remark}
\newtheorem{Not}[Cor]{Notation}
\newtheorem{Ex}[Cor]{Example}

\urldef{\urlMmlquery}\url{http://mmlquery.mizar.org/mmlquery/fillin.php?filledfilename=registrations.mqt&argument=number+1}
\urldef{\seminar}\url{http://www.naproche.net/inc/seminar/seminar_2011_SS.php}
\urldef{\wwwMbc}\url{\site}
\urldef{\wwwMizar}\url{http://www.mizar.org}
\urldef{\wwwRosolini}\url{http://www.disi.unige.it/person/RosoliniG/}
\urldef{\wwwNaproche}\url{http://naproche.net/inc/webinterface.php}
\urldef{\wwwProofCheck}\url{http://www.proofcheck.org/}
\urldef{\wwwMml}\url{http://mizar.uwb.edu.pl/version/current/mml/}
\urldef{\wwwXboole}\url{http://mizar.uwb.edu.pl/version/current/mml/xboole_1.miz}

\bibliography{mbc}

\hypersetup{%
pdftitle={\Title
}%
, pdfauthor={\me}%
, pdfstartview=FitBH
, pdffitwindow=true
, pdfkeywords={logic, proof theory, formalization, completeness, model theory, Mizar, satisfiability theorem, \L{}-Skolem theorem, automated proof checking}
, pdflang=en-US
, pdfstartview= 100 100 10000
, pdfsubject={}
, pdfwindowui=true
, colorlinks=true
, linkcolor=blue
, urlcolor=blue
, citecolor=red
, breaklinks = true
, menucolor=brown
}

\title%
{\Title
}
\author{\me{}%
}

\makeindex

\begin{document}

\frontmatter
\maketitle 


\begin{acknowledgments}
Support and guidance from my advisor, Prof.~Giuseppe~Rosolini, have been invaluable.
\\
I am grateful to Prof.~Claudio~Bernardi for helpful advice.
\\
I am indebted to Prof.~Peter~Koepke, who encouraged me with his interest in my research and gave me the opportunity to meet other people working in my area through his gracious hospitality.
\\
I had the luck of making the acquaintance of Flavia Mascioli and Enrico Rogora, among the friendliest and most supportive people I met in my department.
\\
My neighborly fellow graduate students Stefano, Fabio, Linda, Paolo and Andrea supplied good company and interesting discussion.
\\
Finally, I thank rms for being the zealot he is, which I think made this thesis, and the world, better.
Through him I wish to thank every individual who ever contributed to free information.
\end{acknowledgments}


\tableofcontents

\chapter{Introduction}
\label{RefSectIntroduction}

The axioms of set theory in first-order logic, together with a choice of a deductive system, form the foundations on which most mathematicians set their research work.
Thus it is quite natural that also logicians study formalizations of first-order logic and of deductive systems in those same foundations. 
It appears rather surprising that formalizations of deductive systems are still missing.

\lnote{Lievemente alterato il periodo `One possible explanation\ldots{}'. Come suona ora?}
One possible explanation for the lack of a mathematically-flavored treatment of a foundational block of such kind is that its fundamental role in  
the mechanization of mathematics makes research efforts focus on it as a computational tool and divert them from 
rather 
viewing it as an object of mathematical study in its own sake.  
The adjective ``mathematical'' in the last sentence is crucial: indeed, deductive systems are subject to intense study by proof-theorists, but mainly from a computational point of view and with methods typical of computer science. 
While this is certainly critical for the mechanization, it yields as a consequence that deductive systems are, for instance, usually expressed in languages far from set theory (or any other language a mathematician may be accustomed to).

For example, consider the sequent calculus.
Its rules are usually displayed through diagrams like
\begin{align*}
\begin{aligned}
\Gamma &&  && \psi
\\
\hline
\Gamma && \varphi && \psi
\end{aligned}.
\end{align*}
Such diagrams serve well the goals of mechanization, because generally
they are readily rendered into concrete computer languages adopted by
many proof assistants; on the other hand, they are far from being a definition of the rule itself according to set theory. 
Therefore there is a gap between the mechanization of mathematics and the formalization in (one of the most standard) foundations of mathematics.%
\footnote{In alternative formal systems there are rigorous definition of deductive systems; see for example \cite{dawson2010generic}, section 3 and \cite{mikhajlova1998proving}, section 2.}


Indeed, considering the way standard expositions of sequent calculus or natural deduction define what a derivation or a proof is (often such notions are merely introduced with examples, as in \cite{0387908951} (section IV.1), \cite{MR2319486} (chapter 2)), it is  invariably found that it pivots on some notion describing what an atomic step in a derivation is, and that this latter notion is not rigorous, from a strictly formal point of view, because it is based on the diagrams just discussed, rather than on a set-theoretical description of each single rule (in the quotations below, we emphasize the words referring to entities lacking a rigorous symbolic definition): 

\begin{quote}
\ldots the labels at the immediate successors of a node $\nu$ are the \textsl{premises} of a rule \textsl{application}, the label at $\nu$ the \textsl{conclusion}.\newline\mbox{}\hfill\cite{MR1409368}, section 1.3.

By a derivation of $Y$ from $X$ in the system is meant a finite sequence of lines [\ldots] such that for each $i < n$, the line $X_{i+1}$ \textsl{is a direct consequence of the preceding line $X_i$ by one of the inference rules}.\newline\mbox{}\hfill\cite{MR1314201}, chapter XVII.

A formal proof in first-order logic is a finite sequence of statements of the form $X \proves{} Y$ each of which \textsl{follows from the previous statements by one of the rules we have listed}\ldots\newline\mbox{}\hfill\cite{MR2091075}, chapter 1.
\end{quote}

A symptom of this issue is that virtually every exposition of such matters tends to be rather wordy. It is very usual in other realms of mathematics to turn to symbols and strictly defined concepts even in textbooks (compare the neat definition of group in section 2.1 of
\cite{MR1375019}).
This suggests a pragmatic criterion for assessing the affinity of a treatment with standard set-theoretical language of mathematics, basing on the number and complexity of actual implementations of it in a computer-checked proof system adopting set-theoretical foundations. 
For first-order languages and deductive calculus only one such implementation already existed, and it is written in \M{} \cite{goedelcp}: we discuss its shortcomings in sections~\ref{RefSectLogicInMml} and \ref{RefSectDefSep}. 
One major drawback of \cite{goedelcp} is that it does not aim to be a general framework in which arbitrary rules can be inserted, rather it deals with provability with a fixed set of rules, with the only goal of getting to \G{}'s completeness theorem. 

The first task accomplished in this thesis is the formulation of 
first-order logic and sequent calculus in the standard mathematical
foundations of set theory
.
This is done in chapter~\ref{RefSectFormulation}.
Given the view, exposed above, that a good formulation should be effectively formalizable, we try to keep definitions set-theoretically simple, that is, invoking low-level entities. 
This is especially important for sequent calculus, as already discussed.  
Very few assumptions are made on the actual rules adopted, not even that of monotonicity. 
This is a departure from the only theory sharing some traits with the present one which the author is aware of, 
brought out by Tarski in \cite{tarski1923some,tarski1956foundations,tarski1923fundamental}; on other accounts, that theory is more general than the present one, being agnostic with respect to the type of calculus (Hilbert, natural deduction, sequent calculus, etc\ldots) adopted. 
The same chapter also tests this formulation against the proofs of cornerstone applications to model theory and proof theory, like satisfiability, \L{}-Skolem and completeness theorems.
\lnote{Ritoccata frase `we should stress'\ldots; spero di non avere stravolto il significato}
We should stress here that, while it is certainly obvious to every reader of a textbook on first-order logic that a deductive system can be formalized in set theory, frequently it is not so clear if the writer has even considered the problem of how to face that task. 
Thus the treatment results often in something quite regardless of the mathematization of the deductive system.

Chapter~\ref{RefSectFormalization} brings the effort a step further, testing all the contents of chapter~\ref{RefSectFormulation} even more concretely: it passes from the formulation there contained to its mechanically verified formalization, honoring the criterion hinted above.
Given our starting goal of supplying a mathematically-oriented, that is, set-theoretical, formalization of the foundations in themselves, it is natural to choose a verifier adopting set theory axioms and first-order logic.
This reduces the candidate verifiers to a handful, of which \M{} is surely the one with the largest library of already verified mathematics: \M{} Mathematical Library (\MML{}). 
Besides presenting the \M{} verified formalization, chapter~\ref{RefSectFormalization} aims to supply 
(notably in sections \ref{RefSectDefense}, \ref{RefSectRuleSep} and \ref{RefSectDefSep})
concrete instances and discussions of the thesis that in formalizing a piece of mathematics there is more than just precisely stating it and certifying its correctness: see \cite{boyer1994qed} and \cite{MR2463991} for general analysis of how much more there is.

Chapter~\ref{RefSectTechnical} discusses related issues in a more concrete context:
it gives \M{} examples of design principles stated in chapter \ref{RefSectFormalization} and showcases \M{} coding techniques of general applicability.
Notably, section \ref{RefSectAutomations} discusses some general  methods for the \M{} system, whose support for custom \A{} is usually regarded as poor (\cite{wiedijk2007qed}, section 4), aiming at bypassing, in limited circumstances, this shortage, and thus of possible interest for other \M{} users.

This work can also be viewed as a study of how the process of mechanically verifying some theory influences back the theory itself.
Although mechanization of mathematics presents some important differences with respect to writing common software, the main one being that producing executable code is no longer the final goal, it can bring some arguably beneficial factors from the realm of computer programming into the matter being mechanized.
First of all, since `controlling complexity is the essence of computer programming{}' (\cite{kernighan1981software}, page~311), one is led to eliminate all that is not strictly needed, and in general to find approaches minimizing the code to write. 
This has the side effect of accurately evaluating the point at which some notion or construct is really needed, and which results need which notion or construct.
Secondly, and relatedly, once one chooses a specific foundational framework, set theory in our case, he is brought to favor the employment of some theoretical toolkit in lieu of another, if the former is more naturally or more simply expressed in the chosen framework than the latter and, consequently, is somehow better supported by the software used. See point \eqref{RefItem01} of list below.

It is natural to wonder whether the consequences of adopting design principles like the ones stated above are of a merely technical nature, or rather influence the mathematics to an extent possibly interesting in its own sake.
Of such consequences, I put forward some I believe are 
of more than merely technical interest 
in the particular case of the present work, and refer the reader to the corresponding points of the text, and to related discussion scattered along chapter \ref{RefSectFormalization}: 
\begin{enumerate}
\item
The introduction of a definition of language with only two special symbols, and no need for constant symbols.
\item
The distinction between free and bound occurrences of a variable is not needed to prove the theorems mentioned above.
Indeed it is never stated in this work.
\item
Monotonicity of single inference rules can often be replaced by  monotonicity of a ruleset, which is a weaker condition.
Compare definitions \ref{RefDefMonotone1} and \ref{RefDefMonotone2}.
\item
\label{RefItem01}
The definition of sequent derivation and of proof can be substituted by those of derivability (\ref{RefDefDerivability}) and of provability (\ref{RefDefProvable}), respectively.
The latter, in turn, can be made without resorting to the notion of tree, which in set-theory is quite a high-level object, and instead basing on the notion of function iteration.
This alternative view is shown to be reconcilable with the standard, tree-based one by proposition \ref{RefThmTreesToDerivability}.
\end{enumerate}


\mainmatter

\chapter{A set-theoretical treatment of first-order logic}
\chaptermark{First-order logic in set theory}
\label{RefSectFormulation}


This chapter %
illustrates a way of expressing the building blocks of first-order logic in a standard set-theoretical background.
We will define the notions of first-order language, of formulas, of interpretation, of derivation rule, of derivability and provability.
We will also define how to evaluate a formula given an interpretation, how to extract subformulas, how to perform substitutions in a formula.
Finally, we will deploy this machinery to obtain satisfiability, completeness and \L{}-Skolem theorems, after having introduced a suitable set of derivation rules following our definitions.
In chapter \ref{RefSectFormalization} the task of concretely pouring this formulation into \M{} code will be faced. 


\section{Preliminaries}

In this section we fix most of the set-theoretic notations we will be using throughout the chapter. 
Most of them is certainly conventional; all the same we prefer to make sure that the reader is aware of the meaning of each involved symbol.

\label{RefNotationZero}
\hfill
\begin{enumerate}
\item
$\card{X}$ is the cardinality of the set $X$.
\item
$X \cartprod Y$ is the cartesian product of the sets $X$ and $Y$:
\begin{align*}
X \cartprod Y = & \left\{ \pair{x}{y}: x \in X, y \in Y \right\}.
\end{align*}
\item
$\N$, $\Z$ are the sets of natural numbers (including $0 = \emptyset$) and of the integers, respectively. We also write $\Z^+$ for $\N \sdiff \left\{ 0 \right\}$.
\item
$\dom P$ and $\ran P$ denote the domain and range of a given relation $P$.
\item
We will use the terms function, map and mapping interchangeably.
\item
\label{RefEq31}
$\mapsFromTo{X}{Y}$ is the set of the maps from $X$ into $Y$.

\item
Given sets $Y$ and $X$, $\indicator_X^Y$ is the \emph{characteristic function} (also known as \emph{indicator function}) of $X$, defined on $Y$:
\begin{align*}
\indicator_X^Y := \left( \left( Y \sdiff X \right) \cartprod \left\{ 0 \right\} \right)
\cup \left( \left( Y \cap X \right) \cartprod \left\{ 1 \right\} \right).
\end{align*}
Often, $X$ is declaredly a subset of $Y$ and one can write just  $\indicator_X$.

\item
Since $\powset{X} = \left\{ \indicator_{X'}^X: X' \C X \right\}$, it is in a one-to-one correspondence with the power set of $X$; hence we will also abusively write  
$\powset{X}$ for the power set of $X$. 
$\Finsets{n}{X}$ is the set of the subsets of $X$ having $n$ elements, and $\finsets{X} := \bigcup_{n \in \N} \Finsets{n}{X} \C \powset{X}$ is the set of the finite subsets of $X$.
\item
$\peel{X}$ is the map: 
\begin{align*}
\Finsets{1}{X} \ni \left\{ x \right\} \mapsto x \in X;
\end{align*}
often, we just indicate it with $\peel{}$.

\item
$\id{X}$ is the identity map on the set $X$: $\id{X} := \bigcup_{x \in X} \left\{ x \right\} \cartprod \left\{ x \right\}$.
\item
Given sets $X, Y, Z$, and $f \in \funcs{X \cartprod Y}{Z}$,
the unique $F \in \mapsFromTo{X}{\left( \mapsFromTo{Y}{Z}  \right)}$ such that $\left( F \left( x \right)  \right) \left( y \right) = f \left( \pair{x}{y} \right) \forall x \in X , y \in Y$ 
is the currying (known also as sch\"onfinkeling) of $f$. We denote as 
$\curry{f}{x} \in \funcs{Y}{Z}$ its value in $x \in X$: 
\begin{align*}
\curry{f}{x} : Y \ni y \mapsto f \left( \pair{x}{y} \right).
\end{align*}
\end{enumerate}

\begin{Not}
\lnote{Resi $\im{P}{X} $ e $ \restrict{P}{X} $ in forma equazionale.}
Consider a relation $P$ and a set $X$.
We write
$ \restrict{P}{X}$ for the restriction of $P$ to $X$:
\begin{align*}
\lnote{() e riconduci}
\restrict{P}{X} :=
\left( X \cartprod \rng P \right) \cap P
,
\end{align*}
and 
$\im{P}{X}$ for the set of those  elements of $\ran{P}$ corresponding through $P$ to some element of $X$:
\begin{align*}
\im{P}{X} := 
\rng \left( 
\restrict{P}{X} 
\right)
.
\end{align*}

\end{Not}

\begin{Not}
\label{RefNotationCompositions}
$\relcomp$ is the infix symbol for the composition of relations:
$ \im{\left( Q \relcomp P \right)}{X} = \im{P}{\im{Q}{X}} $.
\\
$\funccomp$ is the infix symbol for the composition of functions:
$
g \funccomp f: x \mapsto g \left( f \left( x \right) \right)
$
\end{Not}

\begin{Rem}
\M{} provides one single symbol to denote both relation and function compositions, being able to resolve ambiguities thanks to the typing of the arguments it is applied to.
This resolution would require an extra effort to the reader, so we chose to adopt distinct symbols in \ref{RefNotationCompositions}.
\end{Rem}

\begin{Not}
Given a set $\mathcal{P}$ all elements of which are relations, we define
\begin{align*}
\symbof{\mathcal{P}} := \bigcup_{P \in \mathcal{P}} \ran{P}.
\end{align*}
\end{Not}

\begin{Not}
If $P$ is a relation such that $\ran P \C \dom P $, we can refer to the $n$-th iteration of $P$ for any given $n \in \N$.
We write it as
\begin{align*}
\iter{P}{n}.
\end{align*}
\end{Not}

\begin{Not}[`Functional pasting with right-hand precedence{}']
\label{RefDefPaste}
Given relations $Q$, $P$, set
\begin{align*}
Q \paste P := Q \sdiff \left( \dom P \cartprod \left( \rng Q \right) \right) \cup P.
\end{align*}
\end{Not}

\begin{Rem}
Given two functions $f$, $g$: 
\begin{itemize}
\item
$f \paste g$ is a function;
\item
if $f$ and $g$ agree on $\dom f \cap \left( \dom g \right)$, then $ f \paste g$ $= f \cup g $.
\end{itemize}
\end{Rem}

\begin{Def}[Simple substitution]
\label{RefDefSimpleSubst}
Given $y, y'$ and a function $f$, we define
\begin{align*}
\symbsubst{y}{y'}{f} 
:= \left( \id{\rng f} \paste \left\{ \pair{y}{y'} \right\} \right) \funccomp f \in \mapsFromTo{\dom f} {\left( \rng f \sdiff \left\{ y \right\} \cup \left\{ y' \right\} \right)}.
\end{align*}
\end{Def}

\begin{Def}
\label{RefDefTuple}
Given $n \in \N$, a $n$-\emph{\tuple}\index{\tuple@\tuple} (or just \tuple{}) is a function having $\left\{ j \in \N: j < n \right\}= n$ as a domain. 
By notation \eqref{RefEq31} introduced on page \pageref{RefNotationZero}, then, $X^n$ is the set of all $n$-\tuple s valued in $X$. 
We set $X^+ := \bigcup_{n \in \Z^+} X^n$, and $X^* := X^+ \cup  \left\{ \emptyset \right\}  $.
We will also refer to an element of $X^n$ or $X^*$ as a ($n$-)tuple on $X$.
\end{Def}

\begin{Def}
Given two \tuple s $p, q$, we set 
\begin{align*}
p \conc q :=
\begin{cases}
p & q=\emptyset
\\
p \cup \left(  q \funccomp 
\left\{ \pair {\card{p}}{0},  \ldots, \pair{\card{p}+\card{q}-1}{\card{q} -1} \right\}
\right)
& \text{ otherwise, }
\end{cases}
\end{align*}
that is
\begin{align*}
p \conc q := p \cup \left(  q \funccomp  \left( \restrict{ \left( x \mapsto x - \card{p}  \right) }{\left( \card{p}+\card{q} \right) \sdiff \card{p}} \right)  \right).
\end{align*}
\end{Def}

Note that
\begin{enumerate}
\item
$p \conc q$ is still a \tuple: the functions $p$ and $\left( \restrict{ \left( x \mapsto x - \card{p}  \right) }{\left( \card{p}+\card{q} \right) \sdiff \card{p}} \right)$ have as domains respectively $\card{p}$ and $ \left( \card{p} + \card{q} \right) \sdiff \card{p}$: being the latter mutually disjoint, $p \conc q$, as a union of the former functions, is still a function; moreover, its domain is precisely the union of 
$\card{p}$ and $ \left( \card{p} + \card{q} \right) \sdiff \card{p}$.
\item
$\ran \left( p \conc q \right) = \left(  \ran p  \right) \cup \ran q$.
\end{enumerate}
Hence the mapping $ \pair{p}{q} \mapsto p \conc q$ is a binary operation on $X^*$:

\begin{Def}
Given $X$, set 
$\conc_X := X^* \cartprod X^* \ni \pair{p}{q} \mapsto p \conc q.
$
\end{Def}

$\conc$ is associative. That is:
\begin{align*}
\left( p \conc q \right) \conc r = p \conc \left( q \conc r \right)
\end{align*}
for any three \tuple s $p, q, r$. 
This permits to consider $\left( X^*, \conc_X, \emptyset \right) $ as a monoid, also abusively indicated with $X^*$.
Similarly, $X^+$ will be also used to denote the sub-semigroup 
$ \left( X^+, \restrict{\left( \conc_X \right) }{ \left(  X^+  \right)} \right) $ of $X^*$ on $X^+$. 

Thanks to its associativity, $\conc_X$ naturally yields a homomorphism 
$(X^*)^* \to X^*$, which restricts to a homomorphism $(X^+)^+ \to X^+$; 
both are denoted by \label{RefDefMultiCat} $\fconc_X$.

\begin{Not}
\label{RefNotationConc}
When no ambiguity arises, we reserve to employ the following shorthand  notations, writing
\begin{enumerate}
\item
\label{Ref063}
$x$ instead of $\left\{  \pair{0}{x}  \right\} \in X^1 \C X^+$;
\item
\label{Ref064}
$p q$ in place of $p \conc q$;
\item
$p \conc q \conc r$ for $\left( p \conc q  \right) \conc r= p \conc \left( q \conc r \right)$.
\item
\label{Ref067}
$\conc$ instead of $\conc_X$; 
\item
\label{Ref068}
$\fconc$ instead of $\fconc_X$.
\end{enumerate}
\end{Not}

\begin{Rem}
It would be natural to add to the ones in \ref{RefNotationConc} the further shorthand notation identifying the distinct mappings $\conc$ and $\fconc$ under one symbol. We refrain from doing so: those distinct functions will occasionally appear together, so being able to resolve between them arguably adds clarity when this happens.
\end{Rem}

\section{Languages}
\label{RefSectSyntax}
\begin{Def}
\label{RefDefLanguage}
A \emph{language}\index{language} is a triple  $\left( \ari, \equiv, \nor \right)$, where $\ari$ is an integer-valued function  and $\equiv$ is an element of its domain, such that
\begin{enumerate}
\item
\label{RefEq30}
${\ari} \left( \equiv \right) = -2$;
\item
\label{RefEq28}
$ \nor \notin \dom {\ari}$;
\item
\label{RefEq29}
${\ari}^{-1} \left( \left\{ 0 \right\} \right)$ is not finite.
\end{enumerate}
\end{Def}

\begin{Not}
\label{RefNotationLanguage}
\hfill\\
\begin{itemize}
\item
$\ari$ is called the \emph{arity}\index{arity} 
of the language, and $\left\{ \nor \right\} \cup \dom {\ari}$ is called the \emph{symbol set} of the language.
\item
$\equiv$ is called the \emph{equality symbol}\index{equality symbol} of the language, and $\nor$ the \emph{logical connective}\index{logical connective} of the language.
\item
Given a language $S$, we also denote by $S$ its symbol set (so that, e.g. $S^*$ is the free monoid on the latter, and $\conc_S$ the operation of this monoid); 
when needed, we may use a subscript to  refer explicitly to the arity, equality symbol or logical connective of $S$: $S =\left( \ari_S, \equiv_S, \nor_S \right)$.
\item
The elements of $\#_S^{-1} \left( \left\{ 0 \right\} \right)$ are called the \emph{literals} of $S$, those of $\#_S^{-1} \left( \Z \sdiff \left\{ 0 \right\} \right)$ its \emph{compounders}.
\end{itemize}
\end{Not}

\begin{Def}[The set of terms of depth not exceeding $n$]\hfill\\
\label{RefDefTerms}
\lnote{Mancava parentesi}
Given a language $S$, we recursively construct the following countable family of sets of tuples on $S$:
\begin{align*}
\terms{S,0} & := 
\mapsFromTo{1}
{
\left( 
\im{\ari^{-1}} { \left\{ 0 \right\} }
 \right)
}
\\
\terms{S,n+1} & := \terms{S,n} \cup  
\bigcup_{o \in \im{\ari^{-1}}{ \Z^+ }} \im{\conc}
{ \left\{
\left\{ 
\pair{0}{o}
\right\}
\right\} \cartprod \im {\fconc} {\left( \terms{S,n} \right)^{\ari\left( o \right)} } }.
\end{align*}
\end{Def}

\begin{Def}[Terms of a language]
\index{terms}
\begin{align*}
\terms{S} &:= \bigcup_{n \in \N} \terms{S,n}.
\end{align*}
\end{Def}

\begin{Def}[The set of formulas of depth not exceeding $n$]\hfill\\
\label{RefDefWffs}
Given a language $S$, we recursively construct the following countable family of tuples on $S$:
\begin{align*}
\wffs{S,0} &:= \U_{ r \in \im{\ari^{-1}} { \Z^- } } 
\im{\conc} 
{ \left\{ 
\left\{ 
\pair{0}{r} 
\right\}
\right\} \cartprod 
\im{\fconc}
{\left( \terms{S} \right)^{ \left| \ari \left( r \right) \right| }} 
}
\\
\wffs{S,n+1} & := \wffs{S,n} \cup 
\im{\conc}
{\left\{ 
\left\{ 
\pair{0}{\nor} 
\right\}
\right\} \cartprod
\im{\conc}{
\wffs{S,n} \cartprod \wffs{S,n}}
}
\cup 
\im{\conc}{  
\mapsFromTo{1}{\left( \im { \ari ^ {-1} }{\left\{ 0 \right\}}  \right) }
\cartprod \wffs{S,n} }.
\end{align*}
\end{Def}


\begin{Def}[The formulas, or well-formed tuples, or wffs of a language]
\index{formulas}
\begin{align*}
\wffs{S}:= \bigcup_{n \in \N} \wffs{S,n}.
\end{align*}
\end{Def}

\begin{Def}[Depth of a term and of a formula]
The \emph{depth} of a term\index{depth of a term} $t$ of $S$ is written $\depth{t}$, and defined as the least $n \in \N$ such that $t \in \terms{S,n}$.
\\
The depth of a formula\index{depth of a formula} $\psi$ of $S$ is written $\depth{\psi}$, and defined as the least $n \in \N$ such that $\psi \in \wffs{S,n}$.
A formula of depth zero is said to be \emph{atomic}\index{atomic formula}.
\end{Def}

\begin{Def}
\label{RefDefSequent}
Given a language $S$, we consider the set
\begin{align*}
\seqs{S} := \finsets{\wffs{S}} \cartprod \wffs{S}.
\end{align*}
An element $\pair{\Gamma}{\varphi}$ of $\seqs{S}$ is called a \emph{sequent}\index{sequent} of the language $S$;
 $\Gamma$ is styled the \emph{antecedent}\index{antecedent} of the sequent, $\varphi$ its \emph{succedent}\index{succedent}.
\end{Def}

\section{Comments and an example}
\label{RefSectExampleLanguage}
The definition of a first order language presented here, and the subsequent ones, have been devised with an eye to \M{} formalization: as little and as basic as possible objects were pushed into them.
In particular, the following points should be emphasized:

\begin{itemize}
\item
The first design choice is to use polish notation: for example $x>y+z$ becomes $>x+yz$. 
This is a common choice in software and in formalization for its simplicity;  both \cite{QC_LANG1} and \cite{ZF_LANG} adopt it as well.
\item
There is no quantification symbol. 
This does not mean that we cannot quantify, of course:
\emph{existential} quantification is indicated by heading a formula with a literal symbol, and this gives rise to no ambiguity.\\ 
Of course, universal quantification can be rendered via existential and negation constructs, as is customarily done; we shall soon an applied instance of this in the example about group axioms below.
\item
There is no native distinction between free and bound variables. 
What's more, there is not even a distinction between variables and constants symbols. 
There are only symbols of arity zero, which are called literals, and symbols of non zero arity, called compounders. 
To be more precise, the distinction is left to the semantics, in the sense that a constant becomes a variable exactly when it is caught by quantification inside a formula.
\item
Arity yields \emph{signed} natural numbers, with the convention that negative arity symbols are relational (predicate) compounders and positive arity symbols are operational compounders. 
The absolute value of the arity will indicate the actual arity of the compounder.
In many treatments, (even inside \M{}'s library, see \cite{QC_LANG1}) there are no operational symbols, which can always semantically be emulated by relational (predicate) symbols, but this makes the definition of well-formed formulas (wff) and, most importantly, that of \free{} interpretation, trickier. 
\item
There is only one logical connector, that is NOR, here denoted by `Peirce arrow' ($\nor$).
This suffices since NOR is universal (functionally complete), as is its dual NAND ($ \uparrow $ or `Sheffer stroke').
\item
Term substitution, \ref{RefDefTermSubst}, will be defined by leveraging the pre-existing notions of reassignment, of evaluation of an interpretation, and of \free{} interpretation.
Additionally, simple substitution, \ref{RefDefSimpleSubst}, is preferred to it when sufficing, as in definition of
$\witnessRel{}$, \ref{RefDefWitnessed}, and of rule $\rWitnessA$, see \ref{RefDefRules}.
\end{itemize}
Therefore, in definitions regarding syntax and semantics, we can take advantage of dealing with just two special symbols: equality and NOR; notably in treating wff formulas and evaluation (see \ref{RefSemantics}), this will be a life-saving simplification. 

To give one among the simplest illustrations, let us rephrase in this language the group axioms, using $\N$ as a symbol set, $1$ as $\equiv$, $0$ as $\nor$, and an arity $f \colon \Z^+ \to \Z$ given by

\begin{align*}
f(n) := &
\begin{cases}
\textsf{-2} & \text{ if n=1 }
\\
\textsf{2} & \text{ if n=2 }
\\
\textsf{0} & \text{ otherwise }
\end{cases}
\end{align*}

Direct translation might result bewildering, so let us first list axioms in standard human-friendly form (on the left in the table below) and in an intermediate jargon made by combining polish notation with shortcut symbols $\exists, \forall, =, +$ for quantifiers and compounders:
\begin{align*}
\forall a,b,c \ a (bc) = (ab)c && \forall 3 \forall 4 \forall 5 \ =+3+45++345
\\
\forall a \ ea=a && \forall 4 \ =+344
\\
\forall a \exists b \ ba=e && \forall 4 \exists 5 \ =+543.
\end{align*}
Finally, we pass to the real coding first by rendering $\forall x \phi$ as $\lnot \exists x \phi$, $\lnot \phi$ as $\nor \phi \phi$, $\exists x \phi$ as $x \phi$, and subsequently by substituting $=$, $+$ respectively with $1$, $2$, in the end obtaining some nasty strings:

\begin{align}
\label{RefGroupTheory}
\begin{aligned}
\begin{aligned} 
03040512324523425512324523425405123245234255123245234253
\\
04051232452342551232452342540512324523425512324523425
\end{aligned}
\\
0412344412344
\\
045124534512453,
\end{aligned}
\end{align}

where the first, exceedingly long axiom has been split across two lines.

This shows how the absence of auxiliary boolean connectors and quantifiers makes even trivial formulas go wildly verbose.
Note that none of the three axioms uses more than seven literals, so we have been able to unambiguously use decimal representation for $\mathbb{N}$.
Also compare the role of the symbol '$3$' in expressing first and second axioms: in the first case it is quantified and thus used as a variable, while in the second it acts as a constant (the unity of the group) since it is not quantified. 
Not having distinguished between constants and variables permits reusing a literal symbol in both ways, as long as the corresponding constant does not appear in the formula in which the symbol is used as a variable.
Given our goals, we do not care much about readability of the language: all that matters is that any first-order theory is expressible in the language, and that a proof calculus being both sound and \emph{complete} (that is, powerful enough to prove any consequence of a first-order theory) is provided, which we did with completeness theorem \ref{RefThmCompleteness2}.
Under these constraints, we sought for the design maximizing simplicity and neatness of formalization.

\section{Formal definition of derivation rule}

\begin{Def}[Rules and rulesets]
\label{RefDefRule}
A \emph{derivation rule}\index{rule}, or \emph{inference rule} for $S$ is any map 
$\powset{\seqs{S}} \to \powset{\seqs{S}}$.
A \emph{ruleset}\index{ruleset} of $S$ is a set of derivation rules, that is, a subset of $\mapsFromTo{\left(\powset{\seqs{S}}\right)} {\left( \powset{\seqs{S}} \right)}$.
\end{Def}
\begin{Not}[Character reservations; abbreviations for writing sequents]
\hfill\\
\label{RefReserve}
\begin{itemize}
\item
As a rule, we will use the letter $S$ to indicate a language, and $X$ to indicate a generic set. 
\item
We conventionally agree to reserve (unless otherwise specified) some characters according to the type of $S$-related objects we will want to denote:
\begin{itemize}
\item
$s$ for an element of $\dom \ari_{S}$,
\item
$v$ for a literal,
\item
$w$ for a tuple on $S$,
\item
$t$ for a term,
\item 
$\Gamma$ for a \emph{finite} set of formulas,
\item
$\varphi, \psi$ for a formula,
\item
$\Psi$ for a set of formulas,
\item
$\sigma$ for a sequent,
\item
$\Sigma$ for a set of sequents,
\item
$R$ for an inference rule, and
\item
$D$ for a ruleset.
\end{itemize}
Subscripts or superscripts will be added when needed.
\item
A sequent $\pair{\Gamma}{\varphi}$ will be often represented as $\Gamma \quad \vdash \quad \varphi$.
\item
When writing a sequent, the following abbreviations can be adopted:
\begin{align*}
\Gamma_1 \quad \Gamma_2 \quad \vdash \quad \varphi && \text{ in lieu of }&& \Gamma_1 \cup \Gamma_2 \quad \vdash \quad \varphi
\\
\Gamma \quad \psi \quad \vdash \quad \varphi && \text{ in lieu of }&& \Gamma \cup \left\{ \psi \right\} \quad \vdash \quad \varphi.
\end{align*}
\item
The turnstile symbol $\left( \vdash \right)$ parting antecedent from succedent can be omitted when adopting the foregoing abbreviations for writing a sequent.
\end{itemize}
\end{Not}

\begin{Ex}
Consider $\Gamma_1 := \left\{ \psi_1, \psi_2 \right\}, \Gamma_2:= \left\{ \psi_3 \right\}, \Gamma := \Gamma_1 \cup \Gamma_2$.
\\
Here is a list of some of the notations rendering the sequent $ \pair{\Gamma}{\varphi}$, obtainable by combining shorthand notations introduced in  \ref{RefReserve}:
\begin{align*}
\psi_1 \quad \psi_3 \quad \psi_2 \quad  \psi_3 \quad \vdash & \quad \varphi
\\
\Gamma_1 \quad \Gamma_2 \quad  \vdash & \quad \varphi
\\
\left\{ \psi_1, \psi_2 \right\} \quad \psi_3 \quad  \vdash & \quad \varphi
\\
\left\{ \psi_1, \psi_2, \psi_3 \right\} \quad  \vdash & \quad \varphi
\\
\psi_1 \quad \psi_2 \quad \psi_3 & \quad \varphi.
\end{align*}
\end{Ex}

\subsection{An example of ruleset}

\begin{Def}
\label{RefDefRules}
We introduce some particular derivation rules of the language $S$ by specifying the way each acts on a given $\Sigma \C \seqs{S}$:

\noindent\hrulefill
$$\begin{array}{r@{}c@{}l}
\rAssumption \left( \Sigma \right) &:= 
& \left\{ \pair{\Gamma}{\varphi} : \Gamma = \left\{ \varphi \right\} \right\}
\\[1ex]
\rThin \left( \Sigma \right) &:=
& \left\{ \pair{\Gamma}{\varphi} :
\exists \pair{\Gamma'}{\varphi} \in \Sigma \st \Gamma' \C \Gamma
\right\}
\\[1ex]
\rEqRefl \left( \Sigma \right) &:=
& \left\{ \pair{\Gamma}{\varphi} :
\exists t \st \Gamma = \emptyset \text{ and } \varphi=\equiv t t
\right\}
\\[1ex]
\rEqSymm \left( \Sigma \right) &:=
& \left\{ \pair{\Gamma}{\varphi} :
\exists t_1, t_2 \st 
\Gamma=\left\{ \equiv t_1 t_2 \right\} \text{ and } \varphi= \equiv t_2 t_1 
\right\}
\\[1ex]
\rEqTrans \left( \Sigma \right) &:= 
& \left\{ \pair{\Gamma}{\varphi} :
\exists t_1, t_2, t_3 \st \Gamma=\left\{ \equiv t_1 t_2, \equiv t_2 t_3 \right\} \text{ and }
\varphi=\equiv t_1 t_3
\right\}
\\[1ex]
\rFunc \left( \Sigma \right) &:=
&
\{ \pair{\Gamma}{\varphi} :
\exists n \in \Z^{+}, s \in S, 
\vec t, \vec t' \in \left( \terms{S} \right)^n \st
\varphi= \equiv s \fconc \left( \vec t \right)    s  \fconc \left( \vec t' \right)
\text{ and }
\\[1ex]&&
\multicolumn1r
{\Gamma = \left\{ \equiv \vec t \left( j \right) \vec t' \left( j \right),
j \in n
\right\}
\}}
\\[1ex]
\rRel \left( \Sigma \right) &:=
&
\{ \pair{\Gamma}{\varphi} :
\exists n \in \Z^+, s \in S, \vec t, \vec t' \in \left({\terms{S}}\right)^n \st 
\varphi= s \fconc \left( \vec t' \right)
\text{ and } \\[1ex]&&
\multicolumn1r
{n= - \ari \left( s \right) \text{ and }
\Gamma = 
\left\{ \equiv \vec t \left( j \right) \vec t' \left( j \right), 
j \in n
\right\} 
\cup 
\left\{ s \fconc \left( \vec t \right) \right\}
\}}
\\[1ex]
\rNor \left( \Sigma \right) &:=
& \left\{ \pair{\Gamma}{\varphi} :
\exists \varphi_1, \varphi_2, \varphi_3, \varphi_4 \in \wffs{S} \st \Gamma = \left\{  
\nor {\varphi_1} {\varphi_2} , \nor {\varphi_3} {\varphi_4}
\right\} \text{ and } \varphi= \nor {\varphi_2} {\varphi_3}
\right\}
\\[1ex]
\rWitnessA \left( \Sigma \right) &:=
&
\left\{ \pair{\Gamma}{\varphi} :
\exists v, v_1, v_2, \psi, \Gamma' \st \left( \Gamma' \cup \left\{ \symbsubst{v_1}{v_2}{\psi} \right\}, \varphi \right) \in \Sigma
\text{ and } 
\right.
\\[1ex]&&
\multicolumn1r
{
\varphi = \xnot {\equiv v v }\text{ and } 
\Gamma=\Gamma' \sdiff \left\{\symbsubst{v_1} {v_2}{\psi}  \right\} \cup \left\{ v_1 \psi \right\}
\text{ and }}
\\[1ex]&&
\multicolumn1r
{\left. v_2 \notin \symbof{\Gamma' \cup \left\{ \psi \right\}}
\right\}}
\\[1ex]
\rCut \left( \Sigma \right) &:=
&
\{ \pair{\Gamma}{\varphi} :
\exists \psi_1, \psi_2 \st 
\pair{\Gamma \cup \left\{ \psi_1 \right\}}{\psi_2},
\pair{\Gamma \cup \left\{ \psi_1 \right\}}{\nor \psi_2 \psi_2} \in \Sigma
\\[1ex]&&
\multicolumn1r
{\text{ and } \varphi = \xnot{\psi_1}
\}}
\\[1ex]
\rIntuitionisticNightmare \left( \Sigma \right) &:=
& \left\{ \pair{\Gamma}{\varphi} :
\pair {\Gamma}{\xnot {\xnot \varphi}} \in \Sigma
\right\}. 
\end{array}$$

\end{Def}

\begin{Not}
When wanting to express the particular language $S$ relative to which one of the rules defined in \ref{RefDefRules} is to be meant, we adjoin its name $S$ to the rule's subscript, as in 
${\rEqRefl}_{, S}$.
\end{Not}

\section{Formal definitions of derivability and provability}
\label{RefSectProvability}
If we want to formalize results about completeness of first-order languages in a first-order language like \M or set theory, we first have to rigorously define in it what a proof is.
It turns out that it is both sufficient and convenient to establish the notion of provability rather than that of proof.



\begin{Def}
\label{RefDefOneStep}
Given a ruleset $D$ of $S$, we define the following derivation rule of $S$:
\begin{align}
\onestep{D} : \Sigma \mapsto 
\bigcup_{R \in D} R\left( \Sigma \right).
\end{align}
\end{Def}





\begin{Def}
\label{RefDefDerivability}
A sequent belonging to $\derivables {n}{D} \left( \Sigma  \right)$ will be said to be \emph{derivable}\index{derivability} from $\Sigma$ through $D$ in $n$ steps.

The set of all sequents derivable from $\Sigma$ through $D$ will be indicated with $\Derivables{D} \left( \Sigma \right)$:
\begin{align*}
\Derivables{D} \left( \Sigma \right) := \bigcup_{n \in \N} \derivables{n}{D} \left( \Sigma \right).
\end{align*}
\end{Def}

\begin{Def}[Formal definition of provability]
\label{RefDefProvable}
Given $S$, $X$ and $D$, we set 
\begin{align*}
\provables{D} \left( X \right) := \projR 
\left( 
\powset{X} \cartprod \wffs{S} \cap 
\left( \Derivables{D} \left( \emptyset \right)  \right) \right) \C \wffs{S}.
\end{align*}
As well as $\varphi \in \provables{D}\left( X \right)$, one can also write $X \proves{D} \varphi$,
and say that $X$ proves $\varphi$ in $D$, or that $\varphi$ is provable\index{provability} from $X$ in $D$.
\end{Def}

\begin{Rem}
\label{RefDefProvable2}
Equivalently, 
$X \proves{D} \varphi $ \If there is a sequent $ \pair{\Gamma}{\varphi} \in \Derivables{D} \left( \emptyset \right) $ such that $\Gamma \C X$.
\\
Alternatively, since $ \derivables{0}{D} \left( \emptyset \right) = \emptyset$, 
$X \vdash_D \varphi$ 
if and only if there are $n \in \N$, $ \Gamma \in \finsets{\wffs{S}}$ such that $\pair{\Gamma}{\varphi} \in \derivables{n+1}{D} \left( \emptyset  \right).$ 
\end{Rem}


\begin{Rem}
In \ref{RefDefSequent} we defined  sequents of $S$ as having for an antecedent a finite subset of $ \wffs{S} $%
.
Other conventions are to define sequents having either multisets or tuples of formulas as an antecedent. 
The one adopted here, however, involves lower-level objects than the other two, if one works in a set-theoretical  formal framework as we are doing.
Moreover, it allows dispensing with introducing exchange and contraction rules.
\end{Rem}


\begin{Def}
$X$ is said to be \emph{deductively closed}\index{closed}\index{deductively closed} with respect to $D$ (or just $D$-closed) if%
\begin{align*}
\provables{D} \left( X \right) \C X.
\end{align*}
\end{Def}

\section{Justification of diagrams}
\label{RefSectDiagrams}
\begin{Def}
\label{RefDefMonotone1}
A rule $R$ of $S$ is said to be \emph{\monotone{}}\index{\monotone{} rule} if it is monotone with respect to the partial order $\C$ of $\seqs{S}$; that is, for any $ \Sigma_1, \Sigma_2 \C \seqs{S} $ such that $\Sigma_1 \C \Sigma_2$, it is:
\begin{align*}
R \left( \Sigma_1 \right) \C R \left( \Sigma_2 \right). 
\end{align*}
\end{Def}

\begin{Rem}
Any hypothesis requesting some rule to be \monotone{} will always be made explicit.
However, all the concrete examples of rule we will introduce will be \monotone{}.
This will be often exploited without explicit mention.
\end{Rem}

\begin{Def}
\label{RefDefRuleMajor}
Given a derivation rule $R$ of $S$ and $n \in \N$, we write 
\begin{align*}
R \leq n
\end{align*}
to mean that for any $\Sigma_2 \C \seqs{S}, \sigma \in R \left( \Sigma_2 \right)$, there is $ \Sigma_1 \C \Sigma_2$ with $\card{\Sigma_1} = n$ such that 
$\sigma \in R \left( \Sigma_1 \right)$.
In this case we say that $n$ is an \emph{upper bound}\index{upper bound for a rule} for $R$.

If $R \leq 0$ we say $R$ is an \emph{axiom}\index{axiom}.
\end{Def}

All the rules introduced in \ref{RefDefRules} are \monotone\ and have $2$ as an upper bound (some even admit $1$ as an upper bound, with many  being just axioms): roughly speaking, this means that each sequent belonging to the image of a given $\Sigma$ through one of those rules can be derived by applying that rule just to a suitable subset of $\Sigma$ having cardinality either $0$ (for those rules being axioms), $1$ or $2$.

This allows us to introduce schematic diagrams succinctly illustrating how each of our rules work by a graphical arrangement describing its action on a given generic pair of sequents (or either respectively on a single sequent or on the empty set).
This description is done simply by listing above a horizontal line the input sequent(s), if any, and below it the output sequent:\index{diagram representation of a rule} 

\begin{center}
\label{RefDefRuleDiagrams}
\begin{align}
\notag
\\
\notag
\begin{aligned}
\rAssumption
	\begin{aligned}
	\\
	\hline
	\varphi && \vdash &&  \varphi
	\end{aligned}
&&
\rThin
	\begin{aligned}
	\Gamma && \vdash && \varphi
	\\
	\hline
	\Gamma' && \vdash && \varphi
	\end{aligned}
&& \text{where } \Gamma \subseteq \Gamma'
\end{aligned}
\\\notag
\\ 	
\notag
\begin{aligned}
\rEqRefl
	\begin{aligned}
	\\
	\hline
	\hphantom {\eq t} && \vdash && \eq t
	\end{aligned}	
&&
\rEqSymm
	\begin{aligned}
	\\
	\hline
	\equiv t_1 t_2 &&  \vdash && \equiv t_2 t_1
	\end{aligned}
&&
\rEqTrans
	\begin{aligned}
	\\
	\hline
	\equiv t_1 t_2 && 
	\equiv t_2 t_3 &&  \vdash && \equiv t_1 t_3
	\end{aligned}
\end{aligned}
\\\notag
\\
\notag
\begin{aligned}
\rFunc
	\begin{aligned}
	\\
	\hline
	\equiv t_1 t'_1 && \ldots && \equiv t_n t'_n
	&& \vdash && 
	\equiv s t_1 \ldots t_n s t'_1 \ldots t'_n
	\end{aligned}
&& 
\text{where } n= \ari (s) \in \Z^{+}
\\\notag
\\
\rRel
	\begin{aligned}
	\\
	\hline
	s t_1 \ldots t_n && \equiv t_1 t'_1 && \ldots && \equiv t_n t'_n
	&& \vdash && 
	s t'_1 \ldots t'_n
	\end{aligned}
&& 
\text{where } n= - \ari (s) \in \mathbb{Z}^+
\end{aligned}
\\\notag
\\
\notag
\rNor 
\begin{aligned}
\\
\hline
\nor {\varphi_1} {\varphi_2} && \nor {\varphi_3} {\varphi_4} && \vdash && \nor {\varphi_2} {\varphi_3}
\end{aligned}
\\\notag
\\
\notag
\begin{aligned}
\rWitnessA
	\begin{aligned}
	\Gamma && \symbsubst{v_1}{v_2}{\varphi} && \vdash && \xnot {\equiv v v}
	\\
	\hline
	\Gamma && v_1 \varphi && \vdash && \xnot {\equiv v v}
	\end{aligned}
&& \text{where $v_2$ does not occur in $\Gamma, \varphi $}
\end{aligned}
\\\notag
\\
\notag
\begin{aligned}
\rCut
	\begin{aligned}
	\Gamma && \varphi && \vdash && \psi 
	&& && &&
	\Gamma && \varphi && \vdash && \nor \psi \psi
	\\
	\hline
	&& && && \Gamma &&   && \vdash && \nor \varphi \varphi
	\end{aligned}
&&
\rIntuitionisticNightmare
	\begin{aligned}
	\\ 
	\Gamma && \vdash &&  \xnot{\xnot \varphi}
	\\
	\hline
	\Gamma && \vdash && \varphi
	\end{aligned}
\end{aligned}
\end{align}
\end{center}

We lastly observe that such a suggestive representation of rules is effective  because each of the latter works in a syntactically simple manner: hence its action is immediately conveyed by glancing at the variations of the morphological patterns between the sequent schematas above and below the horizontal line.

This is one of the reasons for splitting derivations into several applications of different rules: otherwise we could have helped the trouble of introducing the definitions of a ruleset $D$ and of the derived rule $\onestep{D}$ (see \ref{RefDefOneStep}), and rather state directly \ref{RefDefDerivability} and \ref{RefDefProvable} in terms of a single generic, comprehensive rule taking the place of $\onestep{D}$.

\subsection[Derivation trees. Proofs]{Justification for the introduction of derivation trees. Formal definitions of derivation and proof}
\subsubsection{Motivation}
Although the notions of derivability and provability of \ref{RefSectProvability} will turn out, throughout chapters \ref{RefSectFormulation} and \ref{RefSectFormalization}, to be perfectly sufficient to formalize (see~\cite{fomodel4}) all our results, a human is usually more comfortable in carrying out and conveying reasonings involving those notions if he  adopts some interface to them more resembling a calculation.
To this end, we will obtain a graphical representation of such calculi in form of oriented trees, which matches the diagrams introduced in \ref{RefSectDiagrams}.
We start with a rather elementary notational convention.
%
%
For a generic rule $R$ and sequents $\sigma_1, \sigma_2$, instead of writing $\sigma_2 \in R \left( \left\{ \sigma_1 \right\} \right)$, we just write   
\begin{align*}
\frac{\sigma_1}{\sigma_2} & R.
\end{align*}
Now, the convenience we gain is that such writings can be `piled up{}', resulting in a more natural way of expressing a succession of rule applications.
When dealing with rules not all of which are bounded by $1$, such `piles{}' become \emph{trees}\index{derivation tree}.

\subsubsection{Formal definitions}

The aforementioned trees, which will be referred to as \emph{derivations}, can be rigorously defined in terms of derivability (\ref{RefDefDerivability}) and of a basic subset of the usual gear of graph theory.
First of all we note that we need the assumption that the rules
involved are monotone to proceed. In fact the fitting notion is for rulesets.

\begin{Def}
\label{RefDefMonotone2}
\index{\monotone{} ruleset}
A ruleset $D$ is said to be \emph{\monotone{}}\index{\monotone{} ruleset} \If the rule $ \onestep{D}$ is \monotone{}.
\end{Def}

Now the reader may want to consult some reference on graphs (e.g., \cite{knuth1997art}, section~2.3.4.2, `Oriented trees') for the few standard definitions and results about trees we will need in what follows.

\begin{Not}
Given an oriented tree $T := \left( V, E \right)$, we denote with $\depth{T}$ its depth, with
$\Root{T}$ the root of $T$, that is the only element of $ V \sdiff \ran{E}$, and with 
$\leaves{T} $ the set $V \sdiff \dom {E}$ (that is, the set of the leaves of $T$ %
).
\end{Not}

\begin{Def}[Recursive definition of a derivation tree]
\label{RefDefDerivationTree}
Let $T := \left( V, E \right)$ be an oriented tree with $n+2$ vertices for some $n \in \N$. 
Denote as $ r_1, \ldots, r_l$ the distinct elements of  
$ \im{E}{\left\{ \Root{T} \right\}}$ (that is, the vertices of $T$ having depth $1$),
with $T_j, j=1, \ldots, l$ the unique oriented sub-tree of $T$ having $r_j$ as a root.
\\
Let $f$ be a function with $V \C \dom f$ and $\ran f \C \seqs{S}$. 
We say that $(T,f)$ is a \emph{$D$-derivation}\index{derivation}, where $D$ is a ruleset of the language $S$, if
\begin{itemize}
\item
$ \depth {T} = 1$ and $r \in R \left( \im {f} {\leaves{T}} \right) $ for some $R \in D$.
\item
$\depth{T} = m+2$ for some $m \in \N$,  there is $R \in D$ such that $f(\Root{T}) \in R \left( \im{f} 
{\left\{ r_1, \ldots, r_l \right\}} 
\right)$, and, for each $j \in l+1$%
:
\begin{itemize}
\item
$\depth{T_j}=m+1$, and
\item
$ ( T_j, f )$ is a $D$-derivation.
\end{itemize}
\end{itemize}
\end{Def}


The final step is to state the existence of a $D$-derivation  as sufficient condition for the derivability of its root sequent from the set of its leaves according to the rules of $D$:

\begin{Prop}
\label{RefThmTreesToDerivability}
If $D$ is a \monotone{} ruleset of $S$ and $( T = \left( V, E \right), f ) $ is a $D$-derivation of depth $n+1 \in \Z^+$, then $f (\Root{T}) \in \derivables{n+1}{D} \left( \im{f}{\leaves{T}} \right)$.
\end{Prop}

\begin{proof}
By induction on $n$. For $n=0$ the thesis is immediate from \ref{RefDefDerivationTree}.
\\
Assume $n=m+1$ for some $m \in \N$. As done in \ref{RefDefDerivationTree}, denote with $r_1, \ldots, r_l$ the distinct elements of $\im{E}{\left\{ \Root{T} \right\} }$, and with $T_j, j=1, \ldots, l$ the unique oriented subtree of $T$ having $r_j$ as root.
\\
By \ref{RefDefDerivationTree}, each $( T_j, f)$ is a $D$-derivation and has depth $m+1$; thus, by the inductive hypothesis, 
$f(r_j) \in \derivables{m+1}{D} \left( \im{f}{ \leaves{T_j}} \right)$.
\ref{RefDefDerivationTree} also says that $ f(\Root{T}) \in R \left( \im{f}{\left\{ r_1, \ldots, r_l \right\}} \right)$ for some $R \in D$. Hence 
$f( \Root{T} ) \in \onestep{D} \left( \im{f}{\left\{ r_1, \ldots, r_l \right\}} \right)$. 
Since $\onestep{D}$ is \monotone{}, we conclude 
\begin{align}
\label{Ref062}
f ( \Root{T}  ) \in \onestep{D} \left( \bigcup_{j} \derivables{m+1}{D} \left( \im{f}{\leaves{T_j}} \right) \right).
\end{align}
Now, $ \derivables{m+1}{D}$ is \monotone{} as well, and $\im{f}{\leaves{T_j}} \C \im{f}{\leaves{T}}$, yielding 
\begin{align*}
\bigcup_{j} \derivables{m+1}{D} \left( \im{f}{\leaves{T_j}} \right) \C \derivables{m+1}{D} \left( \im{f}{\leaves{T}} \right).
\end{align*}
Using this (again along with the fact that $\onestep{D}$ is \monotone{}) inside \eqref{Ref062}, we get 
$ f( \Root{T} ) \in \derivables{m+2}{D} \left( \im{f}{\leaves{T}} \right)$.
\end{proof}

\begin{Def}
A $D$-\emph{proof}\index{formal definition of proof} is a $D$-derivation $\left( T, f \right)$ such that 
\begin{align*}
\im {f} {\leaves{T}} \C \derivables{1}{D} \left( \emptyset \right).
\end{align*}
\end{Def}

\section{Elementary results concerning derivability and provability}

\begin{Prop}
\label{RefThmMonotonicityIteration}
Given $D_1 \C D_2$ such that at least one among $D_1$ and $D_2$ is \monotone{}, for any $\Sigma_1 \C \Sigma_2$ and any $ n \in \N$ it holds
\begin{align*}
\derivables{n}{D_1} \left( \Sigma_1 \right) \C \derivables{n}{D_2} \left( \Sigma_2 \right).   
\end{align*}
\end{Prop}

\begin{proof}
By induction on $n$. 
For $n=0$, we have trivially
$ \derivables{0}{D_1} \left( \Sigma_1 \right) = \Sigma_1 \C \Sigma_2 = 
\derivables{0}{D_2} \left( \Sigma_2 \right)$.
Now assume $n=m+1$ for some $m \in \N$.
\begin{align*}
\derivables{n}{D_1} \left( \Sigma_1 \right) = 
\onestep{D_1} \left( \derivables{m}{D_1} \left( \Sigma_1 \right) \right) 
\C
\begin{cases}
\overset{\text{\tiny{!}}}{\C} \onestep{D_1} \left( \derivables{m}{D_2} \left( \Sigma_2 \right) \right)
\overset{\text{\tiny{\ref{RefDefOneStep}}}}{\C} \onestep{D_2} \left( \derivables{m}{D_2} \left( \Sigma_2 \right) \right)
\\
\overset{\text{\tiny{\ref{RefDefOneStep}}}}{\C} \onestep{D_2} \left( \derivables{m}{D_1} \left( \Sigma_1 \right) \right)
\overset{\text{\tiny{!}}}{\C} \onestep{D_2} \left( \derivables{m}{D_2} \left( \Sigma_2 \right) \right)
\end{cases}
\\
= \derivables{n}{D_2} \left( \Sigma_2 \right).
\end{align*}
In the reasoning above, upper branch is for the case $D_1$ \monotone{}, lower branch is for the case $D_2$ \monotone{}. 
In both, `!' denotes the passages invoking inductive hypothesis together with (respective) monotonicity hypothesis.
\end{proof}

\begin{Prop}
\label{RefThmMonotonicityIterationEmpty}
If $D$ is \monotone{}, then
\begin{align*}
\derivables{n}{D} \left( \emptyset \right) \C \derivables{n+1}{D} \left( \emptyset \right) 
\end{align*}
for any $n \in \N$.
\end{Prop}

\begin{proof}
By induction on $n$:
\begin{align*}
\derivables{0}{D} \left( \emptyset \right) = \emptyset \C \derivables{1}{D} \left( \emptyset \right).
\end{align*}
Assuming $ \derivables{n}{D} \left( \emptyset \right) \C \derivables{n+1}{D} \left( \emptyset \right) $, one has
\begin{align*}
\onestep{D} \left( \derivables{n}{D} \left( \emptyset \right) \right) \C \onestep{D} \left( \derivables{n+1}{D} \left( \emptyset \right) \right)
\end{align*}
by monotonicity.
\end{proof}

\begin{Def}
\label{RefDefEmul}
Ruleset $D_2$ \emph{emulates}\index{emulation} ruleset $D_1$ \emph{from} $\Sigma$ (written $D_2 \emul{\Sigma} D_1$) if
\begin{align*}
\bigcup_{n \in \Z^+} \derivables{n}{D_1} \left( \Sigma \right) \C \bigcup_{n \in \Z^+} \derivables{n}{D_2} \left( \Sigma \right) .
\end{align*}
$D_2$ emulates $D_1$ (written $D_2 \emul{} D_1$) if, for each $\Sigma \C \seqs{S}$: 
\begin{align*}
D_2 \emul{\Sigma} D_1. 
\end{align*}
\end{Def}

\begin{Rem}
\label{RefThmEmulTransitive}
Given $\Sigma \C \seqs{S}$, the relation $\emul{\Sigma}$ is transitive:
\begin{align*}
D_2 \emul{\Sigma} D_1 && \text{ and } &&  D_3 \emul{\Sigma} D_2 && \text{ imply } && D_3 \emul{\Sigma} D_1. 
\end{align*}
\end{Rem}


\begin{Cor}[of \ref{RefThmMonotonicityIteration}]
\label{RefThmInclusionImpliesEmulation}
If $D_1 \C D_2$ and at least one of $D_1$ and $D_2$ is \monotone{}, then 
\begin{align*}
D_2 \emul{} D_1.
\end{align*}
\end{Cor}

\begin{Prop}
\label{RefThmEmulationAndProvability}
If $X \proves{D_1} \varphi $ and $D_2 \emul{\emptyset} D_1$, then $X \cup Y \proves{D_2} \varphi$.
\end{Prop}

\begin{Cor}[of \ref{RefThmInclusionImpliesEmulation} and \ref{RefThmEmulationAndProvability}]
If at least one of $D_1, D_2$ is \monotone{}, then
\begin{align*}
D_1 \C D_2 && \text{ and } && X \proves{D_1} \varphi && \text{ imply } && X \proves{D_2} \varphi. 
\end{align*}
\end{Cor}

\begin{Cor}[of \ref{RefThmEmulationAndProvability}]
\label{RefThmEmulatedByClosedIsClosed}
If $X$ is \closed{$D_2$} and $D_2 \emul{\emp} D_1$, then $X$ is \closed{$D_1$}.
\end{Cor}

\section{Semantics}
\label{RefSectSemantics}
It is not difficult to show that 
$ \restrict{\fconc}{ \mapsFromTo{1}{\wffs{S}} \cup \mapsFromTo{2}{\wffs{S}}} $ is one-to-one, and, by recursion on $n$ (see section \ref{RefSectSubTerms}), that $\restrict{\fconc}{ \left( \terms{S} \right) ^ n}$ is one-to-one; this permits defining 
the following three functions.

The first one is in $\mapsFromTo {\terms{S}} {\left( \terms{S} ^* \right)} $:
\begin{Def}[Subterms of a term]
\label{RefDefSubtermsOfTerm}
\begin{align*}
\subtermsf_0 :=  t \mapsto
\begin{cases}
\emptyset & \text{ if } t \in \terms{S,0}
\\
\left( 
\left( \restrict{\fconc}
{ \left(   \terms{S}  \right)^{ \ari \left( t\left( 0 \right) \right) } }
 \right)^{-1} 
\funccomp
{\left( \curry{\conc}{ \left( \restrict{t}{1} \right) } \right)}^{-1} 
 \right)
\left( t \right) 
 & \text{otherwise.} 
\end{cases}
\end{align*}
\end{Def}
The second function is in $ \mapsFromTo {\wffs{S,0}} {\left( \terms{S}^*  \right)}$:

\begin{Def}[Subterms of an atomic formula]
\label{RefDefSubtermsOfAtomic}
\begin{align*}
\subtermsf_1 :=  \psi_0 \mapsto
\left( \restrict{\fconc}
{ \left(   \terms{S}  \right)^{ - \ari \left( \psi_0 \left( 0 \right) \right) } }
 \right)^{-1} \left( 
{\left( \curry{\conc}{ \left( \restrict{\psi_0}{
1} \right) } \right)}^{-1} \left( \psi_0 \right) 
 \right).
\end{align*}
\end{Def}
Finally, the third function is in 
$ \mapsFromTo{\wffs{S} \sdiff \wffs{S,0}} 
{ \left( \left( \wffs{S} \right)^1 \cup \left( \wffs{S} \right)^2  \right) }$:

\begin{Def}
\lnote{Resa piu' compatta}
\label{RefDefSubformulas}
\begin{align*}
\substringsf_2 := \psi \mapsto
\left(
{
\left( 
\restrict{\fconc}{
\left( 
\mapsFromTo{1}{\wffs{S}} \cup \mapsFromTo{2}{\wffs{S}} \right)}
\right)
}^{-1}
\funccomp
{
\left( 
\curry{\conc}{
\restrict{\psi}{1}
}
\right)
}^{-1}
 \right)
\left( \psi \right).
\end{align*}
\end{Def}

In \ref{RefDefSubtermsOfTerm}, \ref{RefDefSubtermsOfAtomic} and \ref{RefDefSubformulas}, we took advantage of the easy fact that $ \curry{\left( \conc_X  \right)}{x} $ is one-to-one for any $X$ and $x \in X^*$.

Since $\subtermsf_0$, $\subtermsf_1$ and $ \subtermsf_2$ have mutually disjoint domains, we can refer to the function resulting from their union, denoting it simply as $\subtermsf$:

\begin{Def}[Sub-tuples of a term or wff]
\begin{align*}
 \subtermsf := \subtermsf_0 \cup \subtermsf_1 \cup \subtermsf_2 \in \funcs{\left( \terms{S} \cup \wffs{S} \right) }{ \left( \left( \terms{S} \right) ^ * \cup \left( \wffs{S} \right)^1 \cup 
\left( \wffs{S} \right)^2 \right)}.
\end{align*}
\end{Def}


\begin{Not}
We will often write $\substrings{w}$ in place of $\substringsf \left( w \right)$.
If $w$ is a non-atomic formula, $\substrings{w}$ are the \emph{subformulas}\index{subformula} of $w$, while if it is an atomic formula or a term, $\substrings{w}$ are the \emph{subterms}\index{subterm} of $w$.
\end{Not}

\begin{Rem}
If $\psi$ is a non-atomic formula, then the number of its subformulas, $ \card{\substrings{\psi}} $, is either $1$ (if $\psi\left( 0 \right)$ is a literal) or $2$ (if $\psi \left( 0 \right) = \nor $).
\end{Rem}

\begin{Def}[Interpretation and universe]
\label{RefDefInterpretation}
Given a language $S$, an \emph{interpretation of $S$}\index{interpretation} is a function $i$ for which there is a non empty set $U$ (called the \emph{universe}\index{universe} of the interpretation) such that 
\begin{align*}
\forall s \in \dom{\ari_S},\ i \left( s \right) \in 
\begin{cases}
U^{\left( U^{\#\left( s \right)} \right)} & \text{if } \#\left( s \right) \geq 0
\\
\left\{ 0, 1 \right\}^{\left( U^{-\#\left( s \right)} \right)} & \text{if } \#\left( s \right) <  0.
\end{cases}
\end{align*}
\end{Def}

\begin{Not}
The symbol $i$, with optional subscripts and superscripts, will be reserved for generic interpretations from now on, unless otherwise specified.
\end{Not}

\begin{Rem}
Every interpretation has exactly one universe.
\end{Rem}

\begin{Rem}
\label{RefRemInterpretationOfLiteral}
According to \ref{RefDefInterpretation}, an interpretation having universe $U$ assigns to each literal a map of the form $\left\{  \pair {\emptyset} {u}  \right\} $, where $u \in U$, rather than assigning to it directly the value $u$.
\end{Rem}

\begin{Ex}[The \free{} interpretation]
\label{RefDefFreeInt}
Given $ X $ and a language $S$, the \emph{\free}\index{\free{} interpretation} interpretation of $S$ given by $X$ is the interpretation of $S$ having $\terms{S}$ as universe and defined thus:
\begin{align*}
\freeInt{X} := \dom{\ari} \ni s \mapsto
\begin{cases}
\left( \curry{\conc}{\left\{ \pair{0}{s}  \right\} }  \right)  \funccomp \left( \restrict{\fconc}{ \left( { \terms{S} }^{\ari \left( s \right)} \right) } \right)
& \ari \left( s \right) \ge 0
\\
\indicator_{X}^{\wffs{S}} \funccomp \left( \curry{\conc}{\left\{ \pair{0}{s}  \right\}}  \right)  \funccomp \left( \restrict{\fconc}
{ \left({ \terms{S} }^{ - \ari \left( s \right)} \right) } \right)
& \ari \left( s \right) < 0.
\end{cases}
\end{align*}
\end{Ex}

\begin{Not}[Reassignment of a literal in an interpretation]
\lnote{Mancava parentesi}
\label{RefNotationReassign}
Given an interpretation $i$, an element $u'$ of its universe, and a literal $v$, we introduce the shorthand notation
\begin{align*}
\reassign{v}{u'}{i} :=
i \paste  \left\{ \pair{v}{
\left\{ 
\pair {\emptyset}{u'}
\right\}
}  \right\}
\end{align*}
designating a new interpretation with the same universe of $i$, called a \emph{reassignment}\index{reassignment} of $v$ in $i$.
\end{Not}

\begin{Def}[Evaluation of terms and atomic formulas]
\label{RefDefEvalAtomic}
Given an interpretation $i$ of universe $U$, we define 
\begin{align*}
\termeval{i} \left( t_0 \right) := \left( i \left( t_0 \left( 0 \right) \right)  \right) \left( \emptyset \right) && \forall t_0 \in \terms{S,0},
\end{align*}
then recursively:
\begin{align*}
\termeval{i} \left( t \right) :=
\left( i \left( t \left( 0 \right) \right)  \right) \left( \eval{i} \funccomp  \subterms{ t }  \right),
&& t \in \terms{S,n+1};
\end{align*}
and finally, given $ \psi_0 \in \wffs{S,0}:$
\begin{align*}
\termeval{i} \left( \psi_0 \right) :=
\begin{cases}
\left( i \left( \psi_0 \left( 0 \right) \right)  \right) \left( \eval{i} \funccomp  \subterms{ \psi_0 }   \right)
& \psi_0 \left( 0 \right) \neq \equiv
\\
1 & \psi_0 \left( 0 \right) = \equiv \text{ and } \eval{i} \left( \subterms{\psi_0} \left( 0 \right) \right) 
= \eval{i} \left( \subterms{\psi_0} \left( 1 \right) \right)
\\
0 & \text{ otherwise. }
\end{cases}
\end{align*}
\end{Def}

\begin{Def}[Evaluation of non-atomic formulas]
\label{RefDefEvalCompound}
Given an interpretation $i$ of universe $U$, we recursively define
\begin{align*}
\eval{i} \left( \psi \right) :=
\begin{cases}
1 &
\begin{aligned}
\text{ if } \exists v \in \im{{\ari}^{-1}}{\left\{ 0 \right\}}, u \in U \st \quad 
\left(  
v=\psi \left( 0 \right) \text{ and } \eval{\reassign{v}{u}{i}} \left( \substrings{\psi} \left( 0 \right) \right) = 1 
\right)
\end{aligned}
\\
1 & \text{ if } \psi \left( 0 \right) = \nor \text{ and } 
\eval{i} \funccomp \substrings{\psi} = 
2 \cartprod \left\{ 0 \right\}
\\
0 & \text{ otherwise}
\end{cases}
\end{align*}
for every $\psi \in \wffs{S} \sdiff \wffs{S,0}$.
\end{Def}

\begin{Def}
Merging \ref{RefDefEvalAtomic} with \ref{RefDefEvalCompound}, we in the end obtain a function
\begin{align*}
\termeval{i} :  \left( \terms{S} \cup \wffs{S} \right)  \to \left( U \cup \left\{ 0, 1 \right\} \right),
\end{align*}
called the \emph{evaluation}\index{evaluation} of the interpretation $i$.
\end{Def}

\begin{Not}[Model, or satisfaction, relation]
\label{RefDefSatisfaction}
Instead of writing $ \im{\restrict{\termeval{i}}{\wffs{S}}}{X} \C \left\{ 1 \right\}$, one often writes $ i \models{S} X$, or simply $ i \models{} X$, and says that $i$ is a \emph{model}\index{model} of $X$, or that $i$ \emph{satisfies}\index{satisfaction relation} $X$.
\end{Not}

\begin{Def}
\label{RefDefSound}
A ruleset $D$ is \emph{\sound{}} if
$X \proves{D} \varphi$ and  $i \models{} X $ imply $\eval{i} \left( \varphi \right) = 1$.
\end{Def}

\begin{Rem}
Any hypothesis requesting some generic ruleset to be \sound{} will always be made explicit.
However, all the concrete examples of ruleset we will introduce will be \sound{}.
\end{Rem}

\begin{Def}[Depth-recursive definition of term substitution in a formula]
\label{RefDefTermSubst}
Given $v$ and $ t$, define the map 
$ \substf {v}{t} : \wffs{S} \to \wffs{S} $ as follows:
\begin{align*}
\substf{v}{t} \left(\varphi_0 \right) := 
\left( \restrict{\varphi_0}{\left\{ 0 \right\}} \right)  \conc \left( \fconc
\left(
 \left( \eval{ \reassign{v}{t}{\freeInt{\emptyset}}  } \right)  \funccomp 
 \subterms{ \varphi_0 } 
\right)
\right)
\end{align*}
for any atomic formula $\varphi_0$; then, given  $\varphi \in \wffs{S,n+1} \sdiff \wffs{S,n}$, recursively on $n$:
\begin{align*}
\substf{v}{t} \left( \varphi \right) :=
\begin{cases}
\left( \restrict{\varphi}{\left\{ 0 \right\}} \right) \conc  \left( \fconc \left( \substf{v}{t} \funccomp \substrings {\varphi} \right) \right) & 
\text{ if } \varphi \left( 0 \right) = \nor
\\
\left\{ \pair{0}{v'}  \right\} \conc \left( \substf{v}{t} \left( 
\symbsubst{\varphi \left( 0 \right)}{v'}{ \left( \substrings { \varphi }  \left( 0 \right) \right) }
\right) 
\right)
& 
\begin{aligned}
\text{otherwise, where }  
\\
v' \notin \left\{ v  \right\} \cup \symbof{\left\{ t,   \substrings{\varphi }  \left( 0 \right)  \right\}}.
\end{aligned}
\end{cases}
\end{align*}
\end{Def}

There is a glitch in \ref{RefDefTermSubst}, in that its outcome actually depends on the choice of the literal $v'$ appearing in its definiens. 
This is immaterial, however, since the different formulas obtained by varying $v'$ are all good candidates to be the substitution result for our purpose: as long as the outcome obeys substitution lemma (see \ref{RefThmSubstLemma2}), it is acceptable.
So we chose not to specify this dependance in \ref{RefDefTermSubst}. 
To make matters rigorous, one could fix a suitable choice function $ \eta : \left( \powset{\im{\ari^{-1}}{\left\{ 0 \right\}}} \right) \sdiff \left\{ \im{\ari^{-1}}{\left\{ 0 \right\}} \right\} \ni X \mapsto x \in 
\left( \im{\ari^{-1}}{\left\{ 0 \right\}} \right) \sdiff X
$ and define ${\substf{v}{t}}_{\eta}$ by substituting $v'$ with $\eta \left( 
\left\{ v  \right\} \cup \symbof{\left\{ t,   \substrings{ \varphi}  \left( 0 \right)  \right\}}
\right)$ inside the definiens of \ref{RefDefTermSubst}, which, however, would probably result a bit too cluttered this way.
In \M{} one utterly bypasses such problems generically related to the dependence on some choice function by using the construct \verb|the|%
, which provides an object of the given type, undefined yet usable as if it was defined.
It should be noted, however, that this device as well is merely a convenient way, offered by \M{}, to invoke the axiom of choice: \cite{MizarChoice}.

\begin{Not}
We will often write $ \subst{v}{l}{\psi}$ instead of $ \substf{v}{l} \left( \psi \right)$.
\end{Not}

We now introduce a further derivation rule we will need.

\begin{Def}
\label{RefDefRuleEx}
\begin{align*}
\rEx \left( \Sigma \right) := 
& \left\{ \pair{\Gamma}{\varphi} :
\exists v, t, \psi \st
\Gamma = \left\{ \subst{v}{t}{\psi} \right\} \text{ and } \varphi= v \psi
\right\}.
\end{align*}
\end{Def}
Since $\rEx \leq 0$, we can depict $\rEx$ via a diagram as those from section \ref{RefSectDiagrams}:
\begin{Not}
\begin{align*}
\rEx \quad
	\begin{aligned}
	\\
	\hline
	 \subst{v}{t}{\psi} &&   \vdash  && v \psi
	\end{aligned}
\end{align*}
\end{Not}


\section{\H{} interpretation}
\label{RefSectHenkin}
\subsection{Quotients}
\label{RefSectQuotients}
\begin{Def}
\label{RefDefCompatible}
Let $P, Q$ be relations, $f$ be a function. We say that $f$ is \comp{P}{Q}
if, given  
$\pair{x}{y} \in \dom f \cartprod \left( \dom f \right) \cap P$, it is  
$\pair{f\left( x \right)}{f \left( y \right)} \in Q$.
\end{Def}

\begin{Rem}
In \M{} code, the keyword \verb|-compatible| being already in use, the attribute \verb|-respecting| is used instead.
\end{Rem}

\begin{Def}
Given a non empty relation $P$, we consider the map
\begin{align*}
\toClass{P} : \dom P \ni x \mapsto \im{P}{\left\{ x \right\}} \in \powset{\rng P}.
\end{align*}
Given a set $X$ and a relation $P$ such that $X = \dom P$, we set
\begin{align*}
\Classes{X}{P} := \rng \left( \toClass {P} \right).
\end{align*}
\end{Def}

\begin{Rem}
If $P$ is an equivalence relation over $X$, $\Classes{X}{P}$ is the set of the equivalence classes of $P$ (hence a partition of $X$), and $\toClass{P}$ maps each element of the domain of $P$ to the unique equivalence class including it.
\end{Rem}

\begin{Def}[Quotient of a relation]
\label{RefDefQuotient}
Let $O, P, Q$ be relations, with $P$ and $Q$ non empty. 
The quotient of $O$ by $\pair{P}{Q}$ is defined as:
\begin{align*}
\quotient{O}{P}{Q} := \left\{
\pair{p}{q} \in \rng \left( \toClass{P} \right) \cartprod \left( \rng \left( \toClass{Q} \right) \right) : 
p \cartprod q \cap O \neq \emp 
\right\}.
\end{align*}
\end{Def}

\begin{Prop}
\label{RefThmCompatibleFunction}
Let  $E, F$ be non empty equivalence relations.
\\
If $f \in$ 
$\mapsFromTo{\dom E}{\left( \dom F \right)}$ 
is \comp{E}{F}, then
\begin{align*}
\quotient{f}{E}{F} \in \mapsFromTo{\rng  \toClass{E} }{ \left( \rng  \toClass{F} \right)}.
\end{align*}
\end{Prop}

\begin{proof}
Set 
$
g := \quotient{f}{E}{F}
$.
Since $g \C \rng \toClass{E} \cartprod \rng \toClass{F}$ by \ref{RefDefQuotient}, it is $\rng g \C \rng \toClass{F}$, hence we are left with two points to prove:
\begin{enumerate}
\item
$ g $ is functional.
\item
$ g $ is left-total, that is, $ \rng \toClass{E} \C \dom g $.
\end{enumerate}
The two corresponding proofs are given.
\begin{enumerate}
\item
Consider sets $X$, $Y_1$, $Y_2$ such that 
$ \left\{ \pair {X} {Y_1}, \pair{X} {Y_2} \right\} \C g$. 
The goal is to show $Y_1 = Y_2$.
By \ref{RefDefQuotient}, consider $x_1, x_2, y_1, y_2$ such that 
$\pair{x_1}{y_1} \in X \cartprod Y_1 \cap f$ and
$\pair{x_2}{y_2} \in X \cartprod Y_2 \cap f$.
Since $X$ is an equivalence class of $E$, this implies $\pair{x_1}{x_2} \in E$ which in turn, by \ref{RefDefCompatible}, gives $ \pair{y_1}{y_2} \in F$.
Hence $y_1$ and $y_2$ must belong to the same equivalence class of $F$, which gives $Y_1 = Y_2$.
\item
Let $X \in \rng \toClass{E}$. 
$X$ being an equivalence class of the non empty equivalence relation $E$, there is $x \in X \C \dom E$.
Set
\begin{align}
\label{RefEq32}
y := & f \left( x \right) \in \dom F
\\
\notag
Y := & \toClass{F} \left( y \right) \in \rng F.
\end{align}
Since $\pair {x} {y} \in f$ by \eqref{RefEq32}, and $y \in Y$, we draw $\pair {X}{Y} \in g$ by \ref{RefDefQuotient}.
\end{enumerate}
\end{proof}


Result \ref{RefThmCompatibleFunction} supplies a canonical construction to pass from a function on sets to a function on  classes relative to equivalence relations respected by the original function.
We want to carry this mechanism over to the case in which the function is $i \left( s \right)$ and the equivalence relation is given on $U$, where $i$ is an interpretation of the language $S$, $s$ is a symbol of it, and $U$ is the universe of $i$.
Since $i \left( s \right)$ is defined on $ U ^ {\left| \ari \left( s \right)  \right|}$, we have to specify how to adapt some of the last definitions to tuples.
First of all, we formally specify the natural way to pass from a relation over sets to a relation over tuples:

\begin{Def}[Tupled relation]
\label{RefDefPlaces}
Let $O$ be a non empty relation, and $n$ a natural number. We set
\begin{align*}
\placesof{O}{n} := \left\{ 
\pair{p}{q} \in \left( \dom O \right)^n \cartprod \left( \left( \rng O \right)^n  \right) : q \C p \relcomp O 
\right\}.
\end{align*}
\end{Def}

Now, we want to combine the quotient defined in \ref{RefDefQuotient} with the construction of \ref{RefDefPlaces} to obtain a quotient operating on interpretations. 
A technical nuisance stands on our way, though: when quotienting by a tupled relation, we are left with a function acting on classes of equivalence of tuples, while an interpretation should act on tuples (of equivalence classes, in this case).
So we have to provide an object translating between these two types:

\begin{Def}
Let $P$ be a relation, $n$ be a natural number. Set
\begin{align*}
\tupleToClass{P}{n} := 
\left( \placesof{ \left( \toClass{P}^{-1}  \right) }{n} \right)
\relcomp
\toClass{\placesof{P}{n}}.
\end{align*}
\end{Def}\mbox{}\\
It can finally be plugged into the following definiens:

\begin{Def}[Quotient interpretation]
\label{RefDefInterpretationQuotient}
Given an interpretation $i$ and a relation $P$, set
\begin{align*}
\iQuotient{i}{P} := 
\dom \ari \ni s \mapsto 
\begin{cases}
\tupleToClass{P}{\abs{\ari \left( s \right)}}
\relcomp
\quotient{i \left( s \right)} {\placesof{P}{\abs{\ari \left( s  \right) }}} {P}
& \ari\left( s \right) \geq 0
\\
\tupleToClass{P}{\abs{\ari \left( s \right)}}
\relcomp
\quotient{i \left( s \right)} {\placesof{P}{\abs{\ari \left( s  \right) }}} {\left\{ \pair{0}{0}, \pair{1}{1} \right\}}
\relcomp
\peel{}
& \ari \left( s \right) < 0.
\end{cases}
\end{align*}
\end{Def}

Now we have to put forward some requests to make the quotient in \ref{RefDefInterpretationQuotient} actually an interpretation:

\begin{Def}
\label{RefDefCompatible2}
Given an interpretation $i$ of the language $S$, having $U$ as universe, we say that $i$ and the relation $P$ are \Comp{} if
\begin{align*}
\forall s \in \dom \ari && 
\begin{cases}
i \left( s \right) \text{ is \comp{\placesof{P}{\ari \left( s \right)}}{P}} & \ari \left( s \right) \geq 0
\\
i \left( s \right) \text{ is \comp{\placesof{P}{- \ari \left( s \right)}}{
\left\{ \pair{0}{0}, \pair{1}{1} \right\}
}} & \ari \left( s \right) < 0
\end{cases}
\end{align*}
\end{Def}

\begin{Prop}
\label{RefThmCompatibleInterpretation}
Given an interpretation $i$ of the language $S$ having universe $U$, and an equivalence relation $E$ on $U$ such that $i$ and $E$ are \Comp{}, $\iQuotient{i}{E}$ is an interpretation of $S$ having $\rng \left( \toClass{E} \right)$ as universe.
\end{Prop}

\begin{proof}
Set $I := \iQuotient{i}{E}$. 
Let $s \in \dom {\ari_S}$; set $n := \abs {\ari \left( s \right)} \in \N$, $f := i\left( s \right)$, $\ol E := \placesof {E} {n}$ and $ \eta := \tupleToClass{E}{n}$.
One easily realizes (or may refer to the \M{} article \verb|FOMODEL3.MIZ| to find the proofs) that $\ol E$ is an equivalence relation on $ \mapsFromTo{n}{U}$ and that 
\begin{align}
\label{RefEq33}
\eta : \mapsFromTo{n}{ \left( \rng \toClass E  \right)} \to \rng {\toClass {\ol E}}.
\end{align}
We show that $I$, $s$ and $\ran {\toClass{E}}$ satisfy \ref{RefDefInterpretation}.
By cases
\begin{description}
\item
[$ \ari \left( s \right) \geq 0$] 
Then $I \left( s \right) = \tupleToClass{E}{n} \relcomp \quotient{f} {\ol E} {E}$ and 
$ f : \mapsFromTo {n} {U} \to U $.
The goal is to prove that $I\left( s \right) : \mapsFromTo {n} {\left( \rng \toClass{E} \right) } \to \rng {\toClass{E}}$.
By \ref{RefDefCompatible2}, $ f $ is \comp{\ol E}{E}, so that 
$ \quotient {f} {\ol E} {E} : \ran {\toClass {\ol E}} \to \ran {\toClass E}$ by \ref{RefThmCompatibleFunction}.
This yields thesis by \eqref{RefEq33}.
\item
[$ \ari \left( s \right) < 0$]
Then $ I \left( s \right) = \eta \relcomp \quotient{f} {\ol E} {\id{2}}$ and $ f: \mapsFromTo{n} {U} \to 2$.
The goal is to prove that $I \left( s \right) : \mapsFromTo {n} {\left( \rng \toClass{E} \right)} \to 2$.
By \ref{RefDefCompatible2}, $ f $ is \comp{\ol E}{\id{2}}, so that
$ \quotient {f} {\ol E} {\id{2}} : \ran {\toClass {\ol E}} \to \ran {\toClass{\id{2}}}$ by \ref{RefThmCompatibleFunction}.
This yields thesis by \eqref{RefEq33}, being $\peel{} \quad : \rng {\toClass{\id 2}} = \left\{ \left\{ 0 \right\}, \left\{ 1 \right\} \right\} \to 2$.
\end{description}
\end{proof}

Result \ref{RefThmCompatibleInterpretation} ends this section. 
Wanting to apply it to the \free{} interpretation, in the next section we introduce a relation on terms, and investigate the conditions to make it an equivalence relation, as required by \ref{RefThmCompatibleInterpretation}.
In the subsequent section, we finally face the issue of compatibility.

\subsection[\TTermeq{} relation and \H{} interpretation]{The \Termeq{} relation on terms and the \H{} interpretation}
\label{RefSectEqRel}
\begin{Def}
\label{RefDefTermeq}
Given a ruleset $D$ and a set $X$, we define
\begin{align*}
\termeq{D}{X} := 
\im{
\left( \restrict{\conc}{\terms{S} \cartprod \terms{S}}  \right) ^ {-1}
}
{
\im{
\left( \curry{\conc}{\left\{ \pair{0}{\equiv} \right\}}  \right) ^ {-1}
}
{\provables{D} \left( X \right)}}.
\end{align*}
\end{Def}

\begin{Rem}
Since 
\begin{align}
\label{RefEqTermeqDef2}
\termeq{D}{X} = \left\{ 
\pair{t_1}{t_2} \in \terms{S} \cartprod \terms{S} : X \proves{D} \Eq{t_1}{t_2}
\right\},
\end{align}
$\termeq{D}{X}$ is a relation on $\terms{S}$.
\end{Rem}

\begin{Def}[The \H{} `interpretation']
$\henk{D}{X} := \iQuotient{\freeInt{X}}{\termeq{D}{X}}.$
\end{Def}

\begin{Prop}
\label{RefThmTermeqRefl}
If $D \emul{\emptyset} \left\{ \rEqRefl \right\}$, then $\dom {\termeq{D}{X}} = \terms{S}$  and $\termeq{D}{X}$ is reflexive.
\end{Prop}

\begin{proof}
Set $D_0 := \left\{ \rEqRefl \right\}$, $ P:=\termeq{D}{X}$. Let $t$ be a term. 
We have to show that $ \pair{t}{t} \in P$. Now 
\begin{align*}
\pair{\emptyset}{\eq{t}} \in \rEqRefl \left( \emptyset \right) \C \onestep{D_0} \left( \emptyset \right) 
\C \Derivables{D_0} \left( \emptyset \right) \C \Derivables{D} \left( \emptyset \right), 
\end{align*}
which shows that $ X \proves{D} \eq{t}$ by \ref{RefDefProvable2}, and hence thesis by virtue of \eqref{RefEqTermeqDef2}.
\end{proof}

\begin{Prop}
\label{RefThmTermeqSymm}
If $D \emul{\emptyset} \left\{ \rEqSymm \right\}$ and $X$ is 
\closed{$D$}, then 
$\termeq{D}{X} $ is symmetric.
\end{Prop}

\begin{proof}
Set $D_0 := \left\{ \rEqRefl \right\}$. 
Assume $ X \proves{D} \Eq{t_1}{t_2} $. We have to show $ X \proves{D} \Eq{t_2}{t_1}$.
\begin{align*}
\pair{\left\{ \Eq{t_1}{t_2} \right\}}{\Eq{t_2}{t_1}} \in \rEqSymm \left( \emptyset \right) = \onestep{D_0} \left( \emptyset \right) 
\C \Derivables{D_0} \left( \emptyset \right) \C \Derivables{D} \left( \emptyset \right), 
\end{align*}
and closure yields $ \Eq{t_1}{t_2} \in X$. Hence $X \proves{D} \Eq{t_2}{t_1}$ by \ref{RefDefProvable2}.
\end{proof}

\begin{Prop}
\label{RefThmTermeqTrans}
If $D \emul{\emptyset} \left\{ \rEqTrans \right\}$ and $X$ is 
\closed{$D$}, then
$\termeq{D}{X}$ is transitive.
\end{Prop}

\begin{proof}
Set $D_0 := \left\{ \rEqTrans \right\}$.
Assume $ X \proves{D} \Eq{t_1}{t_2} $ and $ X \proves{D} \Eq{t_2}{t_3} $. We have to show $ X \proves{D} \Eq{t_1}{t_3}$.

\begin{align*}
\pair{\left\{ \Eq{t_1}{t_2}, \Eq{t_2}{t_3} \right\}}{\Eq{t_1}{t_3}} \in \rEqTrans \left( \emptyset \right) = \onestep{D_0} \left( \emptyset \right) 
\C \Derivables{D_0} \left( \emptyset \right) \C \Derivables{D} \left( \emptyset \right), 
\end{align*}
and closure yields $ \left\{ \Eq{t_1}{t_2}, \Eq{t_2}{t_3}  \right\} \C X$. Hence $X \proves{D} \Eq{t_1}{t_3}$ by \ref{RefDefProvable2}.
\end{proof}

\begin{Lm}
\label{RefThmTermeqEqrel}
If $D \emul{\emp} \left\{ \rEqRefl \right\}$, $D \emul{\emp} \left\{ \rEqSymm \right\}$, $D \emul{\emp} \left\{ \rEqTrans \right\}$ and $X$ is 
\closed{$ D $}, then $ \termeq{D}{X}$ is an equivalence relation on $\terms{S}$.
\end{Lm}

\begin{proof}
Immediate from \ref{RefThmTermeqRefl}, \ref{RefThmTermeqSymm}, \ref{RefThmTermeqTrans}.
\end{proof}

\subsection{Compatibility}
\label{RefSectCompatibility}
\begin{Lm}
\label{RefThmTermeqHenkinCompatible1}
If $ D \emul{\emp} \left\{ \rEqRefl \right\}$, $X$ is \closed{$D$}, $ D \emul{\emp} \left\{ \rFunc \right\}$, $X$ is \closed{$\left\{ \rRel \right\} $}, $X$ is \closed{$ \left\{ \rEqSymm \right\}$}, then 
$\freeInt{X}$ and $\termeq{D}{X}$ are \Comp{}.
\end{Lm}

\begin{proof}
Take $s \in \dom \ari$. Set $P := \termeq{D}{X}$ and $ f:= \freeInt{X} \left( s \right)$. By cases.

\textbf{1) $ \ari \left( s \right) = 0 $}\\
By \ref{RefDefCompatible2}, we have to show that $f$ is 
$ \left( \placesof{P}{0},P  \right) $-compatible. 
Since $\placesof{P}{0} = \left\{ \pair{\emp}{\emp} \right\}$, it suffices to show that $ \pair {f \left( \emp \right)} {f \left( \emp \right)} \in P$. $ f \left( \emp \right) $ is in the universe $\terms{S}$ of $\freeInt{X}$ (see \ref{RefRemInterpretationOfLiteral}); hence, since $\dom P = \terms{S}$ and $P$ is reflexive by \ref{RefThmTermeqRefl} and the hypothesis $D \emul{\emp} \left\{ \rEqRefl \right\}$, we have thesis.

\textbf{2) $ \ari \left( s \right) > 0 $}\\
By \ref{RefDefCompatible2}, we have to show that $f$ is 
$ \left( \placesof{P}{n},P  \right) $-compatible, where we set $n := \ari \left( s \right) \in \Z^+$.
As from \ref{RefDefCompatible}, let $ \vec{t}, \vec{t}' \in \left( \terms{S}  \right)^n $, and assume $ \pair {\vec{t}}{\vec t '} \in \placesof{P}{n}$.
The goal is to prove $\pair{f\left( \vec t \right)}{f\left( \vec t' \right)} \in P$.
Set $ \Gamma := \left\{ \Eq{\vec t \left( j \right)} {\vec' \left( j \right)}: j \in n \right\} \in \Finsets{n}{\wffs{S}}$ and $ \varphi := \Eq{s \fconc \left( \vec t \right)}{ s \fconc \left( \vec t' \right)} = \Eq{f \left( \vec t \right)}{f \left( \vec t' \right) }$.
From
\begin{align*}
\pair{\Gamma}{\varphi} \in \rFunc \left( \emptyset \right) = \onestep{\left\{ \rFunc \right\}} \left( \emp \right) \C \Derivables{\left\{ \rFunc \right\}}\left( \emp \right) \C \Derivables{D} \left( \emp \right),
\end{align*}
which takes advantage of the hypothesis $ D \emul{\emp} \left\{ \rFunc \right\}$, and
\begin{align*}
\notag
\pair {\vec t}{\vec t'} \in \placesof{P}{n} & \Leftrightarrow & \forall j \in n  \left( tt \left( j \right), tt'\left( j \right)  \right) \in P 
& \Leftrightarrow  
\\ 
\forall j \in n  X \proves{D} \Eq{tt\left( j \right)}{tt' \left( j \right)} 
&  \Rightarrow & 
\Gamma \C X,
\end{align*}
where last deduction employed $D$-closure, we draw $ X \proves{D} \varphi$ thanks to \ref{RefDefProvable2}.

\textbf{3) $ \ari \left( s \right) < 0 $}\\
By \ref{RefDefCompatible2}, we have to show that $f$ is 
$ \left( \placesof{P}{n}, \id{2}  \right) $-compatible, where we set $n := - \ari \left( s \right) \in \Z^+$ and $\id{2} := \left\{ \pair{0}{0}, \pair{1}{1} \right\}$.
As from \ref{RefDefCompatible}, let $ \vec{t}, \vec{t}' \in \left( \terms{S}  \right)^n $, and assume $ \pair {\vec{t}}{\vec t '} \in \placesof{P}{n}$.
The goal is to prove $\pair{f\left( \vec t \right)}{f\left( \vec t' \right)} \in \id{2}$.
Set 
$ \Gamma := \left\{ \Eq{\vec t \left( j \right)} {\vec' \left( j \right)}: j \in n \right\} \in \Finsets{n}{\wffs{S}}$, 
and preliminarily deduce
\begin{align}
\notag
\pair {\vec t}{\vec t'} \in \placesof{P}{n} & \Leftrightarrow & \forall j \in n  \left( tt \left( j \right), tt'\left( j \right)  \right) \in P 
& \Leftrightarrow  
\\ 
\label{RefEq06}
\forall j \in n  X \proves{D} \Eq{tt\left( j \right)}{tt' \left( j \right)} 
&  \Rightarrow & 
\Gamma \C X
\end{align}
thanks to $D$-closure. 
Now proceed by subcases.

\begin{description}
\item[a) $ f\left( \vec t \right) = 1 $] The thesis reduces to showing 
$ f\left( \vec t'  \right) = 1 $, which, by \ref{RefDefFreeInt}, means 
$ \varphi' := s \fconc \left( \vec t' \right)  \in X$. 
Let 
$ \varphi := s \fconc \left( \vec t \right)  \in X$. 
The subcase assumption gives $ \varphi \in X$ by \ref{RefDefFreeInt}, hence 
$ \Gamma \cup \left\{ \varphi \right\} \C X$ by \eqref{RefEq06}. 
Furthermore,
\begin{align*}
\pair{\Gamma \cup \left\{ \varphi \right\}}{\varphi'} \in \rRel \left( \emptyset \right) = 
\onestep{\rRel} \left( \emptyset \right) \C \Derivables{\left\{ \rRel \right\}} \left( \emptyset \right).
\end{align*}
Thus $ X \proves{\left\{ \rRel \right\}} \varphi'$. By $ \left\{ \rRel \right\}$-closure, we are finished.

\item[b) $ f\left( \vec t \right) = 0 $] Thesis reduces to showing $ f\left( \vec t'  \right) = 0 $, which, by \ref{RefDefFreeInt}, means $ \varphi' := s \fconc \left( \vec t' \right)  \notin X$. 
By contradiction, assume 
\begin{align}
\label{RefEq07}
\varphi' \in X.
\end{align}
Set $\Gamma' := \left\{ \Eq{\vec t' \left( j \right)}{ \vec t \left( j \right)} : j \in n \right\} $.
Given $j \in n$, it is easily seen that $ X \proves{\left\{ \rEqSymm \right\}} \Eq{\vec t' \left( j \right)} {\vec t \left( j \right)}$, since 
$  \left\{  \Eq{\vec t \left( j \right)}{ \vec t' \left( j \right)}  \right\} \C X$ by \eqref{RefEq06}, and
\begin{align*}
\pair{\left\{ \Eq{ \vec t \left( j \right)}{ \vec t' \left( j \right)} \right\}}{
 \Eq{ \vec t' \left( j \right)}{ \vec t \left( j \right)}
} \in \rEqSymm \left( \emptyset \right) = \onestep{\left\{ \rEqSymm \right\}} \left( \emp \right) \C 
\Derivables{\left\{ \rEqSymm \right\}} \left( \emp \right).
\end{align*}
By $\left\{ \rEqSymm \right\}$-closure, we conclude that $\Gamma' \C X$, and hence that $\Gamma' \cup \left\{ \varphi' \right\} \C X$ by \eqref{RefEq07}.
Moreover,
\begin{align*}
\pair{\Gamma' \cup \left\{ \varphi' \right\}}{\varphi} \in \rRel \left( \emptyset \right) = 
\onestep{\rRel} \left( \emptyset \right) \C \Derivables{\left\{ \rRel \right\}} \left( \emptyset \right),
\end{align*}
yielding $ X \proves{\left\{ \rRel \right\}} \varphi $, and hence $ \varphi \in X$ by $\left\{ \rRel \right\}$-closure, contradicting $ f\left( \vec t \right) = 0$.
\end{description}
\end{proof}

\begin{Cor}
\label{RefThmTermeqHenkinCompatible2}
If $D \emul{\emp} \left\{ \rEqRefl, \rEqSymm, \rEqTrans, \rFunc, \rRel \right\}$ and $ \provables{D} \left( X \right) \C X$, then $\termeq{D}{X}$ and $\freeInt{X}$ are \Comp{}.
\end{Cor}

\begin{Cor}[of \ref{RefThmTermeqHenkinCompatible2} and \ref{RefThmCompatibleInterpretation}]
\label{RefThmHenkinIsInterpretation}
If $D \emul{\emp} \left\{ \rEqRefl, \rEqSymm, \rEqTrans, \rFunc, \rRel \right\}$ and $ \provables{D} \left( X \right) \C X$, then 
$ \henk{D}{X} $ is an interpretation having $\Classes{\terms{S}}{\termeq{D}{X}}$ as universe.
\end{Cor}

\subsection{The \H{} model}
\label{RefSectHenkinModel}
Here, the conditions making $\henk{D}{X}$ a model of $X$ are studied.
We first work out two preparatory results.

\begin{Lm}
\label{RefThmQuotientEval}
Let $i$ be an interpretation of $S$, and $P$ an equivalence relation over its universe $U$ such that $i$ and $P$ are \Comp{}. Then 
\begin{align*}
\restrict{\eval{\left(  \iQuotient{i}{P} \right)  }}{\terms{S}} = & 
\restrict{\toClass{P} \funccomp \eval{i}}{\terms{S}}
&& \text { and }
\\
\eval{ \left( \iQuotient{i}{P}  \right) } \left( \varphi_0 \right) = & \eval{i} \left( \varphi_0 \right) && \text{ if } \varphi_0 \left( 0 \right) \in \im{\ari ^ {-1}}{\Z^-} \sdiff \left\{ \equiv \right\}.
\end{align*}
\end{Lm}

\begin{proof}
Set $I := \iQuotient{i}{P}$. 
Let us show that 
\begin{align}
\label{RefEq34}
\restrict {\eval{I}} {{\terms{S}}_{, n}} = \restrict {\toClass{P} \funccomp \eval{i}} {{\terms{S}}_{, n}}
\end{align}
for every $n \in \N$ by complete induction on $n$.
For the case $n=0$, consider $t_0 \in {\terms{S}}_{,0}$; the goal equation is 
$ \eval{I} \left( t_0 \right) = \toClass{P} \left( \eval{i} \left( t_0 \right)  \right) $.
Set $v:= t_0 \left( 0 \right)$, $f:= i\left( v \right)$ and reason as follows:
\begin{align*}
\eval {I } \left( t_0 \right) \overset {\text{\tiny{\ref{RefDefEvalAtomic}}}}{=}
\left( I \left( v \right)  \right) \left( 0 \right)
\overset{\text{\tiny{\ref{RefDefInterpretationQuotient}}}}{=}
\left(  \quotient{f}{\placesof{P}{0}}{P} \funccomp \tupleToClass{P}{0} \right) \left( 0 \right)
= 
\left( \quotient{f} {\id{1}} {P} \funccomp \left\{ \pair {0} {\left\{ 0 \right\}} \right\} \right) \left( 0 \right)
\\
=
\quotient{f} {\id{1}} {P} \left(  \left\{ \pair {0} {\left\{ 0 \right\}} \right\} \left( 0 \right)\right)
=
\quotient{f} {\id{1}} {P} \left( \left\{ 0 \right\} \right) 
=
\toClass{P} \left( f \left( 0 \right) \right).
\end{align*}
Now assume \eqref{RefEq34} holds for every $n \leq m$. 
Let us prove that it holds for $n=m+1$.
Considered arbitrary $t \in {\terms{S}}_{, m+1}$, it suffices to show 
$\eval {I} \left( t \right) = \toClass{P} \left( \eval{i} \left( t \right) \right)$.
Set $s := t \left( 0 \right)$, $k := \ari \left( s \right)$, $ f := i \left( s \right)$.
We can assume $k > 0$; then
\begin{align*}
\eval{I} \left( t \right) 
\overset{\text{\tiny{\ref{RefDefEvalAtomic}}}}{=} I \left( s \right) \left( \eval{I} \funccomp \subterms{t} \right) 
\overset{!}{=}
I \left( s \right) \left( \toClass{P} \funccomp \eval i \funccomp \subterms{t} \right)
\overset{\text{\tiny{\ref{RefDefInterpretationQuotient}}}}{=}
\left( \quotient{f} {\placesof{P}{k}} {P} \funccomp \tupleToClass{P}{k} \right) 
\left( \toClass{P} \funccomp \eval i \funccomp \subterms{t} \right)
\\
= \left( \quotient{f} {\placesof{P}{k}} {P} \funccomp \tupleToClass{P}{k} \right) 
\left( \placesof{\left( \toClass{P}  \right) }{k} \left(  \eval i \funccomp \subterms{t} \right)  \right) 
= 
\left( \quotient{f} {\placesof{P}{k}} {P} \funccomp \tupleToClass{P}{k}  
\funccomp \placesof{\left( \toClass{P}  \right) }{k} \right)
\left(  \eval i \funccomp \subterms{t} \right)
\\
=
\left( \quotient{f} {\placesof{P}{k}} {P} \funccomp
\toClass{\placesof {P} {k}}
\right)
\left(  \eval i \funccomp \subterms{t} \right)
\overset{!!}{=}
\left( \toClass{P} \funccomp f \right) \left(  \eval i \funccomp \subterms{t} \right) = 
\toClass{P} \left( f \left( \eval i \funccomp \subterms t \right) \right)  
\overset{\text{\tiny{\ref{RefDefInterpretationQuotient}}}}{=} \toClass{P} \left( \eval i \left( t \right) \right).
\end{align*}
$!$ denotes the step employing inductive hypothesis. 
$!!$ denotes the spot where compatibility has been used.
This secures the first thesis.

Finally, set $r:= \varphi_0 \left( 0 \right)$, $l := -\ari \left( r \right) \in \Z^+$ and 
$g := i \left( r \right)$:
\begin{align*}
\eval {I} \left( \varphi_0 \right) 
= 
\left( I \left( r \right) \right) \left( \eval{I} \funccomp \subterms {\varphi_0} \right) 
\overset{!}{=}
\left( I \left( r \right) \right) \left( \toClass{P} \funccomp \eval {i} \funccomp 
\subterms {\varphi_0} \right)
=
\left( I \left( r \right) \right) \left( \placesof{ \left(  \toClass{P} \right) }{l} \left(  \eval {i} \funccomp 
\subterms {\varphi_0}  \right) \right)
\\
\overset{\text{\tiny{\ref{RefDefInterpretationQuotient}}}}{=}
\left( \peel{} \funccomp \quotient {g} {\placesof{P}{l}} {\id 2} \funccomp 
\tupleToClass {P}{l} \funccomp \placesof {\left( \toClass{P} \right)} {l}
\right) 
\left( \eval{i} \funccomp \subterms {\varphi_0}  \right)
=
\left( \left( \peel{2}  \right) \funccomp \quotient {g} {\placesof{P}{l}} {\id 2} \funccomp
\toClass{ \placesof{P}{l} }
\right) 
\left( \eval{i} \funccomp \subterms {\varphi_0}  \right)
\\
\overset{!!}{=}
\left( \left( \peel{2} \right) \funccomp 
\toClass{\id 2} \funccomp g
\right) 
\left( \eval{i} \funccomp \subterms {\varphi_0}  \right)
=
g \left( \eval{i} \funccomp \subterms {\varphi_0}  \right).
\end{align*}
Last equality is due to $ \peel{2} = \toClass{\id 2}^{-1}$.
In the passage marked by `!', the freshly proved first thesis were employed.
`!!' denotes the step employing compatibility.
\end{proof}

\begin{Lm}
\label{RefThmFreeIntTermEval}
$\restrict{\eval{\freeInt{X}}}{\terms{S}} = \id{\terms{S}}$.
\end{Lm}

\begin{proof}
Let us show
\begin{align}
\label{RefEq35}
\restrict{\eval{\freeInt{X}}}{ {\terms{S}}_{, n} } = \id{{\terms{S}}_{, n}  } && \forall n \in \N
\end{align}
by complete induction on $n$.
For the case $n=0$, consider $t_0 \in {\terms{S}}_{, 0}$, and set $v := t_0 \left( 0 \right)$.
\begin{align*}
\eval{\freeInt{X}} \left( t_0 \right) 
\overset {\text{\tiny{\ref{RefDefEvalAtomic}}}} {=}
\freeInt{X} \left( v \right) \left( 0 \right)
\overset {\text{\tiny{\ref{RefDefFreeInt}}}} {=}
\left( \curry{\conc}{\left\{ \pair{0}{v}  \right\} }  \right)  \funccomp \left( \restrict{\fconc}{ \left( { \terms{S} }^{0} \right) } \right)
\left( 0 \right)
\\
=
\left( \curry{\conc}{\left\{ \pair{0}{v}  \right\} }  \right)  
\left( \left( \restrict{\fconc}{ \left\{ 0 \right\} } \right)
\left( 0 \right)
 \right)
= \left\{ \pair {0} {v} \right\} \conc \emp = t_0.
\end{align*}
Now, assume \eqref{RefEq35} is verified for every $n \leq m+1$, and consider $t \in { \terms{S}}_{, m+1}$. 
Set $s:= t\left( 0 \right)$, $k:= \ari \left( t \right) \in \N$.
We can assume $ k > 0 $, and have to show that $\eval{\freeInt{X}} \left( t \right) = t$:
\begin{align*}
\eval{\freeInt{X}} \left( t \right)
\overset {\text{\tiny{\ref{RefDefEvalAtomic}}}} {=}
\left( \freeInt{X} \left( s \right)  \right) \left( \eval{\freeInt{X}} \funccomp \subterms{t} \right)
\overset{!}{=}
\left( \freeInt{X} \left( s \right)  \right) \left(  \subterms{t} \right)
\\
\overset {\text{\tiny{\ref{RefDefFreeInt}}}} {=}
\left( \curry{\conc}{\left\{ \pair{0}{s}  \right\} }  \right)
\left( 
\left( \restrict{\fconc}{ \left( { \terms{S} }^{k} \right) } \right) 
\left(  \subterms{t} \right)
 \right)
= \left\{ \pair {0} {s} \right\} \conc \left( \fconc \left( \subterms t \right) \right) 
= t.
\end{align*}
`!' denotes the induction step.
\end{proof}

Now we see that, when restricting to atomic formulas, one actually needs to impose 
very little additional requests 
for $\henk{D}{X}$ to be a model, besides those from \ref{RefThmHenkinIsInterpretation} making it an interpretation:

\begin{Thm}
\label{RefThmHenkin1}
If $D \emul{\emp} \left\{ \rAssumption, \rEqRefl, \rEqSymm, \rEqTrans, \rFunc, \rRel \right\}$ and $ \provables{D} \left( X \right) \C X $, then 
\begin{align*}
\restrict{\eval{\henk{D}{X}}}{\wffs{S,0}} = \indicator_X^{\wffs{S,0}}.
\end{align*}
\end{Thm}

\begin{proof}
We set $i := \freeInt{X}$, $P := \termeq{D}{X}$, $I:= \henk{D}{X} = \iQuotient{i}{P}$.
Let $\varphi_0 \in \wffs{S,0}$, and set $r := \varphi_0 \left( 0 \right)$, $n:= - \ari \left( r \right) \in \Z^+$. 
By cases.
\\
\textbf{Case $ r \neq \equiv $:}\\
\begin{align*}
\eval{I} \left( \varphi_0  \right) \overset{
\text{\tiny{\ref{RefThmTermeqHenkinCompatible2}, \ref{RefThmQuotientEval}}}
}{=}   \eval{i} \left( \varphi_0 \right) 
\overset{\text{\tiny{\ref{RefDefEvalAtomic}}}}{=} 
\left(  i \left( r \right) \right) \left( \eval{i} \funccomp \substrings{\varphi_0} \right) 
\overset{\text{\tiny{\ref{RefThmFreeIntTermEval}}}}{=}  
\left( i\left( r \right) \right) \left( \substrings{\varphi_0} \right) 
\overset{\text{\tiny{\ref{RefDefFreeInt}}}}{=} 
\\ 
\indicator_X^{\wffs{S,0}} \funccomp \left( \curry{\conc}{\left\{ \pair{0}{r} \right\}}  \right) \funccomp \left( \restrict{\fconc}{\terms{S}^n}  \right) \left( \substrings{\varphi_0} \right) = 
\indicator_X^{\wffs{S,0}} \funccomp \left( \left( \curry{\conc}{\left\{ \pair{0}{r} \right\}}  \right) \funccomp \fconc  \right) \left( \substrings{\varphi_0} \right) = \\
\indicator_X^{\wffs{S,0}} \left(  
\left( \curry{\conc}{\left\{ \pair{0}{r} \right\}}  \right) 
\left( \fconc \left( \substrings{\varphi_0} \right)  \right) 
\right) 
=
\indicator_X^{\wffs{S,0}} \left(  
\conc  
\pair{\left\{ \pair{0}{r} \right\}}  
{ \fconc \left( \substrings{\varphi_0} \right)}   
\right)
= 
\indicator_X^{\wffs{S,0}} \left( \varphi_0 \right).
\end{align*}
\hfill\\
\textbf{Case $ r = \equiv $:}\\
Set $t_1 := \subterms{\varphi_0} \left( 0 \right)$, $ t_2 := \subterms{\varphi_0} \left( 1 \right)$.
\begin{align*}
\eval{I} \left( \varphi_0 \right) = 1 
\overset{\text{\tiny{\ref{RefDefEvalAtomic}}}}
{\Leftrightarrow} 
\eval{I} \left( t_1 \right) = \eval{I} \left( t_2 \right) 
\Leftrightarrow 
\toClass{P} \left( \eval{i} \left( t_1 \right) \right) = 
\toClass{P} \left( \eval{i} \left( t_2 \right) \right)
\overset{\text{\tiny{\ref{RefThmFreeIntTermEval}}}}
{\Leftrightarrow}
\\
\toClass{P} \left( t_1  \right) = 
\toClass{P} \left(  t_2 \right)
\overset{\text{\tiny{\ref{RefDefTermeq}}}}
{\Leftrightarrow} X \proves{D} \Eq{t_1}{t_2}
\Leftrightarrow \Eq{t_1}{t_2} = \varphi_0 \in X.
\end{align*}
Last equivalence is due to $D$-closure ($ \Rightarrow $) and to $D \emul{\emp} \left\{ \rAssumption \right\}$ ($ \Leftarrow $).
\end{proof}

The ultimate goal of this section is the extension of \ref{RefThmHenkin1} to the whole $ \wffs{S}$. 
To this end, we will need to employ a couple of auxiliary results significant in their own right, as relating the syntactical constructions of simple substitution and term substitution (defined in \ref{RefDefSimpleSubst} and \ref{RefDefTermSubst}) to the semantical one of reassignment (defined in \ref{RefNotationReassign}):

\begin{Lm}
\label{RefThmSubstLemma1}
\begin{align*}
\eval{\reassign{v_1}{u}{i}} \left( \psi \right) = 
\eval{\reassign{v_2}{u}{i}} \left( \symbsubst{v_1}{v_2}{\psi} \right),
\end{align*}
where $u$ is an element of the universe of the interpretation $i$ and $v_2 \notin \rng \psi$.
\end{Lm}

\begin{proof}
Denote with $S$ the language we are working in, with $A$ the symbol set of $S$, and with $U$ the universe of $i$.
Set 
$i_1 := \reassign{v_1}{u}{i}$, $ i_2 := \reassign{v_2}{u}{i}$, 
$f_1 := \eval{i_1}$, $f_2 := \eval{i_2}$, 
$g:= \id A \paste \left\{ \pair{v_1}{v_2} \right\}$, 
$B:=A \sdiff \left\{ v_2 \right\}$.
We start with showing that
\begin{align}
\label{RefEq36}
f_1 \left( t' \right) = f_2 \left( g \funccomp t' \right) && \forall
t' \in {\terms{S}}_{,n} \cap B^*
\end{align}
by complete induction on $n$.
The case $n=0$ is trivial, and anyway is treated in \MML{} article \verb|FOMODEL3|, at the label \verb|Lm44|.
Now suppose \eqref{RefEq36} holds for every $n \leq m$, and consider $t \in T$ such that $\depth t \leq m+1$. 
Set $s := t\left( 0 \right)$.
We can assume
\begin{align}
\label{RefEq37}
\ari \left( s \right) > 0
\end{align}
and complete the proof of \eqref{RefEq36} as from the following iterative equation
\begin{align*}
f_2 \left( g \funccomp t \right) 
\overset {\text{\tiny{\ref{RefDefEvalAtomic}}}} {=}
\left( i_2 \left( s \right)\right) \left( f_2 \funccomp \subterms{g \funccomp t} \right)
=
\left( i_1 \left( s \right) \right) \left( f_2 \funccomp \subterms{g \funccomp t} \right)
=
\left( i_1 \left( s \right) \right) \left( f_1 \funccomp \subterms{t} \right),
\end{align*}
whose last step rests on inductive hypothesis applied to \eqref{RefEq36}.
The immediately preceding step is due to the fact that \eqref{RefEq37} implies $s \notin \left\{ v_1, v_2 \right\}$.
Similarly, one can show that
\begin{align}
\label{RefEq38}
f_1 \left( \psi_0 \right) = f_2 \left( g \funccomp \psi_0 \right) && \forall \psi_0 \in 
{\wffs{S}}_{,0} \cap B^*.
\end{align}
To avoid repetitions, we refer the interested reader to \verb|FOMODEL3:Lm45| for the proof of 
\eqref{RefEq38}.
At last, we show
\begin{align}
\label{RefEq39}
f_1 \left( \psi' \right) = f_2 \left( g \funccomp \psi' \right) 
&&
\forall \psi' \in B^* \cap { \wffs{S}}_{,n} 
\end{align}
by complete induction on $n$.
The case $n=0$ is given by \eqref{RefEq38}.
Let us then assume \eqref{RefEq39} for every $n \leq m$, and consider 
$\psi \in B^* \cap { \wffs{S}}_{,m+1}$.
We can assume as well $ \depth{\psi} > 0$ and set 
$ s := \psi \left( 0 \right) \in \im {\ari ^ {-1}} {\left\{ 0 \right\}} \sdiff \left\{ v_2 \right\} \cup \left\{ \nor \right\}$.
By cases.
\\
\\
\textbf{Case 1):} $ s = \nor$\\
Then set 
$\psi_1 := \substrings {\psi} \left( 0 \right)$,
$\psi_2 := \substrings {\psi} \left( 1 \right)$ and $N := \indicator_{\left\{ \pair{0}{0} \right\}}^{2 \cartprod 2}$.
We employ \eqref{RefEq39} via induction on the unmarked step of the following chain: 
\begin{align*}
f_2 \left( g \funccomp \psi \right) 
\overset{\text{\tiny{\ref{RefDefEvalCompound}}}}{=} 
N \left( \pair {f_2 \left( g \funccomp \psi_1  \right) } {f_2 \left( g \funccomp \psi_2  \right) } \right)
\\
=
N \left( \pair {f_1 \left( \psi_1 \right)} {f_1 \left( \psi_2 \right)} \right)
\overset{\text{\tiny{\ref{RefDefEvalCompound}}}}{=}
f_1 \left( \psi \right).
\end{align*}
\\
\textbf{Case 2):} $s \in \im {\ari ^ {-1}} {\left\{ 0 \right\}} \sdiff \left\{ v_2 \right\}$\\
Then consider $\varphi \in B^* \cap {\wffs{S}}_{, m}$ such that $\psi = s \varphi$.
By subcases.
\begin{description}
\item
[Subcase $s=v_1$:]\hfill\\
Then 
$ g \funccomp \psi = v_2 \conc \left( g \funccomp \varphi \right)$.
Assume $f_2 \left( g \funccomp \psi \right) = 1$.
Then, by \ref{RefDefEvalCompound}, consider $u' \in U$ such that
\begin{align*}
1 = \eval{\reassign {v_2}{u'} { \reassign {v_2}{u} {i} }} 
\left(  g \funccomp \varphi \right) 
=
\eval{\reassign {v_2}{u'} i}  \left(  g \funccomp \varphi \right)
\overset{\text{\tiny{\eqref{RefEq39}}}}{=}
\eval{\reassign {v_1}{u'} i}  \left(  \varphi \right)
=
\eval{\reassign {v_1}{u'} { \reassign {v_1}{u} {i} }} \left( \varphi \right).
\end{align*}
Hence, again by \ref{RefDefEvalCompound}, 
$ 
\eval{\reassign {v_1}{u} {i}} \left( v_1 \varphi \right)=1
$.
Analogously one shows
$ 
\eval{\reassign {v_1}{u} {i}} \left( v_1 \varphi \right)=1 \ra{}
\eval{\reassign {v_2}{u} {i}} \left( g \funccomp \psi \right)=1
$.
\item
[Subcase $s \neq v_1$:]\hfill\\
Assume $ f_2 \left( g \funccomp \psi \right) = 1$. 
Then, by \ref{RefDefEvalCompound}, consider $u' \in U$ such that
\begin{align*}
1 = \eval{\reassign {s}{u'} {i_2} } \left( g \funccomp \varphi \right) 
=
\eval{\reassign {s}{u'} { \reassign {v_2}{u} {i} }} \left( g \funccomp \varphi \right)
=
\eval{\reassign {v_2}{u} { \reassign {s}{u'}{i} } } \left( g \funccomp \varphi \right)
\\
=
\eval{\reassign {v_1}{u} { \reassign {s}{u'}{i} } } \left( \varphi \right) 
= 
\eval{\reassign {s}{u'} { \reassign {v_1}{u} {i} }} \left( \varphi \right).
\end{align*}
Hence $ \eval{\reassign {v_1}{u}{i}} = 1 $ by \ref{RefDefEvalCompound}.
In a similar way, one shows $1 = f_1 \left( \psi \right) \ra{} f_2 \left( g \funccomp \psi \right)=1$.
\end{description}
\end{proof}

\begin{Lm}[Substitution lemma]
\label{RefThmSubstLemma2}
Given $v$, $t$, $\varphi$:
\begin{enumerate}
\item
$ \depth {\subst{v}{t}{\varphi}} = \depth \varphi$;
\item
$ \eval{i} \left( \subst{v}{t}{\varphi} \right) = 
\eval{\reassign{v}{\eval{i} \left( t \right)}{i}} \left( \varphi \right)$, for any interpretation $i$.
\end{enumerate}
\end{Lm}

\begin{proof}
See appendix \ref{RefSectSubstLemma}.
\end{proof}

\begin{Def}[Witness]
\label{RefDefWitnessed}
Given a language $S$, consider the following relation on $\wffs{S}$:
\begin{align*}
\witnessRel{S} := 
\left\{ 
\pair{\left\{ \pair{0}{v_1}  \right\} \conc \varphi}{\symbsubst{v_1}{v_2}{\varphi}} : 
v_1, v_2 \in \im{\ari^{-1}} { \left\{ 0 \right\}}, \varphi \in \wffs{S} \st \quad v_2 \notin \rng \varphi 
\right\}.
\end{align*}
Often the context will allow to drop the subscript and write just $\witnessRel{}$.

If $\varphi \in \im{\witnessRel{S}} {\left\{ \psi \right\}}$, we say that $\varphi$ is a \emph{witness}\index{witness} for $\psi$.

A set $X$ will be said to be \emph{$S$-\witnessed{}} (simply \emph{\witnessed{}}\index{\witnessed{}} when the context is safe) if 
\begin{align*}
X \cap \dom {\witnessRel{S}} \C \im {\witnessRel{S} ^ {-1}}{X}.
\end{align*}
\end{Def}

\begin{Def}
$X$ is a minimal cover of the language $S$ (or an $S$-\mincov{}, or even just a \mincov{}) if
\begin{align*}
\forall \varphi \in \wffs{S} \left( \varphi \in X \text{ if and only if } \xnot{\varphi} \notin X  \right).
\end{align*}
\end{Def}

\begin{Thm}[\H{}'s theorem]
\label{RefThmHenkin2}
Suppose 
\begin{itemize}
\item
$D \emul{\emp} \left\{ \rAssumption, \rEqRefl, \rEqSymm, \rEqTrans, \rFunc, \rRel, 
\rNor, \rEx \right\}$,
\item
$X$ is a \mincov{}, 
\item
$ \provables{D} \left( X \right) \C X $, and 
\item
$X$ is \witnessed{}. 
\end{itemize}
Then
\begin{align*}
\restrict{\eval{\henk{D}{X}}}{\wffs{}} = \indicator_X^{\wffs{}}.
\end{align*}
\end{Thm}

\begin{proof}
Set $i := \freeInt{X}$, $P := \termeq{D}{X}$, $I:= \henk{D}{X} = \iQuotient{i}{P}$.
We will prove 
\begin{align}
\label{RefEq08}
\restrict{\eval{I}}{\wffs{S,m}} = \indicator_X^{\wffs{S,m}}
\end{align}
by complete induction on $m$.
For $m=0$, thesis is given by \ref{RefThmHenkin1}.
Assume the inductive hypothesis: \eqref{RefEq08} holds for all $m \leq n$. 
Let $\psi \in \wffs{S,n+1}$.
We have to show that 
\begin{align*}
\eval{I} \left( \psi \right) = 1 \Leftrightarrow \psi \in X.
\end{align*}
We can suppose $\psi \notin \wffs{S,0}$, and proceed by cases.
\begin{description}
\item[Case $ \psi \left( 0 \right) \neq \nor $:] Then consider $v_1$, $\varphi$ such that $\psi = v_1 \varphi$. 
\begin{align}
\notag
\eval{I} \left( \psi \right) = 1 \overset{\text{\tiny{\ref{RefDefEvalCompound}}}}
{\Leftrightarrow} \exists t \in \terms{} \st \quad 1 = 
\eval{\reassign{v_1}{\toClass{P} \left( t \right)}{I}} \left( \varphi \right) 
\overset{\text{\tiny{\ref{RefThmFreeIntTermEval}}}}{=}
\eval{\reassign{v_1}{\toClass{P} \left( \eval{i} \left( t  \right) \right)}{I}} \left( \varphi \right) 
\\
\label{RefEq09}
\overset{\text{\tiny{\ref{RefThmQuotientEval}, \ref{RefThmTermeqHenkinCompatible2}}}}{=}
\eval{\reassign{v_1}
{ \eval{I} \left( t  \right) }
{I}} \left( \varphi \right)
\overset{\text{\tiny{\ref{RefThmSubstLemma2}}}}{=} 
\eval{I} \left(  {\subst{v_1}{t}{\varphi}}  \right).
\end{align}

\begin{description}
\item[$\Leftarrow$]\hfill\\
Assume $\psi \in X$. 
Then consider $v_2 \in \im{\ari^{-1}}{\left\{ 0 \right\}} \sdiff \ran \varphi$ such that 
$ \symbsubst{v_1}{v_2}{\varphi} \in X$ by \ref{RefDefWitnessed}. 
Since $\depth{\symbsubst{v_1}{v_2}{\varphi}} = \depth{\varphi} < \depth {\psi}$, we can trigger induction:
\begin{align*}
1 = \eval{I} \left( \symbsubst{v_1}{v_2}{\varphi}  \right) = 
\eval{\reassign{v_2}{ \left( I \left( v_2 \right) \right) \left( \emp \right) }{I}} 
\left( \symbsubst{v_1}{v_2}{\varphi}  \right) 
\overset{\text{\tiny{\ref{RefThmSubstLemma1}}}}{=} 
\eval{\reassign{v_1}{ \left( I \left( v_2 \right) \right) \left( \emp \right) }{I}} 
\left( \varphi \right).
\end{align*}
Thesis follows from \ref{RefDefEvalCompound}.
\item[$\Rightarrow$]\mbox{}\\ 
Assume $\eval{I} \left( \psi \right) = 1$ and, by \eqref{RefEq09}, consider $\ol t \st \quad \subst{v_1}{\ol t}{\varphi} \in X$.
\begin{align*}
\pair{\left\{ \subst{v_1}{\ol t}{\varphi} \right\}}{v_1 \varphi} \in \rEx \left(  \emp  \right) 
\C \Derivables{\left\{ \rEx \right\}} \left( \emp \right) \C \Derivables{D} \left( \emp \right).
\end{align*}
By $D$-closure, we draw $ \psi \in X$.
\end{description}

\item[Case $\psi \left( 0 \right) = \nor$:]
Set $\varphi_1:= \substrings{\psi}\left( 0 \right)$, $ \varphi_2:= \substrings{\psi}\left( 1 \right) $.
\begin{align*}
\eval{I} \left( \psi \right) = 1 
\overset{\text{\tiny{\ref{RefDefEvalCompound}}}}
{\Leftrightarrow} 
\eval{I} \left( \varphi_1 \right) = 0 = \eval{I} \left( \varphi_2 \right) \Leftrightarrow
\\
\left\{ \varphi_1, \varphi_2 \right\} \cap X = \emp \Leftrightarrow 
\left\{ \xnot {\varphi_1}, \xnot {\varphi_2} \right\} \C X,
\end{align*}
where last equivalence is due to \mincov{} hypothesis, and previous one to inductive hypothesis.
Hence we have reduced our task to showing that
\begin{align*}
\psi \in X 
\Leftrightarrow 
\left\{ \xnot {\varphi_1}, \xnot {\varphi_2} \right\} \C X.
\end{align*}
\begin{description}
\item[$\Leftarrow$]\mbox{}\\
Assume $ \left\{ \xnot {\varphi_1}, \xnot {\varphi_2} \right\} \C X $. 
Since 
\begin{align*}
\pair{\left\{ \xnot {\varphi_1}, \xnot {\varphi_2} \right\} }{ \psi } \in \rNor \left( \emp \right)
\C \Derivables{\left\{ \rNor \right\}} \left( \emp \right) \C \Derivables{D} \left( \emp \right),
\end{align*}
thesis follows immediately from \ref{RefDefProvable2} and $D$-closure hypothesis.
\item[$\Rightarrow$]\mbox{}\\
Assume $\psi \in X$. Set $\psi' := \nor{\varphi_2}{\varphi_1}$. 
Now
\begin{align*}
\left\{ \psi \right\} \proves{\left\{ \rNor \right\}} \psi' 
\ra{} \left\{ \psi \right\} \proves{ D } \psi',
\end{align*}
where the implication is given by $ D \emul{\emp} \left\{ \rNor \right\}$.
By $D$-closure, we conclude $ \left\{ \psi, \psi' \right\} \C X$. 
This, together with 
\begin{align*}
\pair{\left\{ \psi, \psi' \right\}}{\xnot{\varphi_1}}, 
\pair{\left\{ \psi, \psi' \right\}}{\xnot{\varphi_2}} \in
\rNor \left( \emp \right)
\\
\C \Derivables{\left\{ \rNor \right\}} \left( \emp \right) \C \Derivables{D} \left( \emp \right),
\end{align*}
ends the proof by virtue of $D$-closure.
\end{description}
\end{description}
\end{proof}

\begin{Rem}
It is readily checked that in proof of \ref{RefThmHenkin2}, the following slightly weaker flavor of $\rNor$ would suffice:
\begin{align*}
\seqs{S} \supseteq \Sigma \mapsto \left\{ 
\pair{\Gamma}{\varphi} :
\exists \varphi_1, \varphi_2, \varphi_3, \varphi_4 \in \wffs{S} \st \Gamma = \left\{  
\nor {\varphi_1} {\varphi_2} , \nor {\varphi_3} {\varphi_4}
\right\}
\vphantom{
\text{ and } \varphi= \nor {\varphi_2} {\varphi_3}
\text{ and } \card{\left\{ \varphi_1, \varphi_2, \varphi_3, \varphi_4 \right\}} \leq 2
}
\right.
\\
\vphantom{
\pair{\Gamma}{\varphi} :
\exists \varphi_1, \varphi_2, \varphi_3, \varphi_4 \in \wffs{S} \st \Gamma = \left\{  
\nor {\varphi_1} {\varphi_2} , \nor {\varphi_3} {\varphi_4}
\right\}
}
\left.
\text{ and } \varphi= \nor {\varphi_2} {\varphi_3}
\text{ and } \card{\left\{ \varphi_1, \varphi_2, \varphi_3, \varphi_4 \right\}} \leq 2
\right\}.
\end{align*}
Since all the forthcoming results requiring $\rNor$ do so precisely to invoke \ref{RefThmHenkin2}, the same goes for them.
We adopt $\rNor$ mainly because it is more straightly put into a diagram than its variant above.
\end{Rem}


\section{Enlarging sets of formulas 
}
\label{RefSectEnlarge}
In this section we study how to enlarge a given set $X$ of formulas to make it 
\begin{itemize}
\item
closed with respect to a given ruleset $D$ and
\item
\witnessed{},
\end{itemize}
so that the enlargement can be applied \ref{RefThmHenkin2}: in particular, this automatically supplies a model for $X$, which is our ultimate goal.
We shall investigate the conditions $X$ and $D$ must obey to perform this operation.
We will restrict to countable languages to more easily develop constructive methods to build the two distinct enlargements corresponding to the points of the above checklist.
The rub is 
how to combine sequentially the two enlargements avoiding the second cancelling the effect of the first.
The property of the witness subjoining construction expressed by \ref{RefThmAddWFurnished} and deployed in \ref{RefThmEnlarge} will be the key.

\subsection{Preliminaries}

\begin{Def}
\label{RefDefNot}
Consider the following element of $\mapsFromTo{\wffs{S}}{\left( \wffs{S}  \right) }$:
\begin{align*}
\notf_S : \varphi \mapsto \nor{\varphi}{\varphi}.
\end{align*}
\end{Def}

\begin{Not}
Again, we can drop the subscript in $\notf_S$ when it is safe to do so.
In addition, we will usually write $\xxnot \varphi$ instead of $\notf \left( \varphi \right)$:
\begin{align*}
\xxnot \varphi = \xnot \varphi,
\end{align*}
and $ \iter{\notf}{n} \varphi$ instead of $\iter{\notf}{n} \left( \varphi \right)$.
\end{Not}

\begin{Def}[Forms of consistency]
\label{RefDefConsistencies}
$X$ is said to be $S$-consistent (or syntactically consistent when the context is clear) if
\begin{align*}
X \cap \im{\notf_S^{-1}}{X} = \emp.
\end{align*}
$X$ is $S$-inconsistent (syntactically inconsistent) if it is not $S$-consistent.
It is said to be an $S$-cover (or just a cover) if
\begin{align*}
X \cup \im{\notf_S^{-1}}{X} \supseteq \wffs{S}.
\end{align*}
It is termed $D$-consistent (or just consistent when no ambiguity can arise) if $\provables{D} \left( X \right)$ is $S$-consistent, otherwise we say it is $D$-inconsistent (inconsistent): we write 
$\Con{D}{X}$ and $ \Inc{D} { X} $, respectively.
\end{Def}

\begin{Rem}
$X$ is a \mincov{} \If{} $X$ is a syntactically consistent cover.
\end{Rem}

\begin{Def}
\label{RefDefAssLikeWeak}
A ruleset $D$ is said to be \emph{\asslikeW{}}\index{\asslikeW{}} if any $D$-consistent cover
is a \closed{$D$} \mincov{}.
\end{Def}

\begin{Def}
\label{RefDefAssLikeStrong}
A ruleset $D$ is said to be \emph{\asslikeS{}}\index{\asslikeS{}} if for any $D$-consistent cover $X$ it holds $\provables{D}\left( X \right) = X \cap \wffs{S}$.
\end{Def}

\begin{Rem}
Any \asslikeS{} ruleset is \asslikeW{}.
\end{Rem}

\begin{Prop}
\label{RefThmAssLikeExample}
$ \left\{ \rAssumption \right\}$ is \asslikeS{}.
\end{Prop}

\begin{proof}
Let $X$ be a $\left\{ \rAssumption \right\}$-consistent cover.
We can assume $X \C \wffs{S}$. 
Of course, being $\pair {\left\{ \psi \right\}} {\psi} \in \rAssumption \left( \emp \right)$ for any $\psi \in \wffs{S}$, one has, in particular, that $X \C \provables{ \left( \left\{
 \rAssumption \right\} \right) } \left( X \right)$. 
Hence it remains to show that 
$ \provables {\left( \left\{ \rAssumption \right\}  \right) } \left( X \right) \C X$.
Assume $\psi \in \provables {\left( \left\{ \rAssumption \right\}  \right) } \left( X \right)$.
Then consider, by \ref{RefDefProvable2}, $n \in \N$ and a finite $\Gamma \C X$ such that 
$
\pair{\Gamma}{\psi} \in \derivables{n+1}{\left\{ \rAssumption \right\}} \left( \emp \right) = 
\rAssumption \left( \derivables{n} {\left\{ \rAssumption \right\}} \left( \emp \right) \right)
$.
%
%
This gives $ \Gamma = \left\{ \psi \right\} $ by definition of $\rAssumption$. 
Hence thesis.
\end{proof}

\begin{Prop}
\label{RefThmAssLikeTransitive}
If $D_1$ is \asslikeS{} and $D_2 \emul{\emp} D_1$, then $D_2$ is \asslikeS{}.
\end{Prop}

\begin{proof}
Given a $D_2$-consistent cover $X \C \wffs{S}$, we must show that $\provables{D_2} \left( X \right) = X$.
First $X = \provables{D_1} \left( X \right) \C \provables{D_2} \left( X \right)$.
To show the reverse inclusion, $ \provables {D_2} \left( X \right) \C X $, consider $\varphi$ and suppose $\varphi \in \provables{D_2} \left( X \right)$:
\begin{align*}
\varphi \in \provables{D_2} \left( X \right) \ra{} 
\xnot {\varphi} \notin \provables{D_2} \left( X \right) 
\ra{} \xnot \varphi \notin \provables{D_1} \left( X \right) \ra{}
\xnot \varphi \notin X \ra{} \varphi \in X.
\end{align*}
First implication is due to consistency, and last one to $X$ being a cover.
\end{proof}

\begin{Def}
\label{RefDefCutLike}
A ruleset $D$ is \cutLike{} if
$ \Inc{D}{X \cup \left\{ \varphi \right\}}$ 
implies 
$X \proves{D} \xnot {\varphi}$ 
for every $X$, $\varphi$.
\end{Def}

\subsection{Witness-subjoining construction for countable languages}

\begin{Def}
\label{RefDefAddW}
Given $X$, $D$ and two mappings 
\begin{align*}
l: \N \ni n &\mapsto v_n \in \im{\ari^{-1}} { \left\{ 0 \right\} }
\\
f: \N \ni n &\mapsto \varphi_n \in F_S,
\end{align*}
define recursively
\begin{align*}
X_0 & := X
\\
a_n & := \im {\ari^{-1}} { \left\{ 0 \right\}} \sdiff 
\symbof {X_n \cup \left\{ \varphi_n \right\} }
\\
X_{n+1} &:=
\begin{cases}
X_n \cup
\left\{
\symbsubst {v_n} 
{
l \left(
\min \im{l^{-1}} { a_n }
\right)
} {\varphi_n}
\right\}
& 
\begin{aligned}
& \text{ if } 
\Con{D} { X_n \cup \left\{  v_n \varphi_n \right\} },
\\
& X_n \cap \im{\witnessRel{}}{\left\{ v_n \varphi_n \right\}} = \emp \text{ and } a_n \neq \emptyset
\end{aligned}
\\
X_n & \text{ otherwise },
\end{cases}
\end{align*}
and finally
\begin{align*}
\addW{D}{l,f} \left( X \right) := \bigcup_{n \in \N} X_n.
\end{align*}
\end{Def}

\begin{Lm}
\label{RefThmAddWCon}
Assume that
\begin{enumerate}
\item
\label{RefEq21}
$D$ is \cutLike{};
\item
\label{RefEq10}
$D \emul{\emp} \left\{ \rEqRefl \right\}$;
\item
\label{RefEq12}
$\rWitnessA \in D$.
\end{enumerate}
If
$\Con{D}{X}$, then  
$ \Con{D} { \addW{D}{l,f} \left( X \right) }$.
\end{Lm}

\begin{proof}
Suppose $\Inc{D}{\addW{D}{l,f} \left( X \right) }$. Then, referring to the objects introduced in \ref{RefDefAddW}, we can take the minimum $m$ of the non-empty subset of $\N$:
\begin{align*}
\left\{ n \in \N \st \Inc{D}{X_n} \right\}.
\end{align*}
If $m=0$, then we are done. 
Otherwise, consider $k \in \N \st m=k+1$.
Having set
\begin{align*}
v_k := & l \left( k \right)
\\
\varphi_k := & f \left( k \right)
\\
v'_k := &
l \left(   \min \ \  \im{l^{-1}} { a_k } \right),
\end{align*}
from definition \ref{RefDefAddW} and that of minimum we must draw
\begin{gather}
\label{RefEq11}
\Con{D} { X_k \cup \left\{ v_k \varphi_k \right\}}
\\
a_k \neq \emptyset
\\
\Inc{D}
{ 
X_k \cup 
\left\{ \symbsubst {v_k} {v'_k}{\varphi_k} \right\} 
 }.
\end{gather}
By the last fact, we also have 
$ \Inc{D}  { X_k \cup 
\left\{ \symbsubst {v_k}{v'_k}{\varphi_k} \right\} 
\cup 
\left\{ \eq {\vv} \right\} 
} $, which gives 
$ X_k \cup \left\{ 
\symbsubst {v_k}{v'_k}{\varphi_k}
\right\} \proves{D} \xnot \veq $ 
by hypothesis \eqref{RefEq21}. 
Hence consider a finite set of formulas 
$ \Gamma \C X_k \cup \left\{ \symbsubst{v_k}{v'_k}{\varphi_k} \right\}$
such that $ \pair {\Gamma}{\xnot \veq} \in \derivables{m+1}{D} \left( \emp \right)$ for some $m \in \N$. 
This in particular implies $\Gamma \proves{D} \xnot \veq$; since it is also true that $ \Gamma \proves{D} \veq$ by hypothesis \eqref{RefEq10}, it must be $\Gamma \not \C X_k$, because $\Con{D}{X_k}$ by \eqref{RefEq11}. 
Then
\begin{align*}
\pair{\Gamma' \cup \left\{ v_k \varphi_k\right\}}{\xnot \veq} 
\in \rWitnessA \left( \left\{ \pair{\Gamma}{\veq} \right\} \right) \C 
\rWitnessA \left( \derivables{m+1}{D} \left( \emp \right) \right)
\\
\overset{\text{\tiny{\ref{RefDefOneStep}}}}{\C} \derivables{1}{D} \left( \derivables{m+1}{D} \left( \emp \right) \right)
= \derivables{m+2}{D} \left( \emp \right),
\end{align*}
where first inclusion is given by monotonicity of $\rWitnessA$, and we set 
$ \Gamma' := \Gamma \sdiff \left\{ \symbsubst{v_k}{v'_k}{\varphi_k} \right\} \C X_k$.
This contradicts \eqref{RefEq11}.
\end{proof}

\begin{Not}
If $S$ is a countable language, one can always find 
$l \in \mapsFromTo{\N}
{\left( \im{\ari_S ^ {-1}}{\left\{ 0 \right\}} \right)}$, 
$f \in \mapsFromTo{\N}{ \left( \wffs{S} \right)}$ 
such that
\begin{align*}
\N \ni n \mapsto l \left( n \right) f\left( n \right)
\end{align*}
is onto $\dom \witnessRel{S}$.
This surjectivity property aside, we will not be interested in how $l$ and $f$ actually work, and we will thus write $\addW{D}{}$ instead of $ \addW{D}{l,f}$ when dealing with a ruleset $D$ of a countable language, implying $l$ and $f$ satisfy it.
\end{Not}

\begin{Lm}
\label{RefThmAddWFurnished}
Let $D$ be a ruleset of a countable language $S$. Assume that the sets $X$, $Y$ satisfy:
\begin{enumerate}
\item
$\Con{D}{Y}$;
\item
$\addW{D}{} \left( X \right) \C Y$;
\item
\label{RefEq13}
$\im{ \ari_S^{-1} }{ \left\{ 0 \right\}} \sdiff \symbof{X}$ 
is not finite.
\end{enumerate}
Then $Y$ is $S$-\witnessed{}.
\end{Lm}
Note that no particular request is placed on $D$.

\begin{proof}
$\addW{D}{} = \addW{D}{l, f}$ for some pair of maps $l, f$.
Thanks to hypothesis \eqref{RefEq13} we have, referring to \ref{RefDefAddW}: 
\begin{align}
\label{Ref010}
a_m \neq \emptyset && \forall m \in N.
\end{align}
Now assume $v \varphi \in Y$. 
By surjectivity, there is $n \in \N$ such that
$ v \varphi = l \left( n \right) f\left( n \right)$.
Set 
\begin{align*}
v_n :=& l \left( n \right) 
\\
\varphi_n :=& f\left( n \right) 
\\
v'_n :=& l \left( \min \im{l^{-1}}{a_n} \right).
\end{align*}
We have 
$X_n \C \addW{D}{} \left( X \right) \C Y$ and 
$ \left\{ v_n \varphi_n \right\} \C Y$, whence $\Con{D}{ X_n \cup \left\{ v_n \varphi_n \right\} }$, which, 
together with~\eqref{Ref010}, implies either
\begin{align*}
& X_n \cup
\left\{ \symbsubst{v_n}{v'_n}{ \varphi_n } \right\}
= X_{n+1} \C \addW{D}{} \left( X \right) \C Y
\\
\text{or} &
\\
& X_{n+1} = X_n \text{ and } X_n \cap \im {\witnessRel{}} {\left\{ v_n \varphi_n \right\}} \neq \emp
\end{align*}
by definition~\ref{RefDefAddW}. 
Given the arbitrariness of $v \varphi$, this yields thesis as demanded by \ref{RefDefWitnessed}.
\end{proof}

\subsection{Consistent maximization for countable languages}

\begin{Def}
\label{RefDefAddF}
Given a mapping $ f: \N \ni n \mapsto \varphi_n \in F_S $, recursively define
\begin{align*}
X_0 &:= X
\\
X_{n+1} &:=
\begin{cases}
X_n \cup \left\{ \xnot {\varphi_n} \right\}
& 
\text{ if } X_n \proves{D} \xnot { \varphi_n }
\\
X_n \cup \left\{ \varphi_n \right\} 
&
\text{ otherwise},
\end{cases}
\end{align*}
and set 
\begin{align*}
\addF{D}{f} \left( X \right) := \bigcup_{n \in \N} X_n.
\end{align*}
\end{Def}

\begin{Lm}[Lindenbaum's lemma]
\label{RefThmLindenbaum}
If $D$ is a \cutLike{} ruleset of $S$ and $f \in \mapsFromTo{\N}{\left( \wffs{S} \right)}$, then
$\Con{D}{X}$ implies $\Con{D}{\addF{D}{f} \left( X \right)}$ for any $X$.
\end{Lm}
\begin{proof}
Assume $\Inc{D} { \addF{D}{f} \left( X \right) }$; then $\min \left\{ 
n \in \N \st \ \Inc{D} { X_n  }
\right\} \in \Z^+$ (it cannot be zero because $X=X_0$ is consistent by hypothesis), so it equals $m+1$ for some $m \in \N$.
Set $\varphi_m := f\left( m \right)$.

Now, it cannot be 
$X_{m+1} = X_m \cup \left\{ \varphi_m \right\}$, 
for in this case we would get 
$X_m \nvdash_D \xnot {\varphi_m }$ 
by definition~\ref{RefDefAddF}, and, as a consequence, 
$\Con{D} {  X_m \cup \left\{ \varphi_m \right\} }$ 
by \ref{RefDefCutLike}, while $X_{m+1}$ is inconsistent.
Hence the upper branch of definition \ref{RefDefAddF} must be the one in charge, that is 
\begin{align}
\label{Ref015}
X_{m+1} = X_m \cup \left\{ \xnot {\varphi_m} \right\},
\end{align}
and consequently
\begin{equation}
\label{Ref016}
X_m \vdash_D \xnot {\varphi_m}.
\end{equation}
On the other hand, from \eqref{Ref015} and \ref{RefDefCutLike} it descends that
\begin{equation*}
X_m \vdash_D \xnot {\xnot {\varphi_m}},
\end{equation*}
yielding, together with \eqref{Ref016}, that $X_m$ is inconsistent according to definition~\ref{RefDefConsistencies}, thus contradicting minimality of $m+1$.

\end{proof}

\begin{Not}
If $S$ is a countable language, one can always find 
$f \in \mapsFromTo{\N}{ \left( \wffs{S} \right)}$ 
being onto $\wffs{S}$.
This surjectivity property aside, we will not be interested in how $f$ actually works, and we will thus write $\addF{D}{}$ instead of $ \addF{D}{f}$ when dealing with a ruleset $D$ of a countable language, implying $f$ satisfies it.
\end{Not}

\begin{Prop}
\label{RefThmAddFCover}
Let $D$ be a ruleset of a countable language $S$.
$\addF{D}{} \left( X \right)$ is a cover of $S$.
\end{Prop}

\begin{proof}
$\addF{D}{} = \addF{D}{f}$ for some function $f$ onto $\wffs{S}$. 
Consider $\varphi \in F_S$, and, by surjectivity, a natural number 
$n$ such that $\varphi = \varphi_n := f \left( n \right)$.
\\
Either $X_n \nvdash_D \xnot {\varphi_n}$ or
$X_n \vdash_D \xnot {\varphi_n}$, 
where $X_n$ is as from \ref{RefDefAddF}.
Therefore, by~\ref{RefDefAddF}, 
either $X_{n+1} = X_n \cup \left\{ \varphi_n \right\}$ or 
$X_{n+1} = X_n \cup \left\{ \xnot { \varphi_n } \right\}$, and $X_{n+1} \C \addF{D}{} \left( X \right)$.
Thus at least one between
\begin{equation*}
X_n \cup \left\{ \varphi_n \right\}
\end{equation*} 
and
\begin{equation*}
X_n \cup \left\{ \xnot {\varphi_n} \right\}
\end{equation*}
is a subset of 
$\addF{D}{} \left( X \right)$, giving that at least one between
$\xnot {\varphi_n}$ and $\varphi_n$ belongs to $\addF{D}{} \left( X \right)$.
This, by the arbitrariness of $\varphi$ and by definition~\ref{RefDefConsistencies}, ends the proof.
\end{proof}

\begin{Rem}
Together, \ref{RefThmLindenbaum} and \ref{RefThmAddFCover} yield that a $D$-consistent set $X$ can be completed to the $D$-consistent cover $\addF{D}{} \left( X \right)$.
This, until one adds the request of $D$ being \asslikeW{} (see \ref{RefDefAssLikeWeak}), does not generally imply that it can be completed to a maximally consistent set, which is the thesis of the standard formulation (see, e.g., \cite{MR1314201}, section III.2 and \cite{chellas1980modal}, 2.19) of Lindenbaum's lemma.
\end{Rem}

\section{Putting it all together}
\label{RefSectAllTogether}
\begin{Lm}
\label{RefThmEnlarge}
Let $D$ be a ruleset of a countable language $S$, and $X$ be a set; assume they  comply with the following requirements:
\begin{enumerate}
\item
\label{RefEq14}
$ \rWitnessA \in D$;
\item
\label{RefEq15}
$D$ is \cutLike{};
\item
\label{RefEq16}
$D \emul{\emp} \left\{ \rEqRefl \right\}$;
\item
\label{RefEq17}
$ \im{\ari_S^{-1}}{\left\{ 0 \right\}} \sdiff \symbof{X}$ is not finite;
\item
\label{RefEq18}
$\Con{D}{X}$.
\end{enumerate}
Then $\addF{D}{} \left( \addW{D} {} \left( X \right) \right)$ is a \witnessed{}, $D$-consistent $S$-\cover{}.
\end{Lm}

\begin{proof}
Set
\begin{align*}
Y:= \addW{D}{} \left( X  \right) && Z := \addF{D}{} \left(  Y   \right) .
\end{align*}
$Z$ is a cover by \ref{RefThmAddFCover}. 
By \eqref{RefEq14}, \eqref{RefEq15}, \eqref{RefEq16}, \eqref{RefEq18} and \ref{RefThmAddWCon}, $Y$ is $D$-consistent. 
Consequently $Z$ is $D$-consistent as well by \ref{RefThmLindenbaum}, \eqref{RefEq15}.
This fact, fed together with \eqref{RefEq17} into \ref{RefThmAddWFurnished}, grants that $Z$ is $S$-witnessed, ending the proof.
\end{proof}

\begin{Lm}
\label{RefThmSat1}
Let $D$ be a ruleset of the language $S$, and $X$ be a set such that
\begin{enumerate}
\item
\label{RefEq27}
$S$ is countable;
\item
$ \Con{D}{X}$;
\item
\label{RefEq19}
$ \im{\ari^{-1}}{\left\{ 0 \right\}} \sdiff \symbof{X}$ is not finite;
\item
$\rWitnessA \in D$;
\item
$D$ is \cutLike{};
\item
$D \emul{\emp} \left\{ \rAssumption, \rEqRefl, \rEqSymm, \rEqTrans, \rFunc, \rRel, \rEx, 
\rNor
\right\}$.
\end{enumerate}
Then 
$ \henk{D}
{\addF{D}{} \left( \addW{D} {} \left( X \right) \right)} 
\models{} X $.
\end{Lm}

\begin{proof}
Set $Y:= \addF{D}{} \left( \addW{D} {} \left( X \right) \right) \supseteq X $. 
By \ref{RefThmInclusionImpliesEmulation} and \ref{RefThmEmulTransitive}, $D \emul{\emp} \left\{ \rEqRefl \right\}$, so \ref{RefThmEnlarge} can be invoked: $Y$ is a witnessed, $D$-consistent $S$ cover.
Analogously, $D \emul{\emp} \left\{ \rAssumption \right\}$, so that $D$ is \asslikeS{} by \ref{RefThmAssLikeExample} and \ref{RefThmAssLikeTransitive}.
By \ref{RefDefAssLikeWeak}, then, $Y$ is also a \closed{$D$} \mincov{}.
Hence, $\varphi \in Y \Leftrightarrow \eval{\henk{D}{Y}} \left( \varphi \right) = 1$ by \ref{RefThmHenkin2}.
In particular, $\henk{D}{Y} \models{} X$.
\end{proof}

\begin{Prop}
\label{RefThmCutLike1}
If $\left\{ \rThin, \rCut \right\} \C D$ and 
$D$ is \monotone{}, then $D$ is \cutLike{}.
\end{Prop}

\begin{proof}
Consider a set $X$ and a wff $\varphi$ such that $\Inc{D}{X \cup \left\{ \varphi \right\}}$. 
We have to show that $X \proves{D} \xxnot \varphi$.
By assumption, there are $\Gamma_1, \Gamma_2 \C X \cup \left\{ \varphi \right\}$ finite, and $\psi$ such that
$ \pair{\Gamma_1}{\psi} \in \derivables{m}{D} \left( \emp \right)$ and
$ \pair{\Gamma_2}{\xxnot \psi} \in \derivables{n}{D} \left( \emp \right)$ for some $m, n \in \N$.
Since $D$ is \monotone{}, by \ref{RefThmMonotonicityIterationEmpty}, we have 
$ \left\{ \pair{\Gamma_1}{\psi}, \pair{\Gamma_2}{\xxnot \psi} \right\} \C \derivables{m+n}{D} \left( \emp \right)$.
So
\begin{align*}
\pair{\Gamma}{\xxnot \varphi} \in \rCut \left( \left\{ 
\pair{\Gamma \cup \left\{ \varphi \right\}}{\psi}, 
\pair{\Gamma \cup \left\{ \varphi \right\}}{\xxnot \psi}
\right\} \right) 
\\
\C \rCut \left( 
\rThin \left( 
\left\{ \pair{\Gamma_1}{\psi}, \pair{\Gamma_2}{\xxnot \psi} \right\}
\right)
\right)
\C 
\rCut \left( 
\rThin \left( 
\derivables{m+n}{D} \left( \emp \right)
\right)
\right)
\\
\C
\rCut \left( 
\onestep{D} \left( 
\derivables{m+n}{D} \left( \emp \right)
\right)
\right)
\C
\onestep{D} \left( 
\onestep{D} \left( 
\derivables{m+n}{D} \left( \emp \right)
\right)
\right) = 
\derivables{2+m+n}{D} \left( \emp \right),
\end{align*}
where we set $\Gamma := \Gamma_1 \cup \Gamma_2 \sdiff \left\{ \varphi \right\}$.
Hence $ X \sdiff \left\{ \varphi \right\} \proves{D} \xxnot \varphi$.
\end{proof}

\begin{Cor}
\label{RefThmSat2}
Given a countable language $S$, and $X$ such that 
$ \im{\ari_S^{-1}}{\left\{ 0 \right\}} \sdiff \symbof{X}$ is not finite, suppose $X$ is $D_0$-consistent, where 
\begin{align*}
D_0 :=  \left\{ 
\rAssumption, \rEqRefl, \rEqSymm, \rEqTrans, \rFunc, \rRel, \rNor, \rEx, \rWitnessA, \rCut, \rThin \right\}.
\end{align*}
Then
\begin{align*}
\henk{D_0}{ \addF{D_0}{} \left( \addW {D_0}{} \left( X \right) \right)} \models{} X.
\end{align*}
\end{Cor}

\begin{proof}
From the fact that $D_0$ is \monotone{} we can draw two conclusions:
$D_0 \emul{\emp} \left\{ \rAssumption, \rEqRefl, \rEqSymm, \rEqTrans, \rFunc, \rRel, 
\rNor, \rEx 
\right\}$, by \ref{RefThmInclusionImpliesEmulation}, and
$D_0$ is \cutLike{} by \ref{RefThmCutLike1}, so that
\ref{RefThmSat1} can be invoked.
\end{proof}

We now want to get rid of requirement \eqref{RefEq19} in the statement of \ref{RefThmSat1}.
This will be accomplished by following the standard path of adjoining to (the symbol set of) the language $S$ a countably infinite family $N$ of fresh literals, enlarging it to a second language $S_N$; then \ref{RefThmSat1} is applied to $S_N$, and carried on to $S$, being the latter a restriction of the former.
To do this, we have to show the natural fact that satisfaction relation, \ref{RefDefSatisfaction}, is preserved through such enlargements and restrictions:

\begin{Lm}[Coincidence lemma]
\label{RefThmCoincidence}
Let $S_1$, $S_2$ be languages.
Let $ i_1$ $i_2 $ be interpretations, of $S_1$ and $S_2$ respectively, over the same universe $U$.
Assume that
\begin{enumerate}
\item
$ \equiv_{S_1} = \equiv_{S_2}$;
\item
$\nor_{S_1} = \nor_{S_2} $;
\item
$ \restrict{\left( \ari_{S_1}  \right) }{dom \left( \ari_{S_1}  \right)} = \restrict{ \left( \ari_{S_2} \right) } {dom \left( \ari_{S_1}  \right)}$;
\item
$ \restrict{i_1} {dom \left(  \ari_{S_1} \right) } = \restrict{i_2} {dom \left( \ari_{S_1}  \right) }$.
\end{enumerate}
Then 
$\wffs{S_1} \C \wffs{S_2}$  and 
$\restrict{\eval{i_1}}{\wffs{S_1}} = \restrict{\eval{i_2}}{\wffs{S_2}}$.
\end{Lm}

Proof of \ref{RefThmCoincidence} turns out to be tedious, giving rise to a `de Bruijn surge{}': its proof in \M{} seem disproportionally verbose with respect to both its informal counterparts and the simplicity of the intuitive idea conveyed, so that its de Bruijn factor (see \ref{RefNum}) sharply increases: that same proof takes less than one page in (\cite{0387908951}, III.5.1). 
Whether this fact depends inherently on the result or the chosen formalization system, or even on the coder not devising a better proof seems very hard to assess.
The reader is thus referred to \M{} sources for that proof (\verb|FOMODEL3.MIZ:12|).

\begin{Thm}[Satisfiability theorem]
\label{RefThmSat3}
Suppose that
\begin{enumerate}
\item
\label{Ref027}
$S$ is a countable language;
\item
\label{Ref028}
$X \C F_S$;
\item
\label{RefEq20}
$\rWitnessA \in D$;
\item
\label{Ref047}
$D$ is \cutLike{};
\item
\label{Ref030}
$D \emul{\emp} \left\{ \rAssumption, \rEqRefl, \rEqSymm, \rEqTrans, \rFunc, \rRel, \rEx, \rNor \right\} $;
\item
\label{Ref031}
$\Con{D}{X}$.
\end{enumerate}
Then there is an interpretation of $S$ having a countable universe and satisfying $X$.
\end{Thm}

\begin{proof}
Consider a countably infinite set $N$ missing both $S$ and $\symbof {X}$, and the language $S_N$ extending $S$ and obtained by setting 
\begin{align*}
\equiv_{S_N} &:= \equiv_S
\\
\nor_{S_N} &:= \nor_S
\\
\ari_{S_N} &:= N \cartprod \left\{ 0 \right\} \cup \ari_S. 
\end{align*}
By construction, $S_N$ is countable (because $S$ and $N$ are) and 
$
N \C \ari_{S_N} ^{-1} \left( \left\{ 0 \right\} \right) \sdiff \symbof {X}.
$
Now set
\begin{multline*}
D_N :=\left\{ 
{\rAssumption}_{, S_N},
{\rThin}_{, S_N},
{\rEqRefl}_{, S_N}, 
{\rEqSymm}_{, S_N}, 
{\rEqTrans}_{, S_N}, 
{\rFunc}_{, S_N},
\vphantom{
{\rRel}_{, S_N}, 
{\rNor}_{, S_N},
{\rEx}_{, S_N}, 
{\rWitnessA}_{, S_N}
{\rCut}_{, S_N}
}
\right.
\\
\left.
\vphantom{
{\rAssumption}_{, S_N},
{\rThin}_{, S_N},
{\rEqRefl}_{, S_N}, 
{\rEqSymm}_{, S_N}, 
{\rEqTrans}_{, S_N}, 
{\rFunc}_{, S_N},
}
{\rRel}_{, S_N}, 
{\rNor}_{, S_N},
{\rEx}_{, S_N}, 
{\rWitnessA}_{, S_N}
{\rCut}_{, S_N}
\right\}.
\end{multline*}
Suppose we manage to show
\begin{align}
\label{Ref025}
\Con{D_N} { X }.
\end{align}

Then we can deploy \ref{RefThmSat2}, and infer that
\begin{align}
\label{Ref026}
H_{D_N, \addF{D_N}{} \left( \addW {D_N} {} \left( X  \right) \right) } \vDash X.
\end{align}
The very final step towards thesis is to realize
that 
$ H_{D_N, \addF{D_N}{} \left( \addW {D_N} {} \left( X  \right) \right) } $ 
can be restricted to an interpretation $i$ of $S$, and that this latter interpretation returns the same truth value as 
$H_{D_N, \addF{D_N} {} \left( \addW {D_N} {} \left( X  \right)  \right) } $ on every formula of $X$ thanks to \ref{RefThmCoincidence}, so that 
$i \models{} X$ by \eqref{Ref026}.

\begin{description}
\item[Subproof for claim \eqref{Ref025}]
It will suffice to show \eqref{Ref025} holds for a generic \emph{finite} $Y \C X$
:
\begin{align}
\Con{D_N}{ Y }.
\end{align}
Thus, let $Y \C X$, $Y$ being finite. 
Now, $\Con {D} {Y}$ 
(use hypothesis \eqref{Ref031}%
) and 
$\ari_{S}^{-1} \left( \left\{ 0 \right\} \right) \sdiff \symbof Y$ 
is not finite, so 
$H_{D, \addF{D}{} \addW {D} {} Y} \vDash Y$ by \ref{RefThmSat1} and hypotheses~\eqref{RefEq20}, \eqref{Ref047} and~\eqref{Ref030}.
Consider an interpretation $i_{N,Y}$ of $S_N$ obtained by extending 
$H_{D, \addF {D}{} \addW {D} {} Y}$ to $S_N$ arbitrarily: 
we can do so keeping the universe of $i_{N,Y}$ the same as that of $H_{D, \addF {D} {} \addW {D} {} Y}$, 
so that, given $\varphi \in Y$, 
one has (again by \ref{RefThmCoincidence}) $\eval{i_{N,Y}} \left( \varphi \right) = \eval{H_{D, \addF {D} {} \addW {D} {} Y}} \left( \varphi \right)$; hence 
$i_{N,Y} \models{S_N} Y$. 
This in the end implies $\Con{D_N} Y$, as $D_N$ is \sound{}.
\end{description}
\end{proof}

\begin{Cor}
\label{RefThmSat4}
Let $D$ be a ruleset of a countable language $S$, and $X \C \wffs{S}$. Suppose
\begin{enumerate}
\item
$ \left\{ \rWitnessA, \rCut, \rThin \right\} \C D$;
\item
$D$ is \monotone{};
\item
$D \emul{\emp} \left\{ \rAssumption, \rEqRefl, \rEqSymm, \rEqTrans, \rFunc, \rRel, \rEx, \rNor \right\} $;
\item
$ \Con{D}{X}.$
\end{enumerate}
Then there is an interpretation of $S$ having a countable universe and satisfying $X$.
\end{Cor}

\begin{Cor}[Countable downward \L{}-Skolem theorem]
\label{RefThmSkolem}
Assume $X \C \wffs{S}$ is countable, and suppose there is an interpretation $i$ of $S$ such that $ i \models{} X$. Then there is an interpretation $i'$ of $S$ having a countable universe and satisfying $X$ as well.
\end{Cor}

\begin{proof}
Let $N$ be a countably infinite subset of the symbol set of $S$ such that $ \symbof{X} \cup \left\{ \equiv_S, \nor_S \right\}  \C N $.
Restrict $\ari_{S}$ and $i$ to $N$, obtaining respectively a countable language $S'$ and an interpretation $i'$ of the latter over the same universe of $i$.

For any $\varphi \in X$, one has that $\varphi$ is also a formula of $S'$, and that $\eval{i} \left( \varphi \right) = \eval{i'} \left( \varphi \right)$ by construction and coincidence lemma, \ref{RefThmCoincidence}, so that $i' \models{S'} X$, and hence $\Con{D} {X}$, where we set
\begin{align*}
D := \left\{ \rAssumption, \rEqRefl, \rEqSymm, \rEqTrans, \rFunc, \rRel, \rEx, \rNor, \rWitnessA, \rCut, \rThin  \right\},
\end{align*}
thanks to soundness.
This allows to consider an interpretation $j'$ of $S'$ having a countable universe and satisfying $X$ by \ref{RefThmSat4}.
This latter interpretation can be arbitrarily enlarged to one of $S$ with the same universe, preserving the satisfiability of $X$ through it (again thanks to coincidence lemma), and thus yielding thesis.
\end{proof}

\begin{Rem}
We note that the language $S$ in \ref{RefThmSkolem} is not required to be countable.
\end{Rem}

\begin{Def}[Entailment]
\label{RefDefEntailment}
Given sets $X$, $Y$, we say that $X$ \emph{entails}\index{entailment} $Y$ with respect to the language $S$ if any interpretation $i$ of $S$ satisfying $X$ also satisfies $Y$. 

In this case we write $X \entails{S} Y$ or just $ X \entails{} Y$. 
We also will usually write $X \entails{S} \varphi$ (or $X \entails{} \varphi$) in lieu of $ X \entails{S} \left\{  \varphi \right\}$.
\end{Def}

\begin{Rem}
The symbol $\entails{}$ results thus overloaded by definitions of satisfaction (\ref{RefDefSatisfaction}) and entailment (\ref{RefDefEntailment}).
The type of the argument on its left will usually resolve which use is being made. 
\end{Rem}

\begin{Cor}[of \ref{RefThmSat3}]
\label{RefThmCompleteness1}
Let $X$ be a subset of the set of formulas $\wffs{S}$ of a countable language $S$, and $D$ be  a \cutLike{} ruleset of $S$ such that
\begin{enumerate}
\item
$ \rWitnessA \in D$
\item
$D \emul{\emp} \left\{ \rAssumption, \rEqRefl, \rEqSymm, \rEqTrans, \rFunc, \rRel, \rEx, 
\rNor \right\}.$
\end{enumerate}
Then $X \models{} \xnot \varphi$ implies 
$ X \proves{D} \xnot \varphi $ for any 
$\varphi \in \wffs{S}$.
\end{Cor}

\begin{proof}
By contradiction. Suppose that $ X \proves{D} \xnot \varphi $ is false. 
Then, $D$ being \cutLike, $\Con{D} {X \cup \left\{ \varphi \right\}}$.
Hence, allowed by \ref{RefThmSat3}, let us consider an interpretation $i$ of $S$ such that 
\begin{gather}
\label{Ref034}
\eval{i} \left( \varphi \right) = 1 
\\
\label{Ref035}
i \models{} X.
\end{gather}
Given the hypothesis, $ X \entails{} \nnot \varphi $, so that by definition of entailment and \eqref{Ref035}, $ \eval{i} \left( \nnot \varphi \right) =1$.
Now, by \ref{RefDefEvalCompound}, 
$\eval{i} \left( \varphi \right) = 0$, contradicting \eqref{Ref034}.
\end{proof}

\begin{Thm}[\G{}'s completeness theorem]
\label{RefThmCompleteness2}
\begin{align*}
X \models{} \varphi && \text{ implies } && X \proves{D_1} \varphi,
\end{align*}
where we set
\begin{align*}
D_1 := &  \left\{ \rAssumption, \rEqRefl, \rEqSymm, \rEqTrans, \rFunc, \rRel, \rEx, \rNor, \rWitnessA, \rCut, \rThin, \rIntuitionisticNightmare \right\}.
\end{align*}
\end{Thm}

\begin{proof}
Assume $ X \models{} \varphi$. 
Then $X \models{} \xnot {\xnot \varphi}$ by \ref{RefDefEvalCompound}.
This implies the existence of $ \Gamma \C X $ such that $ \pair {\Gamma} {\xnot {\xnot \varphi}} \in \Derivables{D_1 \sdiff \left\{ \rIntuitionisticNightmare \right\}} \left( \emp \right)$ by \ref{RefThmCompleteness1}. 
Hence there is $k \in \Z^+$ such that 
$ \pair {\Gamma} {\xnot {\xnot \varphi}} \in \derivables{k}{D_1} \left( \emp \right)$, so that
\begin{align*}
\pair {\Gamma } {\varphi} \in \rIntuitionisticNightmare \left( \left\{ \pair {\Gamma} {\xnot {\xnot \varphi}} \right\} \right)
\C \rIntuitionisticNightmare \left( \derivables{k}{D_1} \left( \emp \right) \right)
\\
\C \onestep{D_1} \left( \derivables{k}{D_1} \left( \emp \right) \right) 
= \derivables{k+1}{D_1} \left( \emp \right).
\end{align*}
\end{proof}

\section{Alternative rules}

The attributes `\asslikeW{}', `\asslikeS{}', and `\cutLike{}' have been introduced to detach, to some extent, the main results proven 
from the particular choice of derivation rules.
Indeed the rulesets occurring in hypotheses of main theorems we saw are often required to be applicable such attributes, rather than to include some specific rules.
This means that if one of those results is valid for a given ruleset, it remains valid if we substitute in that ruleset some rules satisfying a given attribute with others, as long as the new rules still make the ruleset satisfy the corresponding attribute.
As an example, consider the following pair of new rules.
\begin{Def}
\label{RefDefCutLikeFragment}
Given a literal $\vv$, define
\begin{align*}
& \rVThinFork :
\seqs{S} \supseteq
\Sigma \mapsto  \left\{ \left( \Gamma, \varphi \right): \exists \Gamma_1, \Gamma_2, \psi_0, \psi \st \quad
\pair{\Gamma_1} {\psi_0}, \pair {\Gamma_2} {\xnot {\psi_0}} \in \Sigma \text{ and }
\vphantom{
\varphi = \vcontr {\xxnot {\psi_0}} \text{ and } \Gamma = \Gamma_1 \cup \Gamma_2 \cup \left\{ \psi \right\}
}
\right.
\\
&
\vphantom{
\left( \Gamma, \varphi \right): \exists \Gamma_1, \Gamma_2, \psi_0, \psi \st \quad
\pair{\Gamma_1} {\psi_0}, \pair {\Gamma_2} {\xnot {\psi_0}} \in \Sigma \text{ and }
}
\left.
\varphi = \vcontr {\xxnot {\psi_0}} \text{ and } \Gamma = \Gamma_1 \cup \Gamma_2 \cup \left\{ \psi \right\}
\right\}
\C \seqs{S}
\\
& \rVCut : 
\seqs{S} \supseteq
\Sigma \mapsto  
\\
& \left\{ \pair {\Gamma} {\varphi}: \exists \psi, \psi_0 \st \pair{\Gamma \cup \left\{ \psi \right\}} {\vcontr{\xxnot {\psi_0}}} \in \Sigma \text{ and } \varphi = \xxnot {\psi} \text{ and } \Gamma \sdiff \left\{ \psi \right\} = \Gamma  
\right\}
\end{align*}
\end{Def}

\begin{Not}
We also give the diagram representation (introduced in section \ref{RefSectDiagrams}) for rules defined in \ref{RefDefCutLikeFragment}: 
\begin{align*}
\rVThinFork: & 
\begin{aligned}
\begin{aligned}
\Gamma_1 && \vdash && \psi
\end{aligned}
&&
&&
\begin{aligned}
\Gamma_2 && \vdash && \xxnot{\psi}
\end{aligned}
\\
\hline
&& 
\begin{aligned}
\Gamma_1 && \Gamma_2 && \varphi && \vdash && \vcontr{\xxnot {\psi}}
\end{aligned}
&&
\end{aligned}
\\
\rVCut: &
\begin{aligned}
\Gamma && \varphi && \vdash && \vcontr{\xxnot{\psi}}
\\
\hline
\Gamma && && \vdash && \xxnot{\varphi}
\end{aligned}
\end{align*}
\end{Not}

The following result, mirroring \ref{RefThmCutLike1}, permits to replace $\left\{  \rThin, \rCut   \right\}$ with 
$ \left\{ \rVThinFork, \rVCut \right\}$ in the statement of \ref{RefThmCompleteness2}.

\begin{Prop}
A \monotone{} ruleset $D \emul{} \left\{ \rVThinFork, \rVCut \right\}$ is \cutLike{}. 
\end{Prop}

\begin{proof}
Assumed $\Inc{D} {X \cup \left\{ \varphi \right\}} $, we must show $ X \proves{D} \xxnot \varphi$.
There are $\Gamma_1 , \Gamma_2 \C X \cup \left\{ \varphi \right\}$, $m, n \in \Z^+$, $\psi$ such that
$ \pair {\Gamma_1} {\psi}  \in \derivables{m}{D} \left( \emp \right)$, 
$\pair{\Gamma_2} {\xxnot \psi} \in \derivables{n}{D} \left( \emp \right)$.
By the fact that $D$ is \monotone{}, we have 
$ \pair {\Gamma_1} {\psi}, \pair{\Gamma_2} {\xxnot \psi} \in \derivables{p}{D} \left( \emp \right)$, where $p := \max \left\{ m, n \right\}$. 
Now
\begin{align*}
\pair {\Gamma_1 \cup \Gamma_2 \cup \left\{ \varphi \right\}} {\vcontr{\xxnot \psi}} \in 
\rVThinFork \left( \left\{ 
\pair {\Gamma_1} {\psi}, 
\pair {\Gamma_2} {\xxnot {\psi}}
\right\} \right) 
\C \rVThinFork \left( \derivables{p}{D} \right) 
\\
\C \Derivables{D} \left( \derivables{p}{D} \left( \emp \right) \right) 
\ra{} \exists q \in \N \st \quad 
\pair {\Gamma_1 \cup \Gamma_2 \cup \left\{ \varphi \right\}} {\vcontr{\xxnot \psi}} \in 
\derivables{q}{D} \left( \derivables{p}{D} \left( \emp \right) \right),
\end{align*}
so that
\begin{align*}
\pair {\Gamma_1 \cup \Gamma_2 \cup \left\{ \varphi \right\} \sdiff \left\{ \varphi \right\}} {\xxnot \varphi} \in
\rVCut \left( \left\{
\pair{\Gamma_1 \cup \Gamma_2 \cup \left\{ \varphi \right\}} { \vcontr {\xxnot \psi} }
\right\} \right)
\\
\C 
\rVCut \left( \derivables{p+q}{D} \left( \emp \right) \right) 
\C \Derivables{D} \left( \derivables{p+q}{D} \left( \emp \right) \right)
\\
\ra{} 
\exists l \in \N \st \quad
\pair {\Gamma_1 \cup \Gamma_2 \cup \left\{ \varphi \right\} \sdiff \left\{ \varphi \right\}} {\xxnot \varphi} \in \derivables{l}{D} \left( \derivables{p+q}{D} \left( \emp \right) \right).
\end{align*}
\end{proof}

\chapter{The formalization}
\label{RefSectFormalization}
This chapter illustrates the actual \M{} implementation of the set-theoretical treatment of first-order languages built in chapter \ref{RefSectFormulation}; 
it includes material from \cite{caminati2010basic} and \cite{CaminatiJar2011}.
Introductory sections \ref{RefSectSoftware} and \ref{RefSectMizarOverview} give background on proof checkers and on the particular proof checker chosen in our case, respectively.

\section{Software for proving}
\label{RefSectSoftware}
Rigor and creativity are both essential qualities of mathematics.
Logic supplies precise notions of rigor, and tools to attain it: for example, Zermelo-Fraenkel set theory with the axiom of choice (ZFC) is commonly accepted as a first-order axiom system in which most parts of current mathematics could be rendered; 
however, such renditions (commonly referred to as formalizations) are usually reputed to be tedious if not impracticable, and anyway a hindrance for the creative process, equally essential for mathematics.
Thus, instead of actually formalize mathematics, the classical compromise is to supply a sketch of formalization in a variably rigorous pseudo-code, the purpose of which is to get accepted (and thus possibly trusted, relied on and employed, in the end) as a result of what is ultimately a social process: the one of persuading other people of its correctness (\cite{asperti2009social}).


The success of Hilbert's program in thrusting towards formalization of mathematics and the advent of digital computers set the scene for a change. 
Virtually every scientific realm presents examples of endeavors which were unthinkable before the advent of computers: given the evident affinity between formalization and mechanization, one can arguably maintain that formalization of mathematics might well become such an endeavor (\cite{boyer1994qed}, \cite{wiedijk2007qed}).%
\footnote{Recent years have provided a further strong reason, probably not foreseeable at the time in which \cite{boyer1994qed} was written, to be optimistic about the feasibility of this endeavor: the several blatant and huge successes brought by the commons-based peer production model (\cite{benkler2006wealth}), like, most notably, the GNU/Linux operating system and the Wikipedia project.%
}
And indeed, since de~Bruijn's Automath (\cite{de1970mathematical}), the software implementations of proof checkers proliferated.%
\footnote{\url{http://www.cs.ru.nl/~freek/digimath/}}





In the vast landscape of software born to carry out the old idea of mechanizing proofs, a first distinction can be drawn between \emph{proof checkers} (like \M{}, Metamath, Twelf, Automath) and \emph{automated theorem provers} (like E, ACL2, SPASS, Vampire).
The latter \emph{find} proofs, rather than merely certifying them. 
One of the first known concrete computer programs developed for proving, namely \emph{Logic Theorist}\index{Logic Theorist}, \cite{newell1956logic}, was a representative of this category.

\emph{Proof assistants}, or \emph{interactive theorem provers} (Coq, Isar, Matita, PhoX, to name a few), stand between the two ends, requiring some user intervention, the amount and form of which varies greatly among different systems, to guide the proof, yet saving him to spell out a full proof.

There is a further family of recent projects (\cite{cramer2010naproche}%
\footnote{At the time of writing, Naproche seems the only one in this family having made tangible progress, to the point of offering a web interface:
\wwwNaproche{}, with \LaTeX{} support.
}%
, \cite{humayoun46mathnat}, \cite{schodltowards}) taking an alternative, `linguistic{}' approach: the very rough idea is to supply a `controlled natural language{}' coupled with some automated prover which validate the formal language extracted from the higher-level natural language.
This would relief mathematicians from both the burdens of proving the trivial details and of facing a language less friendly than the common mathematical language, with the controlled natural language acting as an interface with both the automated prover and the formal language backends.
Less ambitiously, ProofCheck (see \cite{nevelnproofcheck}) embeds a low-level proof checker directly into the \TeX{} and \LaTeX{} languages via additional \TeX{} macros.

The largest digital libraries of already formalized mathematics are those written with the proof checkers \M{}
, HOL Light, Coq and Isabelle.
\M{} 
is the most mathematically-oriented one, adopting a grammar resembling common mathematical language, a declarative style, and being based on set theory.

\section{An overview of \M{}}
\label{RefSectMizarOverview}
The \M{} project (\wwwMizar) delivers a few provisions:
\begin{enumerate}
\item
\label{RefItemLanguage}
\M{} \emph{language} permits to write formulas in first-order set theory which read close to common mathematical language.
For example, the formula 
\begin{align*}
 X \neq \emptyset \Longrightarrow \exists x ( x \in X )
\end{align*}
is written
\begin{equation*}
\verb|X <> {} implies ex x st x in X;| 
\end{equation*}
In addition to the few reserved words pertaining to the first-order alphabet of set theory,  
the language specifies grammar and reserved words to invoke the verifier (see point \ref{RefItemVerifier}) and to exploit  advanced features of the system.

\item
\label{RefItemVerifier}
\M{} \emph{verifier} (PC \M{}) is a software certifying whether one such formula can be deduced 
(according to some formal system for classical logic, see sections 2.2.1~and 3.5 of \cite{grabowski2010mizar}) from  other given formulas, specified
via the keyword \verb|by| of the \M{} language:
\begin{verbatim}
A1: x in X;
A2: for y being set holds y in X\/Y iff (y in X or y in Y);
x in X\/Y by A1, A2;
\end{verbatim}

\item
\label{RefItemLibrary}
The \M{} Mathematical Library (\MML{}) builds on the  components \eqref{RefItemLanguage} and \eqref{RefItemVerifier} above to provide a mass of \M{} language formulas certified, by \M{} verifier, to be derivable from a handful of set-theoretical axioms affine to ZFC axioms. 
The set theory resulting from these axioms, Tarski-Grothendieck (TG), is an extension of ZFC, and more on it can be found in \cite{rudnicki1999equivalents}.
\end{enumerate}

\MML{} is made up of \M{} source files called \emph{articles}, and its latest version is always browsable at \wwwMml{}.
In the following, we  will be using typewriter font for referencing articles and results inside \MML{}: for example, \verb|XBOOLE_1:4| denotes the fourth theorem appearing in the \MML{} article xboole\_1.miz, which is thus viewable at \wwwXboole{}.
We will also adopt typewriter font for \M{} code, as already done in point \eqref{RefItemLanguage} of the numbered list above.


\subsection{Types and definitions}
\label{RefSectMizarLayers}
The primitive workflow consisting of writing set-theoretical formulas, linking them together via the \verb|by| keyword, and invoking the verifier on them, as depicted in section \ref{RefSectMizarOverview}, would theoretically suffice to accomplish a great deal of first-order formalization tasks.
In practice, one cannot actually get very far without higher-level abstractions to structure the code. 
Among others, \M{} supplies (soft) \emph{types} and \emph{definitions}:
\begin{description}
\item[Types]
A term can be assigned a type (via the reserved word \verb|let|); 
as a consequence, the type of a term can be the subject of a first-order atomic formula. 
The special first-order relation symbol \verb|is| has exactly this use:
\begin{verbatim}
let x be Function;
x is Function;
\end{verbatim}
The formulas based on the special relation symbol \verb|is|, as the one above, present the distinctive property of needing no justification: they are a way to query \M{} type system.
This means that the last line of code in the example above is accepted by the verifier without the need of a \verb|by| statement (see item \eqref{RefItemVerifier} on page \pageref{RefItemVerifier}).
The basic type \verb|set| is applicable to any term.
\item[Functors]
New function symbols (called \emph{functors} in \M{} jargon) can be added to the first order language via the reserved word \verb|func|. 
This can be done in two ways: 
\begin{enumerate}
\item
\label{RefItemEquals}
either in a macro-like fashion:
\begin{verbatim}
definition
  let x, y be set;
  func [x,y] equals { { x,y }, { x } };
  ...
\end{verbatim}
In this case, the keyword \verb|equals| is used.

\item
or by stating some formula the new object must satisfy, subject to the proof that exactly one term exists for which this happens:
\begin{verbatim}
definition
  let X, Y be set;
  func X /\ Y -> set means
  	for x being set holds x in it iff x in X & x in Y;
  existence
  proof
    ...
  end;
  uniqueness
  proof
    ...
  end;
end;
\end{verbatim}
In this case, the keyword \verb|means| is used, and the entity to be defined is denoted by the keyword \verb|it| in the definiens, as seen above.
\end{enumerate}
\end{description}
The two functionalities just introduced can work together, meaning that the definition of a functor can accept as arguments a finite list of \emph{typed} arguments; and, viceversa, the term obtained by the application of the defined functor can be associated a type (keyword \verb|->|):

\begin{verbatim}
definition
  let R be Relation;
  func R~ -> Relation means
    [x,y] in it iff [y,x] in R;
  ...
\end{verbatim}
After this association, the verifier will know the type returned by any application of that functor.
This suggests that the very presence of types can be a first, seminal step to some form of \A{}: some methods we shall see in section \ref{RefSectAutomations} rely on the capability of the system to know the type of each term straightaway, and all of them somehow revolve around the type system. 

Sometimes, the abstraction of types hides the fact that two functors behave the same way at the underlying set-theoretical level, even if they operate on, or yield, different types; in this case one can make the verifier aware that the results coincide, using the keyword \verb|identify|.
For example:

\begin{verbatim}
registration
  let x,y be real number, a,b be complex number;
  identify x+y with a+b when x = a, y = b;
  compatibility
  proof
    ...
  end;
end;  
\end{verbatim}
This correspondence can be achieved because the type system implemented in the verifier is a soft one (\cite{wiedijk2007mizar}): terms are actually untyped sets, and one can always forget about their type, which is offered for a matter of convenience.

\subsection{Attributes and registrations}
\label{RefSectAttributes}
\emph{Functorial registrations} are a further form of \M{} automation, and one of the most powerful and least restricted. 
To see how it works, we need to introduce attributes.
\subsubsection{Attributes}
Attributes are a flexible and natural way to define types; they are used to qualify and restrict a given type (called \emph{radix} type) by just prefixing it with the attribute name (or with its name preceded by the keyword \verb|non|, to negate it). 
For example, article \verb|XBOOLE_0| defines the attribute \verb|empty|, applicable to any term, so that one can write: 

\begin{verbatim}
{} is empty set;
\end{verbatim}

It is important to note that this juxtaposition is a subtype of the radix, and therefore can be treated like it under many aspects; at the same time, being itself a type, attributes can in turn be applied to it.
To put it differently, attributes can be clustered:

\begin{verbatim}
{} is empty finite set;
\end{verbatim}
This flexibility is a first reason to prefer them to the standard way of defining types seen in section \ref{RefSectMizarLayers}.

\subsubsection{Functorial registrations}
\label{RefSectFunctorialRegistrations}
Functorial registrations automatically attach an attribute to all terms presenting a given syntactic form or pattern, once one proves (keyword \verb|coherence| in the snippet below) that terms of that form can be assigned the given attribute.
For example:%
\footnote{\texttt{bool X} is the power set of \texttt{X}. See appendix \ref{RefSectNotations}}.
\begin{verbatim}
registration
  let X be set;
  cluster (bool X) \ X -> non empty for set;
  coherence
  proof
    ...
  end;  
end;
\end{verbatim}
Note that the term in the example above contains two nested functors; there are no limitations on the syntactical complexity of a term being applied a functorial registration.
This kind of registration will have a fundamental role in doing sequent calculus in \M{} (section \ref{RefRuleClusters}) and in implementing custom \M{} \A{}s (section \ref{RefSectAutomations}).

\subsubsection{\Adj{} registrations}
\label{RefSectAdjRegistrations}
\Adj{} registrations works in a way similar to functorial registrations: the \adj{} on the right of the keyword \verb|->| gets automatically attached to a term (which must have the type appearing on the right of the keyword \verb|for|) based on the condition expressed by the matter on the left of that special symbol. 
What is different is how this condition works: instead of checking that a term has a given shape to apply the automation, now it is applied when a term of a given type possesses a given attribute. 
So this registration has the form
\begin{equation}
\verb|cluster | attribute1 \verb! -> ! attribute2 \verb| for | type.
\end{equation}
Once such a registration is enforced, for any term of type $type$ one has that if the checker knows this term enjoys $attribute1$, the checker also knows this term enjoys $attribute2$.
Note that, contrary to functorial registrations, the left hand side of \verb|->| can be empty, which means that the checker will attach an attribute to any term of a given type, regardless of the term being applicable a further attribute.
Of course, upon registering, one has to prove the corresponding first order formula
\begin{center}
\verb|for X being| $type$ \verb|st X is| $attribute1$ \verb|holds X is| $attribute2$.
\end{center}
Such proofs has to be enclosed in a \verb|coherence| block immediately following the registration statement.
For example
\begin{verbatim}
registration
  cluster empty -> one-to-one for Function-like (Relation-like set);
  coherence
  proof
    ...
  end;
end;
\end{verbatim}

\subsection{Predicates}
\label{RefSectMizarPredicates}
In many formulations of first-order languages, as in the one seen in chapter \ref{RefSectFormulation}, one has operation symbols (also said function symbols) each operating on terms and yielding a term; correspondingly there are predicate symbols (also said relation symbols) each operating on terms and yielding truth values.
In the same manner, besides functors, which yield terms, \M{} offers predicates, which yield truth values.
Alongside of the basic predicates \verb|in| (the primitive binary relation of ZFC and TG set theories) and \verb|is| (introduced in section \ref{RefSectMizarLayers}), one of the most pervasive relations in set theory is that of inclusion, which we take as an instance to show how \M{} predicates work:

\begin{verbatim}
definition 
 let X,Y be set;
 pred X c= Y means
  for x being set st x in X holds x in Y;
end;
\end{verbatim}

Note that  \M{} does not provide for predicates forms of \A{}s as powerful as those seen in section \ref{RefSectFunctorialRegistrations} for \adj{}s.
For example, the following can be automated
\begin{verbatim}
let X, Y be set; X /\ Y \ X is empty;
\end{verbatim}
while the predicate-based equivalent formula
\begin{verbatim}
let X, Y be set; X /\ Y c= X;
\end{verbatim}
cannot. 
We will detail on such topics in chapter \ref{RefSectTechnical}.


\section{First-order logic in \MML{}}
\label{RefSectLogicInMml}
Inside Mizar Mathematical Library there are at least three strains hosting articles of content suitable for the treatment of first-order logic:

\begin{enumerate}
\item
\label{RefZf}
A series of articles supplying a language apt to describe set theory according to Zermelo-Fraenkel axioms, started with \cite{ZF_LANG}.
\item
\label{RefQc}
A series of articles supplying a general language for first-order logic, started with \cite{QC_LANG1}.
\item
\label{RefUniAlg}
A series of articles supplying terminology and results about universal algebras, started with \cite{UNIALG_1}.
\end{enumerate}

Most of the classical results of first order logic have, during the years, found their way in strain \eqref{RefQc}: building on those articles a fairly equipped gear of formalizations has been created.\\
There are treatments about the most elementary syntactical properties (those of variables and free variables in a formula (\verb|QC_LANG3|), of subformulas (\verb|QC_LANG2|, \verb|QC_LANG4|), of substitution (\verb|CQC_LANG,SUBSTUT1,SUBSTUT2|), of similarity between formulas (\verb|CQC_SIM1|)), which in turn allow for less and less elementary results, regarding: propositional calculus (\verb|PROCAL_1|, \verb|LUKASI_1|), interpretation and satisfiability (\verb|VALUAT_1|), Gentzen-style sequent calculus (\verb|CALCUL_1|, \verb|CALCUL_2|), up to a basic version of G\"odel's completeness theorem (\verb|HENMODEL,GOEDELCP|).

Unfortunately, the coding of the first order language adopted from the very beginning in \cite{QC_LANG1} is somewhat rigid: roughly sketching the situation, strings of first-order language are represented as tuples of couples of natural numbers, with special symbols (quantifiers, connectives, truth symbol) represented by couples in which the first component is a reserved (small) natural.

This inherently prevents treating uncountable languages, which, alas, would be quite the point for developing even the most fundamental results of model theory, starting with \L{}-Skolem and compactness theorems.

What is more, the completeness theorem currently present in \MML{} has some limitations that look hardly removable in the established framework.
For example, it is restricted to equality-lacking languages, while it would be of interest to talk about languages with equality: 
\M{} first-order language itself is furnished with equality, and the option of possibly applying results worked out to \M{} itself is desirable.  

The following is an account of how a fully developed codebase for model theory in \M{} has been laid down, given the considerations above.
They imposed reformulating things from scratch with a hopefully more flexible approach.

This codebase culminates, as a testbed for itself, with formalizations of the fundamental G\"odel's completeness and \L{}-Skolem theorems, restricted to the case of a generic countable language, and has been submitted to \MML{} Library Committee for peer-reviewing; after triple refereeing, it got accepted in \MML{} in January 2011, with the corresponding five articles 
(\cite{fomodel0}, \cite{fomodel1}, \cite{fomodel2}, \cite{fomodel3}, \cite{fomodel4}) 
published on `Formalized Mathematics' in 2011.
A `dynamic{}' (i.e. constantly updated) version of it is accessible at the author's homepage%
\footnote{\wwwMbc{}}
.
More precisely, among the many flavors of \L{}-Skolem theorem, the one checked is
the `downward{}' flavor, like the one stated in \ref{RefThmSkolem}.
Its \M{} statement sounds like:
\begin{verbatim}
for 
  U2 being non empty set, S being Language, 
  X being countable Subset of AllFormulasOf S, 
  I2 being Element of U2-InterpretersOf S st X is I2-satisfied 
ex U1 being countable non empty set, 
   I1 being Element of U1-InterpretersOf S st 
X is I1-satisfied;
\end{verbatim}

Let us report the \M{} statement of satisfiability theorem (compare \ref{RefThmSat4}), too:
\begin{verbatim}
for C being countable Language st 
  X is (C-rules)-consistent & X c= AllFormulasOf C 
    ex U being non empty countable set, 
    I being Element of U-InterpretersOf C st 
      X is I-satisfied;
\end{verbatim}

Finally, the completeness theorem (see \ref{RefThmCompleteness2}) runs thus:
\begin{verbatim}
for C being countable Language, 
phi wff string of C, X being set st 
  X c= AllFormulasOf C & phi is X-implied 
holds 
  phi is X-provable;
\end{verbatim}
\label{RefThmMizarCompleteness}
Note that this last restriction to countable languages is a mere matter of convenience: the whole work was set up to treat an arbitrary language up to \H{}'s theorem (see \ref{RefThmHenkin2}); on the other hand, reducing to the least-cardinality case was desirable in order to have the job done more quickly (under the urge of demonstrating its usability), without having to handle complications related to the axiom of choice and the likes. 

Those theorems are here regarded as significant goals because of their fundamental role in mathematical logic. In particular, the family of \L{}-Skolem theorems have a fruitful interplay with the cardinality of the language, 
which the ability to deal with, as said, was a starting, motivating point for the present work.
Moreover, this latter kind of results seem to be underrepresented in the global repository of mechanically checked mathematics:
the only work sharing the aims of the present which the author is aware of is \cite{harrison1998formalizing}; both the checker and the proof techniques used there are entirely different than what we are going to deploy here, however.
Additionally, that work is subject to the issue, hinted in the introduction, of being stated in a language far from the standard mathematical one.
Finally, this is the only known presentation of several fundamental theorems for model theory and proof theory formalized together and in a coherent, unitary framework.

\section{Organization of the codebase}
\label{RefSectDefense}
With a total of about $700 k$ bytes and $19 k$ lines of \M{} code, this turned out to be a fairly complex project, so care has been constantly taken to orderly arrange the various results according to their scope into five separate \M{} articles, each depending on the previous ones and hosting affine themes:
\begin{itemize}
\item
\verb|FOMODEL0.MIZ| is the receptacle of all results of broader scope stemmed during the various formalizations, with results and registrations about objects already in \MML{} and quite few dependencies. 
\item
\verb|FOMODEL1.MIZ| introduces the type \verb|Language|, the classification of symbols according to their arity and of terms according to their depth, and the functor to extract subterms from a term or an atomic formula. The bulk of syntax (section \ref{RefSectSyntax}) is done here and in next article.
\item
\verb|FOMODEL2.MIZ| (corresponding roughly to sections \ref{RefSectSyntax} and  \ref{RefSectSemantics}) deals with syntax of non atomic formulas and all the semantics by giving the following constructions: 
the definition of an interpretation $I$ relative to a non empty set $U$ (universe), the constructions saying how to evaluate a term in $U$, how to evaluate an atomic formula in $\{0, 1\}$, what can be regarded as a generic wff formula, how to evaluate it in $\{0, 1\}$ according to $I$, and how to evaluate its depth.
In addition, the functor to obtain another interpretation in the same universe $U$ from $I$ by changing the evaluation of a single literal symbol of the language (reassignment), and the definitions of satisfaction and of entailment are given.
\item
\verb|FOMODEL3.MIZ| (mainly mirroring sections \ref{RefSectSemantics} and \ref{RefSectQuotients} ) supplies a toolkit of constructions to work with languages and interpretations, and results relating them:
the free interpretation of a language, having as a universe the set of terms of the language itself, is defined; 
the quotient of an interpretation with respect to an equivalence relation is built, and shown to remain an interpretation when the relation respects it.
Both the concepts of quotient and of respecting relation are defined in broadest terms, with respect to objects as general as possible.
This is arguably the most `technical{}' article in the tier.
\item
\verb|FOMODEL4.MIZ| (reflecting material from sections \ref{RefSectEqRel}, \ref{RefSectCompatibility},  \ref{RefSectHenkinModel}, \ref{RefSectEnlarge} and \ref{RefSectAllTogether}) introduces the proof-theoretical notions and binds all together.
As a first more general task, it defines what a sequent and a rule are, and what means for a rule to be correct.
Then, using these definitions, it builds the particular set of derivation rules we chose in \ref{RefDefRules}.
Among many other results, satisfiability theorem is proven. 
Finally, restricting to countable languages, completeness and downward \L{}-Skolem are proved.
\end{itemize}
Having sketched the themes dealt with in each article, now the idea is that each formalized result should be placed in the lowest article in which the entities to enunciate it are available, so to give a precise criterion for the arraying of \M{} code among the five articles.

\label{RefModularization}
About one sixth of the code dwells in \verb|FOMODEL0.MIZ|, thus applying to already-defined \M{} entities; the results located there tend to be shorter and more numerous than the lemmas showing up in subsequent articles.
This is a clue of a general separation and modularization design policy pursued across the whole work, aiming at 
\begin{itemize}
\item
stating results in terms of the most general possible \M{} entities;
\item
breaking statements into smaller lemmas, especially if the latter as a result get applicable to a broader class of objects or if the smaller lemmas can be put together in more than a way to get significant theorems. The same applies to definitions.
\end{itemize}

As an example, take the construction of the already discussed Henkin model. 
In \cite{0387908951}, it is introduced just before the proof of the satisfiability theorem, and so, given the rather instrumental nature of its role, its definition is quite condensed.
Here, on the other hand, it has been split into the pair of definitions of free interpretation, \ref{RefDefFreeInt}, and of quotient interpretation, \ref{RefDefInterpretationQuotient}, with a twofold benefit.
First, the former object gets reused to define the term substitution in \ref{RefDefTermSubst}, and hence one of the deduction rules in \ref{RefDefRuleEx}.
On the other hand, the latter applies not only to the former, but to any interpretation.
What's more, the quotient functor is defined more generally as quotient of a relation by a pair of equivalence relations.
Relations are more general than equivalence relations, which are in turn more general than functions, which finally are more general than interpretations, if one call an entity more general than another when the latter is defined in terms of the former.

Accordingly, the various results needed for the Henkin interpretation break into smaller and more general statements, sometimes of interest themselves, or occurring more than once in building further theorems, or maybe just hopefully useful to a possible coder in the future: having stated them in less restrictive terms increases the probability that this will be the case.

This process of separation and modularization may provide a further benefit: in breaking a statement into smaller steps, a fine-grained analysis of which assumptions are needed for each step is encouraged.
This blatantly occurs in chopping down satisfiability theorem: in section~\ref{RefSectHenkin} each step specifies which derivation rules are needed for it to hold (see also section~\ref{RefSectRuleSep}).
Indeed, keeping track of which result traces back to which rules did provide the main guidance in forming our ruleset.
In the sequel, other, more specific occurrences of this attitude will be given: see especially section~\ref{RefConsiderations}.

\label{RefFreeVars}
Here, another facet of this policy is examined: closely related to the just discussed tendency to predicate about as less specialized entities as possible is the choice of encoding formulas in simple strings of symbols.%
\footnote{In the context of \M{} formalizations, we will use the synonyms `string' and `finite sequence' (\texttt{FinSequence}) for the notion of `tuple{}' defined in \ref{RefDefTuple}.}
As for a generic language, this concrete syntax can be opposed by
some representation-agnostic device describing the abstract syntax, in the same spirit of de~Bruijn indexes (\cite{MR0321704}) or parse trees (\cite{MR2319486}, pages 34-36) approaches, which directly model the semantics and thus inherently dispense one from undergoing the twofold labor of first  specifying the syntax rules for well-formedness and then give a way to attach a meaning to each formula.
This is surely a strong plus for them.

\label{RefPlainText}
We maintain that using `plain text{}', as done here, presents advantages, too.
A first advantage is readability: as strings require little assumed knowledge to be understood and have simple notations, the results worked out here are themselves very readable. 
Indeed plain text, concrete syntax is arguably one of the best representations of any data to be read by a human, in most diverse contexts ranging from didactic expositions of formal languages to software design (classical Unix philosophy advocates it as an universal interface, \cite{salus1994quarter}, p.52).
This is of importance especially for a project like \M{} which, besides verifying, also aims at building a library of mathematical knowledge straightforwardly accessible to humans.
\\
Secondly, in the same vein of what has just been discussed, all the results worked out here are likely to produce sub-lemmas of interest to more \M{} coders than if we assume we chose parse trees: indeed, there is a series of \M{} articles supplying the machinery of parse trees in the context of formal languages (\verb|DTCONSTR.MIZ|), and in this assumption, many of the general results in \verb|FOMODEL0.MIZ| would have been in a form available only to the users of that machinery.
This is a two-way phenomenon, of course: the author, using plain sequences instead of parse trees, has been able to take advantage of the massive amount of pre-existing results about the mode \verb|FinSequence|.
As an example of a `by-product{}' of the present formalization which could be of more general interest, and which has been brought out because of the choice of using strings instead of more abstract representations, we pick a result regarding monoids and prefixes (see \eqref{RefUnamb1} in section \ref{RefSectSubTerms}); it is one of the numerous results got by treating sub-terms.

As a last argument supporting our choice, we remark a fundamental quality of our treatment of first order languages notably alleviating one  arguably major drawback typically encountered when using `plain text{}'; that is, the study of \emph{free} occurrences of variables in strings, faced generally when studying the semantics of a previously defined syntax.
In the present framework, one does not even need to \emph{introduce} the concept of free occurrence, because our sequent calculus only demand to watch for simple occurrences of literals inside formulas (rule $\rWitnessA$).
The issues of free occurrences and of substitution are two related hindrances when describing or teaching (see \cite{MR0202575}) a formal language.
They are related because when doing, or formalizing, substitution, attention is to be paid to prevent the capture of free variables: see~\cite{0387908951},~ III.8 for a standard exposition and for the typical complications arising.
\\
In our case, we managed to devise a sequent calculus not needing this concept, and, on the other hand, substitution is resolved using a  novel formalization approach, to the best of author's knowledge, that is, reusing the functors \verb|-freeInterpreter|, \verb|-TermEval| and \verb|ReassignIn|, which sets the scene for the complete disposal of the former notion.
\\
It should be noted that the issue of free occurrences can be arguably regarded as a hindrance, with several papers either devoted to mitigate (or even eliminate) the problem:

\begin{quote}
The relatively complex character of these two [the second being that of term substitution] notions is a source of certain inconveniences of both practical and theoretical nature~\ldots{}~we shall show in this paper that~\ldots~we can simplify the formalization in such a way that the use  of the notions discussed proves to be considerably reduced or even entirely eliminated~\ldots
\\\mbox{}\hfill~\cite{MR0202575},
\end{quote}
or merely devoted to treat the problem; to limit ourselves to \MML{}:
\verb|QC_LANG3|, \verb|QC_LANG2|, \verb|QC_LANG4|, \verb|CQC_LANG|, \verb|SUBSTUT1|, \verb|SUBSTUT2|, \verb|CQC_SIM1|.

The argument above does not imply, of course, that introducing the concept of free occurrence of a variable in a formula is not worth the toil; 
it just stands as a grant (certified by machine checking) that it is not needed to provide a complete sequent calculus.

\section{Dealing with subterms}
\label{RefSectSubTerms}
In key points of any treatment of first-order logic, one has to extract the subterms of a term or of an atomic formula (see, e.g., \ref{RefDefEvalAtomic} and \ref{RefDefTermSubst}), hence 
the formalization supplies a functor \verb|SubTerms| doing this.
\\
It is used crucially in the definition of \verb|TermEval| and \verb|TruthEval| functors, see section~\ref{RefSemantics}.
Its coding will not be explicitly shown here for space reasons.

Here, we want to discuss how its construction slightly departs from standard treatments.
The task at hand is plain dull: one usually does it recursively starting from literals and iterating through operational symbols, and there is not much room from alternative approaches.
However, since the language is presently constructed in terms of strings and concatenation, we tried to do the job at the more general level of monoids and associative operations.
We discuss briefly the idea, without displaying \M{} code.

Take a monoid $\left( M, \op \right)$. One can easily extend the operation $\op$ to a function $\Op$ taking any finite number of arguments iteratively, for example setting
\begin{align*}
\Op \left( a,b,c \right) := \left( a \op b \right) \op c,
&&
\Op \left( a,b,c,d \right) := \left( \Op \left( a ,b,c  \right) \right) \op d,
\end{align*}
and so on. 
To do this in \M{} we introduced the functor \verb|MultPlace|, which actually takes any binary operation (associativity is not needed yet).
Consider any $X \subseteq M$, and call it \emph{unambiguous} (similarly to \cite{lothaire2002algebraic}, 1.2.1) if the restriction of $\op$ to $X \times M$ is injective:
\begin{align*}
\op \left( x_1, m_1 \right) = \op \left( x_2, m_2 \right) \Rightarrow x_1=x_2, m_1=m_2
&& x_1,x_2 \in X, m_1, m_2 \in M
\end{align*}
Now associativity comes into play for the result:
\begin{align}
\label{RefUnamb1}
\op \text{ associative and } X \text{ unambiguous } \Rightarrow \Op |_{X^n} \text{ is injective } && \forall n \in \mathbb{N},
\end{align}
that is, unambiguity is sort-of preserved for $n$-tuples. 
Now, taking the case $M=S^*$, where $S$ is a language, and taking as $\op$ the concatenation (which is associative), it is easy to show that $\terms{S,0} 
$ is unambiguous; indeed, any one-letter strings subset of a language is unambiguous with respect to concatenation.
Starting from that, and using \eqref{RefUnamb1}, it is easily shown by induction that any $\terms{S,m}$ is unambiguous, too; and finally:
\begin{Thm}
$\terms{S}$ is unambiguous.
\end{Thm}
\begin{proof}
Suppose $t, t' \in \terms{S}$ and $y, y' \in S^*$ are such that $t y = t' y'$. 
Call $m$ the greater among the depths of $t$ and $t'$. 
Since $t, t' \in \terms{S, m}$ and $\terms{S,m}$ is unambiguous, it must be $t=t'$ and $y=y'$.
\end{proof}

This permits defining subterms of a term $t$ as the $n$-tuple of terms
$t_1, \ldots, t_n$ such that:

\begin{align*}
t = \fconc(o, t_1, \ldots, t_n),
\end{align*}
where $o$ is the first operation symbol, of arity $n$, of the string $t$. 
Since we know that $t_1, \ldots, t_n$ all belong to $T_S$, which is unambiguous, we can again apply \eqref{RefUnamb1} to decree their uniqueness, which is the point.
We have discussed the general idea, the exact formulation is contained inside \M{} articles.

\section{Encoding in \M{}}
\label{RefMizarEnc}
In reporting here \M{} formalizations, some minor typographic changes to the original code have been made to accommodate it and make it more readable; 
thus the snippets reported here should not be expected to compile correctly. 
For the real code, please refer to \M{} articles.\\
For a concise reminder of the \M{} notations we will be using, refer to appendix \ref{RefSectNotations}.
An extensive tutorial specific to \M{} is \cite{wiedijk2006writing}, while a systematic, up-to-date user manual is \cite{grabowski2010mizar}.

\subsection{The Language type}
\label{RefLangEnc}
Here the ground mode \verb|Language| we will be talking about all the time is defined; it is the \M{} counterpart of the structure `language' introduced in~\ref{RefDefLanguage}. 
There is good support in \MML{} for finite sequences (articles \verb|FINSEQ_1| through \verb|FINSEQ_8|), so it is natural to identify the strings of  the language we are defining with the finite sequences over its carrier.
The same was done originally in \cite{QC_LANG1}. 
The difference is that there it has been imposed to use exclusively sequences of Kuratowski pairs of natural numbers. 
Moreover, the encoding of special logical symbols is ``hardwired'' into that scheme. 
Then a layer of functors and modes definitions is added to be able to refer to these pairs with more suggestive names instead of using directly the encoding.
\\
However, there is no apparent need to impose preemptively how a first-order language should be encoded into sets, rather it seems more sensible to work only at the level of \M{} types, leaving freedom to choose what actual symbol set to use to the instantiator of the type.
\\
Indeed, we will see that such a rigidity, imposing how to encode even only pieces of the language happens to be troublesome for further development (see page~\pageref{RefTooSimple}).
So let us start by introducing a preparatory type named \verb|Language-like|:

\begin{verbatim}
definition
  struct (ZeroOneStr) Language-like 
  (#carrier->set, ZeroF, OneF->Element of the carrier, 
  adicity->Function of the carrier\{the OneF}, INT#);
end;
\end{verbatim}

In this definition there appears yet another provision of \M{} to cope with types.
\verb|struct| is a ``structured type'', similar in spirit to the ones found in many programming languages (called something like aggregates, records, structures, as appropriate).
It is a concise way to group a finite number of types into one entity which becomes a new type. 
Each entry, or selector, of the new type is denoted by an arbitrary type name. 
In our case, we took a pre-defined (see \verb|STRUCT_0|) structure type, called \verb|ZeroOneStr|, inherited all of its fields and added one more. 
So we end up with a quadruple consisting of an alphabet (the carrier), two distinguished symbols of it, and a arity (adicity) function. 
For brevity, a couple of devices are introduced here: first, \verb|OneF| will serve as our logical connective Nor ($\nor$), and it will turn out convenient not to have the arity defined on it; secondly, we agree that a negative arity will denote a relation symbol, a positive arity an operation symbol, and a zero arity a literal; these two points had been already introduced in section \ref{RefSectExampleLanguage}. 
With this in mind, the following definitions are obvious shorthands:

\label{RefLang1}
\begin{verbatim}
definition
  let S be Language-like;
  func AllSymbolsOf S equals the carrier of S;
  func LettersOf S equals (the adicity of S) " {0};
  func OpSymbolsOf S equals (the adicity of S) " (NAT \ {0});
  func RelSymbolsOf S equals (the adicity of S) " (INT \ NAT);
  func TermSymbolsOf S equals (the adicity of S) " NAT;
  func LowerCompoundersOf S equals 
       (the adicity of S) " (INT \ {0});
  func TheEqSymbOf S equals the ZeroF of S;
  func TheNorSymbOf S equals the OneF of S; 
  func OwnSymbolsOf S equals 
  (the carrier of S)\{the ZeroF of S,the OneF of S};
end;
definition
  let S be Language-like;
  mode Element of S is Element of (AllSymbolsOf S);
  func AtomicFormulaSymbolsOf S equals 
    AllSymbolsOf S\{TheNorSymbOf S};
  func AtomicTermsOf S equals 1-tuples_on (LettersOf S);
end;
\end{verbatim}

This almost suffices to encode any first-order language. We only add a couple of further features we wish to endow our new type with:

\begin{verbatim}
definition
  let S be Language-like;
  attr S is eligible means LettersOf S is infinite & 
  (the adicity of S).(TheEqSymbOf S)=-2; 
end;
\end{verbatim}

These two requests impose to have access to an infinite number of letters (we do not know the length of the terms and formulas we will need to write down), and that the arity of the equality symbol is $-2$, as already discussed in section~\ref{RefSectExampleLanguage}, and as dictated by~\ref{RefDefLanguage}.
This automatically likens equality symbol to any other predicate symbol. 
However, this is true only at this stage of syntax. 
The equality symbol will acquire of course special meaning in evaluation, as discussed in section \ref{RefFree2}.
Finally, \verb|Language| type is:

\begin{verbatim}
definition
  mode Language is eligible (non degenerated Language-like);
end;
\end{verbatim}

\verb|degenerated| is an attribute inherited from the type \verb|ZeroOneStr|, and means that the \verb|ZeroF| and the \verb|OneF| coincide. 
So we are requesting that the equality symbol and the logical connective symbol are distinguishable.
For a more elegant formalization and a purely technical convenience (the deployment of registrations, see section~\ref{RefSectAttributes}), we also translate definitions in \ref{RefLang1} attribute-wise:

\label{RefAttributes}
\begin{verbatim}
definition
  let S be Language-like;
  let s be Element of S;
  attr s is literal means s in LettersOf S;
  attr s is low-compounding means  s in LowerCompoundersOf S;
  attr s is operational means  s in OpSymbolsOf S;
  attr s is relational means s in RelSymbolsOf S;
  attr s is termal means  s in TermSymbolsOf S;
  attr s is own means  s in OwnSymbolsOf S;
  attr s is ofAtomicFormula means s in AtomicFormulaSymbolsOf S;
end;
\end{verbatim}

\subsubsection{Too simple an encoding}
\label{RefTooSimple}
We want to hint at an alternative definition for the  \verb|Language| type, which originally was adopted for its further simplicity, but then deprecated and removed for reasons we will discuss.
It was modeled after the idea that, looking at definition \ref{RefDefLanguage}, there is no reason to separate the concept of a language and its arity, with the latter being able to carry an almost full description of the language itself in ZF.
So, instead of using a higher level, structured type to declare the type \verb|-Language|, initially the code relied on a simpler definition based on the $Function$ type, which is one of the most basic and rich in already-made results inside \MML{}: 

\begin{verbatim}
definition
let f be Function;
attr f is eligible means :DefEli: f"{0} is infinite;
end;
definition
mode lang is eligible INT-valued Function;
end;
definition
let S be lang;
func OwnSymbolsOf S equals dom S;
coherence;
end;
notation
let S be lang;
synonym TheEqSymbOf S for OwnSymbolsOf S;
end;
end;
definition
let S be lang;
func TheNorSymbOf S equals {TheEqSymbOf S};
coherence;
end;
definition
let S be lang;
func AllSymbolsOf S equals 
OwnSymbolsOf S \/ {TheEqSymbOf S} \/ {TheNorSymbOf S};
coherence;
end;
definition
let S be lang;
mode Element of S is Element of AllSymbolsOf S;
end;
\end{verbatim}

This definition presents some nice aspects:
\begin{itemize}
\item
Relying straightforward on \verb|Function| type, the type \verb|lang| presents a terse definition, and, thus and most importantly, carries very little work to show existence of entities: 
it is to be noted that in \M{} one has to prove, in the end, existence of any construct he introduces.

\item
The conditions

\begin{enumerate}
\item
\verb| TheEqSymbOf S <> TheNorSymbOf S| (see request \eqref{RefEq28} of \ref{RefDefLanguage}),
\item
\verb| not TheEqSymbOf S in OwnSymbolsOf S|, and
\item
\verb| not TheNorSymbOf S in OwnSymbolsOf S|
\end{enumerate}

are automatically honored, since Tarski-Grothendieck axioms easily allow to show, respectively:
\begin{enumerate}
\item
\verb| X <> {X}|,
\item
\verb|not X in X|,
\item
\verb|not {X} in X|
\end{enumerate}
for any set \verb|X|.
\end{itemize}

So we have conditions \eqref{RefEq28} and \eqref{RefEq29} of definition \ref{RefDefLanguage} already satisfied, the former automatically and the latter via an explicit, yet posing little difficulties to be existentially proved, attribute \verb|eligible|, thus fulfilling the same tasks of the attribute of the same name in the ultimate \M{} code.
The remaining condition \eqref{RefEq30} in definition \ref{RefDefLanguage} was actually not imposed at all; rather, the arity of the language was successively overlaid with an 
\verb| ar | functor based on it, and which was subsequently used in its place:

\begin{verbatim}
definition
let S be lang, s be Element of S;
attr s is own means :DefOwn: s in OwnSymbolsOf S;
attr s is ofAtomicFormula means s in AtomicFormulaSymbolsOf S;
end;
definition
let S be lang;
let s be ofAtomicFormula Element of S;
func ar s equals 
S.s if s is own
otherwise -2;
coherence;
consistency;
end;
\end{verbatim}

Actually, an utterly similar \verb|ar| functor, 
for the respective \verb|Language| mode,
is still present in current \M{} code and largely preferred to direct invocation of \verb|adicity| function because the former is handier to typewrite and leaves to \M{} the burden of checking its argument having the correct type.
It looks like the original definition of language given above was neater and required less preliminary work, so why has it been replaced by \verb|Language|?
The trouble with this definition becomes apparent when trying to restrict or  extend a language. 
In a handful of key steps along the proof of satisfiability theorem, and of \L{}-Skolem, we needed to apply the following scheme: 
take two languages agreeing on some common symbols (typically because one is the restriction/extension of the other), and apply coincidence lemma on a formula consisting only of some of those symbols to conclude that it is a formula in both languages, 
and that its evaluations in two interpretation of the respective languages coincide.
This kind of reasoning is fundamental in the following points:
\begin{itemize}
\item
In eliminating the demand for $\im {\ari^{-1}} { \left\{ 0 \right\} }$ to be infinite from \ref{RefThmSat1} in proof of \ref{RefThmSat3}. 
In turn, the coincidence lemma occurs twice there, once in the main proof, to pass through \emph{restriction} from an interpretation of $S_N$ to one of $S$, and once in the subproof, to pass through \emph{extension} from an interpretation of $S$ to an interpretation of $S_N$, thus in the opposite verse as before.
\item
In the proof of \ref{RefThmSkolem}, to restrict a generic language to the countable one made by the symbols appearing in a countable set of formulas, suitable to be applied \ref{RefThmSat4}, and in extending it back, to supply the interpretation thus found as the witness for the thesis.
\end{itemize}

Obviously, for the coincidence lemma to work, the special symbols, that is $\equiv$ and $\nor$, of the two languages must coincide (see \ref{RefThmCoincidence}).
This fails to hold in the definition above; indeed, one is \emph{granted} that this will not happen, unless the two languages are the same.
Indeed, explicitly constructing the \M{} representations of $\equiv$ and $\nor$ from a given language is a form of the rigid ``hardwiring{}'' we wanted to depart from, as explained in motivating our work: see the beginning of section \ref{RefLangEnc}.

\subsection{Syntax and semantics}
\label{RefSyntax}
The main objects introduced in this section are the three functors \verb|-termsOfMaxDepth|,  \verb|-formulasOfMaxDepth|, \verb|-TruthEval| and the type \verb|Interpreter|.
They are the counterparts of the entities presented in \ref{RefDefTerms}, \ref{RefDefWffs}, \ref{RefDefEvalCompound} and \ref{RefDefInterpretation}, respectively, and have the fundamental roles of describing the sets of terms and formulas of a given (or smaller) depth, of defining what is an interpretation, and of evaluating a term or a formula in a given interpretation.
For the sake of convenience, let us introduce a dedicated type for the generic S-string:

\begin{verbatim}
definition
  let S be Language;
  mode string of S is Element of ((AllSymbolsOf S)*\{{}});
end;
\end{verbatim}

The present construction will be split in stages: first atomic terms (already introduced in \ref{RefLang1}), then terms inductively, and finally atomic formulas. 
Let us start with an auxiliary function performing the basic construction for polish notation, that is, appending an n-tuple of strings to a leading symbol according to its arity:

\begin{verbatim}
definition
  let S be Language,s be ofAtomicFormula Element of S; 
  let Strings be set;
  func ar(s) -> Element of INT equals (the adicity of S).s;
  func Compound(s,Strings) -> Subset of (AllSymbolsOf S)*\{{}} 
  equals 
    {<*s*> ^ ((S-multiCat).StringTuple) where 
    StringTuple is Element of (AllSymbolsOf S)**:
    rng StringTuple c= Strings & 
    StringTuple is (abs(ar(s)))-long};
end;
\end{verbatim}

Here, \verb|S-multiCat| is a dedicated function which concatenates tuples of strings, and renders the mapping $\fconc{}$ introduced on page \pageref{RefDefMultiCat}. 
Roughly speaking, it is the finite iteration of the functor \verb|^|. 
Now recursive construction of terms is straightforward:

\begin{verbatim}
definition
  let S be Language;
  func S-termsOfMaxDepth -> 
  Function of NAT,bool((AllSymbolsOf S)*\{{}}) 
  means dom it=NAT & it.0 = (AtomicTermsOf S) & for n being Nat 
    holds it.(n+1) = (union {Compound(s,it.n) 
    where s is ofAtomicFormula Element of S:s is operational}
    ) \/ it.n;
  func AllTermsOf S equals union rng (S-termsOfMaxDepth);
end;
\end{verbatim}

Again, let us rephrase above definitions in terms of attributes:

\begin{verbatim}
definition
  let m be Nat, S be Language, w be string of S;
  attr w is m-termal means  w in S-termsOfMaxDepth.m;
  let w be string of S;
  attr w is termal means w in AllTermsOf S;
  attr w is atomic means 
    ex s being relational Element of S, 
    V being abs(ar(s))-long Element of (AllTermsOf S)* st 
      w=<*s*>^(S-multiCat.V);
end;
\end{verbatim}

\subsection[Saving work in doing semantics]{Saving work: completing syntax and doing semantics, concurrently}
\label{RefSemantics}
Definitions in \ref{RefSyntax} are the \M{} version of definitions up to \ref{RefDefTerms}.
Now, instead of proceeding with the syntax of non-atomic formulas, we digress to start concurrently putting forth some building blocks of semantics. 
We will then be able to define both syntax and semantics of non-atomic formulas in one shot, taking advantage of the fact that, in contrast to the building of terms, the compounders to derive higher-level formulas from lower-level ones are fixed and well-known. 
The fact of having reduced them to just two types (that is, one logical connective and one existential quantifier) will ease the job.
This strategy saves a good deal of work for our purpose.
First, we start with defining what is an interpretation of a \verb|Language| \verb|S| in a non empty set \verb|U| (standing for universe). 
The definition is similar to the one given in \cite{0387908951}; only, since we don't make distinction between 0-arity compounders (constants) and variables symbols, the distinction made there between interpretation, structure and assignment vanishes too. 
Besides, we separate the universe from the interpretation (the corresponding type is called \verb|Interpreter|; in informal talking we will use both words), more precisely, we make the latter a type dependent on the former.
Here, too, we proceed gradually:

\label{RefInterpreter1}
\begin{verbatim}
definition
  let S be Language, U be non empty set, 
  s be ofAtomicFormula Element of S;
  mode Interpreter of s, U -> 
  Function of (abs(ar(s)))-tuples_on U, U\/BOOLEAN means 
    it is Function of (abs(ar(s)))-tuples_on U, BOOLEAN 
  if s is relational otherwise 
    it is Function of (abs(ar(s)))-tuples_on U, U;
end;
\end{verbatim}

It is worth noting that in case of a literal ($0$-arity) symbol \verb|s|, the interpreter of s,U reduces to a function from $\emptyset$ into an element of $U$.
So, the assignment of a literal, instead of being directly a constant of $u$ of $U$, is rendered as a function \verb|{{}} --> u|, see \ref{RefRemInterpretationOfLiteral}. 
This is convenient for reducing the cases in subsequent proofs and definitions from three (positive, negative and zero arity) to two (negative and non negative arity). 
Now the definition of an interpreter of the whole alphabet is straightforward:

\begin{verbatim}
definition
  let S be Language, U be non empty set;
  mode Interpreter of S, U -> Function means 
  for s being own Element of S holds 
    it.s is Interpreter of s, U;
end;
definition
  let S be Language, U be non empty set, f be Function;
  attr f is (S,U)-interpreter-like means 
  f is Interpreter of S,U & f is Function-yielding;
:: Function-yielding not fundamental; 
:: added for technical convenience
end;
definition
  let S be Language, U be non empty set;
  func U-InterpretersOf S equals {f where f is 
  Element of Funcs(OwnSymbolsOf S, PFuncs(U*,U\/BOOLEAN)): 
  f is (S,U)-interpreter-like};
end;
\end{verbatim}

Before going on we introduce two further constructs: the first is the standard \M{} functor (\verb|FUNCT_4:def 1|) \verb|+*| which `pastes' two function \verb|f| and \verb|g| into a function \verb|f +* g| defined on the union of their domains, with \verb|g| (the right term) prevailing in case of conflicts: a generalization of it to relations was introduced in \ref{RefDefPaste}.

The second is the functor \verb|ReassignIn| which implements the operator changing the assignment of a single literal in a given interpretation, defined in \ref{RefNotationReassign} 
and examined thoroughly in section  \ref{RefConsiderations}.

Now, building a functor \verb|I-AtomicEval phi| yielding the truth value of the \emph{atomic} formula \verb|phi| in the interpretation \verb|I| is standard practice, and the corresponding code is omitted here.
As anticipated, we rather want to indulge on the interpretation of non atomic formulas. 
Usually, one has to do first a recursive definition of the set of wffs, then another recursive definition to evaluate a wff in a given interpretation.
The idea here is to do both in one single recursive definition.
This technically can be done by having, as an object of the recursive definition, a partial function, here called \verb|F| provisionally for brevity, such that, for any natural \verb|mm|, \verb|F.mm|
\begin{itemize}
\item
It has as a domain exactly the cartesian product of \verb|U-InterpretersOf S| with the set of wff of depth not exceeding \verb|mm|.
\item
On that domain it maps a pair (interpretation, string) into the right truth value. 
\end{itemize}
We are thus working on a higher level, where also the interpreter \verb|I| is a variable which gets evaluated together with a wff to return a truth value; only \verb|L| and \verb|U| are fixed parameters. 
For this reason, we first need a tedious but necessary step to transform \verb|I-AtomicEval phi| from a functor into a \emph{function} of \verb|I| and \verb|phi|, named 
\verb|S-TruthEval U| (its name is regretfully not too descriptive):
\label{RefAtomicEval}
\begin{verbatim}
definition
  let S,U;
  func S-TruthEval(U) -> Function of 
  [: U-InterpretersOf S, AtomicFormulasOf S :], BOOLEAN 
  means for I being Element of U-InterpretersOf S, 
  phi being Element of AtomicFormulasOf S holds 
    it.(I,phi)=I-AtomicEval(phi);
end;
\end{verbatim}

For the same reason, in \M{} code the name of the functor \verb|F| contains only \verb|S| and \verb|U|, and is \verb|(S,U)-TruthEval|; so we can get the expected behaviour for it via the fundamental definition:

\begin{verbatim}
definition
  let S be Language, U be non empty set;
  func (S,U)-TruthEval -> Function of NAT, PFuncs
  ([:U-InterpretersOf S, (AllSymbolsOf S)*\{{}}:], BOOLEAN) 
  means it.0=S-TruthEval(U) & for mm being Element of NAT holds 
  it.(mm+1)=G(it.mm) +* it.mm;
end;
\end{verbatim}

At each step the partial function \verb|(S,U)-TruthEval.mm|, which applied to the generic pair \verb|[:I, phi:]| yields a defined, and correct, truth value if and only if \verb|phi| is of depth not exceeding \verb|mm|, is extended by the operator \verb|G|, which of course must yield a partial function of domain extended to the wffs of depth \verb|mm+1|.
So the task is now the construction of \verb|G|. 
We divide the problem in two simpler parts, taking care respectively of the existential symbol and of the NOR symbol separately, so that \verb|G(it.mm)| in the actual \M{} definition is written as
\begin{verbatim}
ExIterator(it.mm) +* NorIterator(it.mm)
\end{verbatim}
Let us illustrate only the construction of \verb|ExIterator g| alone: the idea behind the other half is the same. Here \verb|g| is a generic, appropriate \verb|PartFunc|.
We said that \verb|ExIterator| has to take care simultaneously that the \verb|PartFunc| it returns has both the right domain and the right output on it, based on \verb|g|.
This does not mean that we cannot further divide the problem into simpler parts: the 
definition of \verb|ExIterator g| will actually specify only the correct domain, delegating the evaluation to yet another functor \verb|-ExFunctor|:
\begin{verbatim}
definition
  let S be Language, U be non empty set; 
  let g be Element of PFuncs
  ([:U-InterpretersOf S, (AllSymbolsOf S)*\{{}}:], BOOLEAN); 
  func ExIterator(g) -> PartFunc of 
  [:U-InterpretersOf S, (AllSymbolsOf S)*\{{}}:],BOOLEAN means
    (for x being Element of U-InterpretersOf S, 
    y being Element of (AllSymbolsOf S)*\{{}} holds  
    ([x,y] in dom it iff (
    ex v being literal Element of S, w being string of S st 
    [x,w] in dom g & y=<*v*>^w
     ))) &
  (for x being Element of U-InterpretersOf S, 
  y being Element of (AllSymbolsOf S)*\{{}} st [x,y] in dom it 
    holds it.(x,y)=g-ExFunctor(x,y));
end;
\end{verbatim}
We have indented the part of definition which actually does something (i.e. the specification of the domain, as we were just saying); it does that something quite trivially, too.
Also trivial is the action of the functor \verb|-ExFunctor(x,y)| to which we delegated the semantical part:
\begin{verbatim}
definition
  let S be Language, U be non empty set, f be PartFunc of 
  [:U-InterpretersOf S, (AllSymbolsOf S)*\{{}}:], BOOLEAN; 
  let I be Element of U-InterpretersOf S; 
  let phi be Element of(AllSymbolsOf S)*\{{}};
  func f-ExFunctor(I,phi) -> Element of BOOLEAN equals 
  TRUE if ex u being Element of U, v being literal Element of S 
    st (phi.1=v & f.((v,u) ReassignIn I, phi/^1)=TRUE) 
  otherwise FALSE;
end;
\end{verbatim}
Just notice that this functor is expected to be accurate only when yielding \verb|TRUE|, since otherwise it could yield \verb|FALSE| when actually it is supposed to be undefined.
This is not a problem anymore, since the previous definition already took care of that matter.

Now the significant part of the work is done: all the syntactical and semantical knowledge is thus stored in \verb|(S,U)-TruthEval|, we just may want to rearrange it in a more accessible way, a task with which we end this section.
First, we can go back to the lower level and get a function of just the string we want to evaluate:

\label{RefTruthEval2}
\begin{verbatim}
definition
  let S be Language, U be non empty set, m be Nat;
  let I be Element of U-InterpretersOf S;
  func (I,m)-TruthEval -> 
  Element of PFuncs((AllSymbolsOf S)*\{{}},BOOLEAN) 
  equals (curry ((S,U)-TruthEval.m)).I;
end;
\end{verbatim}

Information about both syntax and semantics is now carried by \verb|(I,m)-TruthEval| in respectively its domain and its return value, so:

\label{RefDepth}
\begin{verbatim}
definition
  let S be Language, m be Nat, w be string of S;
  func S-formulasOfMaxDepth m -> 
  Subset of ((AllSymbolsOf S)*\{{}}) means 
  for U being non empty set,
  I being Element of U-InterpretersOf S holds 
    it=dom (I,m)-TruthEval;
  attr w is m-wff means  w in S-formulasOfMaxDepth m;
  attr w is wff means ex m st w is m-wff;
  func AllFormulasOf S equals 
  {x where x is string of S: ex m st x is m-wff};
end;
\end{verbatim}
\begin{verbatim}
definition
  let S be Language, U be non empty set; 
  let I be Element of U-InterpretersOf S, w be wff string of S;
  func I-TruthEval w -> Element of BOOLEAN means 
  for m being Nat st w is m-wff holds it=((I,m)-TruthEval).w;
end;
\end{verbatim}
Here only the independence of \verb|dom (I,m)-TruthEval| on \verb|I| and \verb|U| needs to be shown to finally be able to evaluate the truth value of a wff formula, which is omitted here.
Let us end this part with stating the remaining semantical definitions implied in the statement of \L{}-Skolem and completeness theorems, both traditionally indicated by the double turnstile $\vDash$; the satisfaction relation (cmp. \ref{RefDefSatisfaction}):
\begin{verbatim}
definition
  let U be non empty set, S be Language; 
  let I be Element of U-InterpretersOf S; let X be set;
  attr X is I-satisfied means 
  for phi being wff string of S st phi in X holds 
    I-TruthEval phi=1;
end;
\end{verbatim}
and the logical implication (entailment, cmp. \ref{RefDefEntailment}):
\begin{verbatim}
definition
  let X be set, S be Language, phi be wff string of S;
  attr phi is X-implied means 
  for U being non empty set, 
  I being Element of U-InterpretersOf S st 
  X is I-satisfied holds I-TruthEval phi=1;
end;
\end{verbatim}

\subsection{Free interpretation}
\label{RefFree1}
The free interpreter of a given operational symbol $s$ of arity $n$ of a Language $S$ is the operation on the set of $n$-tuples of terms of $S$ obtained by concatenating the tuple and appending it to the symbol $s$. 
Obviously the result is again an element of the set of all terms of $S$,
which now acts as a universe and makes this operation an interpreter as of \ref{RefInterpreter1}.

If we add to the picture an arbitrary set $X$ of formulas of $S$ we can talk also of the free interpreter of a relational symbols $r$ of $S$, of arity $-n \in \mathbb{Z}^-$. 
In this case an $n$-tuple of terms  is evaluated TRUE if and only if the atomic formula obtained by concatenating and appending to $r$ (the same job done in previous case) belongs to $X$. 

\label{RefFree2}
\begin{verbatim}
definition
  let X be set, S be Language; 
  let s be ofAtomicFormula Element of S;
  func X-freeInterpreter(s) -> Interpreter of s,(AllTermsOf S) 
    equals s-compound |(abs(ar(s))-tuples_on(AllTermsOf S)) 
  if not s is relational otherwise 
    chi(X,AtomicFormulasOf S) * 
    (s-compound | (abs(ar(s))-tuples_on (AllTermsOf S)));
end;
\end{verbatim}

It is worth noting that this definition is also applicable to the equality symbol. 
This does not matter since, for \emph{any} interpreter, the evaluation of any $\equiv$ atomic formula is overridden at the level of the definition of \verb|-TruthEval| to give the correct value.
This is indeed what is meant when talking about a language with equality.
The functor \verb|-compound| appearing above is introduced to aid the typing and has a trivial definition (see \ref{RefSyntax} for \verb|-multiCat|):

\begin{verbatim}
definition
  let S be Language, s be Element of S;
  func s-compound -> Function of ((AllSymbolsOf S)*\{{}})*, 
  (AllSymbolsOf S)*\{{}} means for V being Element of 
  ((AllSymbolsOf S)*\{{}})* holds it.V = <*s*>^(S-multiCat.V);
end;
\end{verbatim}

\begin{sloppypar}
And finally here is the free interpretation over all the symbols of S, 
with \verb|AllTermsOf S| as universe.
\end{sloppypar}

\begin{verbatim}
definition
  let S be Language, X be set;
  func (S,X)-freeInterpreter -> 
  Element of (AllTermsOf S)-InterpretersOf S means 
  dom it=OwnSymbolsOf S & for s being own Element of S holds
    it.s=X-freeInterpreter(s);
end;
\end{verbatim}

\label{RefSeqCal}


\subsection{Justification of ruleset choice}
\label{RefSectRuleSep}
The complete ruleset appearing in the statement of \ref{RefThmCompleteness2} has formed as a result of the process of \M{}ing completeness theorem.
This means that, as the proof of the latter is staged into a string of roughly escalating results, each rule has been gradually introduced when the previously introduced ones no longer sufficed to proceed.
This way, a tight bound between each intermediate result and the corresponding needed subset of rules have been established, and consequently a hierarchy among rules have been established; for example:
\begin{enumerate}
\item
\label{Ref051}
rules $\rEqRefl, \rEqSymm, \rEqTrans$ are needed for 
\termeq{D}{X} to be an equivalence relation (see~\ref{RefThmTermeqEqrel}),
\item
\label{Ref052}
$\rEqSymm, \rEqRefl, \rFunc, \rRel $ are needed for it to be compatible with $\freeInt{X}$ (see~\ref{RefThmTermeqHenkinCompatible2}), so that
\item
\label{Ref053}
rules $\rEqRefl, \rEqSymm, \rEqTrans, \rFunc, \rRel $
are needed to merely \emph{define} the Henkin interpretation,
\item
\label{Ref054}
rules $\rAssumption, \rEqRefl, \rEqSymm, \rEqTrans, \rFunc, \rRel $ are needed for this interpretation to be a model of the \emph{atomic} formulas of $X$ (\ref{RefThmHenkin1}), 
\item
\label{Ref055}
rule $\rEx$ permits extension of result as from point \eqref{Ref054} to existential formulas like $v \varphi$, while
\item
\label{Ref056}
rule $\rNor$ permits to extend point \eqref{Ref054} to non-existential, non-atomic formulas like $\nor \varphi_1 \varphi_2$.
\item
\label{Ref057}
Since the extension as from points \eqref{Ref055} and \eqref{Ref056} pertain to a witnessed and expanded theory, we use only rules 
$\rThin, \rCut, \rWitnessA, \rEqRefl$
to complete a theory with witnesses, and
\item
we use only rules $\rThin, \rCut $ 
to expand a theory into a closed one, so that
\item
the ruleset appearing in satisfiability theorem's statement, \ref{RefThmSat2}, are exactly the one needed to prove it.
\item
Rule $\rIntuitionisticNightmare$ has to be added to the remaining only to prove non-negative formulas entailed by a consistent theory (\ref{RefThmCompleteness2}).
\end{enumerate}
Rules can thus be precisely tiered according to their functional role during the various proofs.

Moreover, each single subjunction of a new rule in such stepped enlargement of the ruleset was made trying to comply with secondary criteria
such as simplicity and minimality: 
axioms (that is, rules with no input sequents) have been preferred over rules having one, and, even more, over rules having two premisses; 
rules involving atomic formulas have been preferred over rules involving non-atomic formulas.

\label{RefFeedback}
Some rules (in particular $\rFunc$ and $\rRel$), besides complying with the above ideas, are also more formalization-friendly than the ones initially conceived (see \cite{caminati2009yet}), so that how to formalize back-influenced what to formalize, a phenomenon occurred several times along the realization of the whole project. 
Instead of the one-way dynamics (from human to machine) one could expect when starting digging into formalization, this turned into  a sort of feedback leading the human to rethink and rephrase along the way what he is formalizing. 
Every time this happened, the final outcome was always tidier and neater than the initial idea; some reflections on this facet of formalization are in section~\ref{RefFormalizationAndKnowledge}.

Admittedly, $\rFunc$ and $\rRel$ are a bit clumsy to write down, but their proof-theoretical weakness turned out to be quite helpful in easing formalization.\\
Anyway, writing derivation rules in the manner above is like drawing diagrams, in that their goal is to communicate to another human how the rule works; 
what matters is the formalizability, and maybe the computability (which is likely to be good if the former is), so we should not worry about the appearance of those two rules.


Given the guiding ideas according to which we formed our ruleset, and for the reasons exposed in section~\ref{RefPlainText}, it is therefore natural to wonder  whether we can dispense from these notions, or if we can provide simplified versions of them.
We could not help using the notion of term substitution in $\rEx$; 
however, the form of $\rWitnessA$ presents two notable simplifications:

\begin{itemize}
\item
Only the trivial literal-with-literal form of substitution (simple substitution, \ref{RefDefSimpleSubst}) appears.
\item
There is no request on the freeness of the occurrence of the substituted letter.
\end{itemize}

\subsection{Sequents and rules}
\label{RefSectDefSep}
We first define what sequents are in just a plain way:

\begin{verbatim}
definition
  let S be Language; func S-sequents equals 
  {[antecedent,succedent] where 
  antecedent is Subset of AllFormulasOf S, 
  succedent is wff string of S: antecedent is finite};
end;
\end{verbatim}

Only observe that \verb|antecedent| is an (unsorted) finite set, not a $n$-tuple or a bag.

Since the common way of representing sequent derivation rules, as already noticed, has more the nature of a diagram rather than that of a precise formulation, encoding them has presented a number of fundamental design choices.
When starting from scratch, as in this case, one should put an effort in laying down a structure with enough flexibility and generality to last in time and possibly be reused for other purposes.

The first decision regarded modularization: the framework specifying what a rule is and its general properties has been separated from the description itself of the single rule \emph{and} from the definition of derivability. 
\MML{} presents at least two further formalizations of a proof system: 
see definitions of \verb|is_a_correct_step_wrt| inside \verb|CQC_THE1| and of \verb|is_a_correct_step| inside \verb|CALCUL_1|.
Both adopt a monolithic, less articulated approach, simply hardcoding inside the definition itself the possible cases admitted by each single calculus rule via \M{} \verb|if| statements.
A proof is deemed correct if each step of it is correct according to the above cluster of cases.
This is arguably another instance of rigidity in a basic definition, like the one we complained about in justifying the introduction of a new encoding of language (see section \ref{RefLangEnc}).

Here are some benefits brought by our modular approach:

\begin{itemize}
\item
Definitions are terse and readable, compared with other approaches 
like those of \verb|CALCUL_1| and \verb|CQC_THE1|, see below.
\item
The effect of allowing or forbidding the use of a rule can be studied. Indeed, here for each result proved the single rules needed are resolved.
\item
Possible expansion upon this schemes would be feasible; e.g. for applying logic flavors other than classical one.
\end{itemize}

So we first define a framework in which to deal with rules by specifying an abstract \verb|Rule| type as done in \ref{RefDefRule}:
\label{RefRuleMizDef}
\begin{verbatim}
definition
  let S be Language;
  mode Rule of S is 
    Element of Funcs (bool (S-sequents), bool (S-sequents));
  mode RuleSet of S is 
    Subset of Funcs (bool (S-sequents), bool (S-sequents));
end;    
\end{verbatim}

One should think of a \verb|Rule| as the function mapping a set $X$ of sequents into 
the set of all sequents obtainable by applying the rule to all the sequents in $X$
.

Having to do generally with deductions using several rules in succession, we introduce the functor \verb|OneStep| to specify all the sequents derivable from some starting sequents using only one rule of a given \verb|RuleSet D|, as in \ref{RefDefOneStep}.

\begin{verbatim}
definition
  let D be RuleSet of S;
  func OneStep(D) -> Rule of S means
  dom it = bool (S-sequents) & 
  for Seqs being set st Seqs in dom it holds 
    it.Seqs = union ((union D) .: {Seqs});
end;
\end{verbatim}

With that, we have started specifying how to pass from rules to derivations, and the next definition will complete the job. 
Sequent calculus separates the concepts of formal derivability and of provability, so we have two distinct, corresponding attributes as well;
the first (to be compared with \ref{RefDefDerivability}) is applied to a sequent and certifies it to be derivable from an initial set of sequents, while the second (see \ref{RefDefProvable}) applies to a formula and witnesses it is the tail of a sequent derivable from no assumptions and whose premises are given: 

\label{RefDefDerivableMizar}
\begin{verbatim}
definition
  let S be Language, D be RuleSet of S, Seqs1, Seqs2 be set;
  attr Seqs2 is (Seqs1,D)-derivable means 
    Seqs2 c= union (((OneStep D) [*]) .: {Seqs1});
  let X,phi be set;
  attr phi is (X,D)-provable means  
    ex seqt being set st 
      (seqt`1 c= X & seqt`2 = phi & {seqt} is ({},D)-derivable)
end;    
\end{verbatim}

Note how the passage from \verb|OneStep| to derivability leverages some most general constructs as \verb|union|, \verb|[*]| and \verb|.:| (cfr appendix \eqref{RefSectNotations} for their standard notation equivalents). 
This would have not been possible without having detached the notion of rule from that of provability.
Had not we done that, we probably would have ended up to setting some dedicated construction to describe a derivation, including in it an in-line (and verbose) condition of correctness, as it happens in \verb|CQC_THE1| (see definitions of \verb|Proof_Step_Kinds| and \verb|is_a_correct_step_wrt|) 
and in \verb|CALCUL_1| (see the definition of \verb|is_a_correct_step|).
This latter kind of formalizations is not likely to bring any formalization useful outside of its scope and seems much harder to work with.
It seems arguable, however, that the original choice of rigidly encoding the language (see \ref{RefLangEnc}) encourages rigidity as in the constructs just cited.
On the other hand, as stressed in other circumstances, our approach leads to possibly useful by-products of general interest regarding the general objects occurring in definitions: see section \ref{RefConsiderations}.

Now we want to actually code the rules given in section \ref{RefDefRules} in this framework.
The difficulties in encoding a general definition of derivation rule arise from how they are customarily represented; that is, in a diagrammatic form leveraging on the excellent pattern-matching capabilities of the human reader.
These diagrams operatively represent the mechanics of a rule by representing how formulas, or parts of formulas, get altered when passing from the input to the output of a rule. 
Usually the manipulations thus represented are limited to string concatenations and substitutions, and are possibly `decorated{}' with side-conditions (typically regarding the demand of some literal not occurring free inside some formulas occurring in sequents).
In other proof checkers (e.g. Isabelle and HOL variants in general, see section 1.2 of \cite{MR2547330}) there is stronger support for computations and automation, which is just what we would need here (as done in \cite{MR2540934} with Isabelle).

In \M{}, however, there is just set theory: we have therefore to express a rule in this language; one does not have a provision to compute a function, one can just describe a function by encoding its graph in set theory.
Similarly, we cannot compute a rule as its diagram suggest; instead, we must set-theoretically describe what sequents it can associate to a given set of sequents.
This is why the type \verb|Rule| has been defined as from \M{} code above.
With such an approach, doing even most elementary derivations becomes extremely tiresome: every single rule application must be validated by formally checking it satisfies the corresponding \M{} predicate (see section \ref{RefRulePred}).
With no other provision to do sequent calculus, any subsequent \M{} formalization would probably have been much tougher.
Luckily, we will find out a scheme to overlay raw rule definitions with a much more friendly calculus based on \M{}'s functorial registrations: see section \ref{RefRuleClusters}.
On the other hand, even without this overlay, this merely descriptive method presents at least one advantage over the computational method:

\begin{quotation}
The disadvantage is that there is no explicit encoding of a derivation. The
derivation is kept implicitly by the proof-assistant and we cannot manipulate
its structure. \cite{MR2540934}
\end{quotation}
We, on the contrary, have full control on a derivation: indeed each derivation will be hand-crafted into single rule application steps.

\subsection{How to define a single specific rule}
\label{RefRulePred}
A slight nuisance we have to face preliminarily is given by the fact that the symbol set of \M{} is pure ASCII, which forced to translate the names of the rules introduced in \ref{RefDefRules} and elsewhere into plain text, as from the following table
\begin{center}
\begin{tabular}{|c|c|}\hline
\verb|Rule0|
   &
      $\rAssumption$
\\ \hline
      \verb|Rule1|
   &
      $ \rThin $
\\ \hline
      \verb|Rule2|
   &
      $ \rEqRefl $
\\ \hline
      \verb|Rule3a|
   &
      $ \rEqTrans $
\\ \hline
      \verb|Rule3b|
   &
      $ \rEqSymm $
\\ \hline
      \verb|Rule3d|
   &
      $ \rFunc $
\\ \hline
      \verb|Rule3e|
   &
      $ \rRel $
\\ \hline
      \verb|Rule4|
   &
      $ \rEx $
\\ \hline
      \verb|Rule5|
   & 
      $ \rWitnessA $
\\ \hline
      \verb|RuleNor|
   &
      $ \rNor $
\\ \hline
      \verb|Rule8|
   &
      $ \rCut $
\\ \hline
      \verb|Rule9|
   &
      $ \rIntuitionisticNightmare $
\\ \hline
\end{tabular}
\end{center}
We try to separate the jobs of typing from that of actually specifying how a rule works, by proceeding in stages. 

First we specify the core of the rules as \M{} predicates (which were introduced in section \ref{RefSectMizarPredicates}); compare this with their definition \ref{RefDefRules} and with their customary representation of page~\pageref{RefDefRuleDiagrams}:

\begin{verbatim}
definition
let Seqts be set; let S be Language; let seqt be S-null set;

pred seqt Rule0 Seqts means seqt`2 in seqt`1;

pred seqt Rule1 Seqts means ex y being set st y in Seqts &
y`1 c= seqt`1 & seqt`2 = y`2;

pred seqt Rule2 Seqts means seqt`1 is empty & 
ex t being termal string of S st 
  seqt`2 = <* TheEqSymbOf S *> ^ t ^ t;

pred seqt Rule3a Seqts means
ex t1,t2,t3 being termal string of S, x being set st
(seqt=[{<*TheEqSymbOf S*>^t1^t2,<*TheEqSymbOf S*>^t2^t3}, 
<*TheEqSymbOf S*>^t1^t3]);

pred seqt Rule3b Seqts means
ex t1,t2 being termal string of S st 
  seqt`1 = {<*TheEqSymbOf S*>^t1^t2} &
  seqt`2 = <*TheEqSymbOf S*>^t2^t1;

pred seqt Rule3d Seqts means
ex s being low-compounding Element of S, 
T,U being (abs(ar(s)))-element Element of (AllTermsOf S)* st 
  (s is operational & seqt`1=
  {<*TheEqSymbOf S*>^(TT.j)^(UU.j) where 
  j is Element of Seg abs(ar(s)),  
  TT,UU is Function of Seg abs(ar(s)), (AllSymbolsOf S)*\{{}} 
  : TT=T & UU=U} 
  & seqt`2=<*TheEqSymbOf S*>^(s-compound(T))^(s-compound(U)));

pred seqt Rule3e Seqts means 
ex s being relational Element of S, 
T,U being (abs(ar(s)))-element Element of (AllTermsOf S)* st 
  (seqt`1={s-compound(T)} \/
  {<*TheEqSymbOf S*>^(TT.j)^(UU.j) where 
  j is Element of Seg abs(ar(s)),  
  TT,UU is Function of Seg abs(ar(s)), (AllSymbolsOf S)*\{{}} 
  : TT=T & UU=U} 
  & seqt`2=s-compound(U));

pred seqt Rule4 Seqts means  
ex l being literal Element of S, 
phi being wff string of S, 
t being termal string of S st 
  seqt`1={(l,t) SubstIn phi} & seqt`2=<*l*>^phi;

pred seqt Rule5 Seqts means ex v1,v2 being 
(literal Element of S), x being set, p being FinSequence st 
seqt`1=x \/ {<*v1*>^p} & v2 is (x\/{p}\/{seqt`2})-absent &
[x\/{(v1 SubstWith v2).p},seqt`2] in Seqts;

pred seqt RuleNor Seqts means
ex phi1, phi2, phi3, phi4 being wff string of S st seqt=
[{<*TheNorSymbOf S*>^phi1^phi2, <*TheNorSymbOf S*>^phi3^phi4},
				<*TheNorSymbOf S*>^phi2^phi3];

pred seqt Rule8 Seqts means
ex y1,y2 being set, phi,phi1 being wff string of S st 
y1 in Seqts & y2 in Seqts & y1`1=y2`1 & y1`2=phi1 & 
y2`2 = <* TheNorSymbOf S *> ^ phi1 ^ phi1 & 
seqt`1\/{phi}=y1`1 & seqt`2=<*TheNorSymbOf S*>^phi^phi;

pred seqt Rule9 Seqts means  
ex y being set, phi being wff string of S st 
y in Seqts & seqt`2=phi & y`1=seqt`1 & y`2=xnot (xnot phi);
end;  
\end{verbatim}

In the definiens of last rule we took advantage, for a matter of convenience, of the \M{} analog of the map seen in \ref{RefDefNot}:
\begin{verbatim}
definition
let S be Language, w be string of S;
func xnot w -> string of S equals <*TheNorSymbOf S*>^w^w;
end;
\end{verbatim}




We want at this stage to reduce at a minimum the role of types, to concentrate on the mechanics of the rule, so we declare the starting sequents, represented by \verb|Seqts|, as an untyped variable (a set); at the same time, to do the  correct typing later, we need to preserve a link to the type of the specific language \verb|S| we are referring to, so we introduce a fake attribute \verb|-null|, and save it in the variable \verb|seqt|, which represents the derived sequent (the ``denominator{}'') of the rule.

Now we pass from the predicate \verb|RuleX| to a rule as specified by \verb|Rule| type; let us take \verb|Rule0| for example:
\begin{verbatim}
definition
  let S be Language, 
  R be Relation of bool (S-sequents), S-sequents;
  func FuncRule(R) -> Rule of S means 
  for inseqs being set st inseqs in bool (S-sequents) holds 
    it.inseqs=
    {x where x is Element of S-sequents:[inseqs,x] in R};
end;
registration
  let S be Language;
  cluster -> S-null Element of S-sequents;
end;
definition
  let S be Language;
  func P0(S) -> Relation of bool (S-sequents), S-sequents 
  means for Seqts being Element of bool (S-sequents), 
  seqt being Element of (S-sequents) holds 
    [ Seqts, seqt ] in it iff seqt Rule0 Seqts;
end;
\end{verbatim}

\begin{verbatim}
definition
let S be Language;
func R0(S) -> Rule of S equals FuncRule(P0(S));
end;
\end{verbatim}

When having to code many rules this scheme is convenient because one needs only to define a \M{} predicate without much worrying about typing; 
afterwards, the rule is easily, and standardly, converted into a \verb|Relation| and finally applied \verb|FuncRule|.
The last couple of definitions have to be manually repeated verbatim inside \M{} code, only changing \verb|P0(S)| to \verb|P1(S)| and \verb|R0(S)| to \verb|R1(S)| (and so on for each rule\ldots), 
because \M{} lacks second-order definitions.
The code contains the proofs of soundness and monotonicity for all the rules above.
We warn the reader that in it, the attribute \verb|isotone| is used, since the keyword \verb|monotone| was already in use.

\subsection[Derivation rules as \M{} registrations]{Exploiting \M{}'s functorial registrations to restore a sequent calculus}
\label{RefRuleClusters}
As discussed earlier, there is only one other proof checker in which a sequent calculus has been encoded, to the best of author's knowledge:
Isabelle (or variants, 
\cite{dawson2010generic}, \cite{MR2540934}, \cite{chapman2010mechanising}), 
probably due to some nice facilities provided, as inductive definitions and structured proofs (\cite{nipkow2003structured}).
\M{} has fewer provisions to actually calculate things apart from small integer  arithmetics; 
thus, the idea is to exploit its  functorial registrations (see section \ref{RefSectFunctorialRegistrations}), which actually do some pattern matching on a term of the first order language of \M{}: 
we can try to employ this capability to recognize whether a sequent is derivable from another using a given rule.
Once finished, we will have adapted \M{}'s powerful registrations to gain back some resemblance to a calculus, lost with the purely descriptive \emph{definition} of derivation rules in the set theory of \M{} (given in section \ref{RefSectDefSep}) as opposed to their computational \emph{application} possible in Isabelle.

Preliminarily, however, we need to make more precise the definition of \verb|-derivable| attribute: 
in that definition, derivability is assessed first taking all sequents derivable from an initial set of sequents using one rule of $D$, and exactly once (\verb|OneStep D|).
The sequents derivable from a fixed initial set of sequents are those obtainable by iterating the scheme above a finite number of times, that is its transitive closure (\verb|[*]|).
Now we want to be able to resolve that finite number of times, by defining, in parallel with \ref{RefDefDerivability}:

\begin{verbatim}
definition
let S be Language, D be RuleSet of S, m be Nat;
func (m,D)-derivables -> Rule of S equals iter(OneStep D,m);
end;
\end{verbatim}

and

\label{RefDefDerivableNMizar}
\begin{verbatim}
definition
let m be Nat, S be Language, D be RuleSet of S; 
let Seqts,seqt be set;
attr seqt is (m,Seqts,D)-derivable means 
seqt in (m,D)-derivables.Seqts;
end;
\end{verbatim}

This at first looked straightforward, since it seemed sufficient to replace the transitive closure operator with the iteration operator: 
we have constantly advocated the use of as general objects as possible also as good practice in such situations.
Indeed, it turned out to be sufficient, the only shame being that no ready-made result connecting those two operators existed in \MML{} strong enough to be useful in this case.
As we insistently maintained, however, there is a good side also in this worst case, that is: some additional work had to be done, but there is good chance somebody else will use it in the future.
The general result we obtained is reported in section \ref{RefIterResult}.
Here, it permits:

\begin{verbatim}
Lm18: union (((OneStep D)[*]).:{X}) = union 
{(mm,D)-derivables.X where mm is Element of NAT: 
not contradiction};
\end{verbatim}

and finally, the redefinition:

\begin{verbatim}
definition
let S be Language, D be RuleSet of S; let X,x be set;
redefine attr x is (X,D)-provable means 
ex H being set, m st H c= X & [H,x] is (m,{},D)-derivable;
\end{verbatim}
\label{RefDefMizarProvable}

The redefinition above allows to exhibit derivations (and hence proofs) in single steps, and allow finally to render most of our derivation rules as functorial registrations (which were introduced in section \ref{RefSectFunctorialRegistrations}).

\begin{verbatim}

definition
let x be set; let S be Language;
attr x is S-premises-like means 
x c= AllFormulasOf S & x is finite;
end;

registration
let S be Language; let H1, H2 be S-premises-like set;
let l, l1 be literal Element of S;
let phi, phi1, phi2 be wff string of S;
let t, t1, t2 be termal string of S;
cluster [Phi \/ {phi}, phi] -> (1,{},{R0(S)})-derivable set;

cluster [H1\/H2, phi] -> (1,{[H1,phi]},{R1(S)})-derivable set;

cluster {[{},<*TheEqSymbOf S*>^t^t]} -> {R2(S)}-derivable set;

cluster
[{<*TheEqSymbOf S*>^t^t1, 
<*TheEqSymbOf S*>^t1^t2}, <*TheEqSymbOf S*>^t^t2] 
-> (1,{},{R3a(S)})-derivable set;

cluster [{(l,t) SubstIn phi}, <*l*>^phi] -> 
  (1,{},{R4(S)})-derivable set;

let l2 be (H\/{phi1}\/{phi2})-absent literal Element of S;
cluster [(H\/{<*l1*>^phi1}) null l2, phi2] -> 
(1,{[H\/{(l1,l2)-SymbolSubstIn phi1},phi2]},{R5(S)})-derivable 
set;

cluster [{<*TheNorSymbOf s*>^phi1^phi1, <*TheNorSymbOfs*>^phi2^phi2},
<*TheNorSymbOf s*>^phi1^phi2] -> 
(1,{},{RNor(S)})-derivable set;

cluster 
[{<*TheNorSymbOf S*>^phi1^phi2}, <*TheNorSymbOf S*>^phi2^phi1] 
-> (1,{},{RNor(S)})-derivable set;

cluster [H null (phi1^phi2),xnot phi] -> (1,
{[H\/{phi},phi1],[H\/{phi},<*TheNorSymbOf S*>^phi1^phi2]}, 
{R8(S)})-derivable set;

cluster [H, phi] null 1 -> 
(1,{[H, xnot (xnot phi)]},{RD(S)})-derivable set;

end;
\end{verbatim}

Please see section \ref{RefSectAutomations} for remarks on the \verb|null| functor, which ignores the operands on its right and serves merely syntactical, technical purposes connected with some \M{} idiosyncrasies.

Combining the one-step derivations above
, one can perform standard multi-step derivations; 
additionally, if some particular multi-step derivation is found to occur recurrently, one can of course register it in turn into a composite, macro-like derivation (often called derived rule); for example, the following registration might be handy: 

\begin{verbatim}
registration
let S be Language, t be termal string of S;
let phi be wff string of S;
cluster [{phi}, <*TheEqSymbOf S*>^t^t] -> 
(2, {}, {R1(S),R2(S)})-derivable set;
end;
\end{verbatim}
Once he has a decent set of clustered rules, 
one can perform a derivation in a very natural manner, close to a standard derivation of sequent calculus, especially combining them together
, which is essential in calculations, permitting to transitively concatenate derivations, and moreover keeping precise track of their \emph{depth}: the latter results stowed in the first argument of the \verb|-derivable| attribute at the end of the derivation chain.

Here is a sample taken from \verb|FOMODEL4| and rendering a simplest chained derivation:

\begin{verbatim}
[H1\/H2, phi] is (n+1,{[H1, phi]},{R1(S)})-derivable &
[(H1\/H2)\/(H1\/H2),phi] is 
(1,{[H1\/H2,phi]},{R1(S)})-derivable; then
[H1\/H2,phi] is 
(n+1+1,{[H1,phi]},{R1(S)}\/{R1(S)})-derivable by Lm28;
\end{verbatim}
The lastly derived sequent's attribute always stores the depth of the respective derivation, in this case \verb|n+2|.
Notice that invariably, when combining at least two rules to perform multi-step derivations or to obtain a derived rule, one needs monotonicity (see definition \ref{RefDefMonotone2}), which 
accounts for the invoking of \verb|Lm28| above.

Clearly, our original predicate-based definitions of rules, given in section \ref{RefRulePred}, are much more  obnoxious to deal with than this device exploiting \M{} clusters, and serve only to validate the latter, being doomed to disuse after that. 

\subsection{Definitions for readability}
Tinkering with rulesets, as we did by weighing the exact needed rules in statements of results from chapter \ref{RefSectFormulation}, is not a common practice.
Usually, the ruleset is fixed in advance, with everything thereafter meant relative to that unique ruleset.
As a reward, statement of theorems result terser.
We of course can regain back that same advantage by introducing shorthand \M{} definitions, which make possible to state completeness theorem in the concise form seen on page~\pageref{RefThmMizarCompleteness}.

\begin{verbatim}
definition
let S be Language;
func S-rules -> RuleSet of S equals 
{R0(S), R1(S), R2(S), R3a(S), R3b(S), R3d(S), R3e(S), R4(S)} \/ 
{R5(S), RNor(S), R8(S)};
coherence;
end;
\end{verbatim}
\begin{verbatim}
definition
let X be set, S be Language, phi be wff string of S;
attr phi is X-provable means 
phi is (X,{R9(S)}\/S-rules)-provable;
end;
\end{verbatim}

These can be regarded as placeholders, introduced to make theorem statements more mainstream, so that a casual reader will better grasp an idea of what a theorem deals with upon reading it.
This is important for \MML{}, which aims to supply a library of mathematics being human-readable, besides being machine-verified{}.

As a side-note, we observe that the keyword \verb|-provable| now results overloaded to denote two distinct attributes (compare definition above with that on page \pageref{RefDefMizarProvable}).
\M{} has no problem with that, being able to resolve which use is being made by looking at the number of the arguments accompanying the identifier (the \emph{format}\index{\M{} identifier format}); in case this is not sufficient, it looks at both the number of arguments and at their type (the \emph{pattern}\index{\M{} identifier pattern}).

\chapter[Technical aspects of the formalization]{The formalization from a technical point of view}
\label{RefSectTechnical}
This chapter provides techniques and practical considerations, pertaining the practice of writing \M{} code and formalizations in general, accrued while working with the system.
It features material from \cite{CaminatiJar2011}.

\section{Custom \A{}s in \M{}}
\label{RefSectAutomations}
When writing a \M{} formalization, a significant amount of the user's time usually goes into browsing the Mizar Mathematical Library (\MML) for those results that he needs and that are already proved.
Here a few techniques to reduce this time are illustrated.
Let us begin by pointing out two shortcomings related to the \M{} verifier, which was introduced in section \ref{RefSectMizarOverview}:

\begin{enumerate}
\item
\label{RefItemLowLevel}
At a low level, a \M{} user has no practical way to specify the logic
the \M verifier 
applies to approve an inference: 
no full programmability is provided, besides tweaking the source code, to plug in alternative proof systems.
\item
\label{RefItemHighLevel}
At a higher level, there is no general provision to instruct the verifier to `know{}' a generic custom-defined formula already proved, in order to 
avoid to  list explicitly some, or all, of the labels following the keyword \verb|by| when the writer perceives the inference as obvious, natural, or recurring so often to deserve some kind of automation.
\\
For example, one might want to program the verifier to `know{}' the trivial set-theoretical inclusion
\begin{align}
\label{RefEqSampleAutomation}
X \cap Y \C X,
\end{align}
so as not to have to `\verb|by|' the corresponding \MML theorem in reasonings involving it.
\end{enumerate}
We will not discuss the reasons and implications of these design choices: considerations on such topics can be found in \cite{urban2006mizarmode}.
Rather, we will focus on how certain \M features can be exploited to mitigate issue \ref{RefItemHighLevel},
which is relevant to a user from a purely practical point of view: 
it is frequently the case that the user knows the steps to lay down a proof, or the statements of the needed theorems  (especially when trivial or natural) and then must go and dig into the vastities  of the \MML to justify each of them. 
While this can turn out to be a highly instructive experience, it also leads to distraction and to longer formalization times, and urged the creation of a range of tools to aid the user in facing this task (\cite{rudnickiescape}, \cite{bancerek2004integrated}, \cite{urban2006momm}, \cite{cairns2007integrating}). 
Here, a different, possibly complementary, approach is proposed aiming instead at reducing the occasions when he faces such a task.

Ideally, to a generic inference submitted to the verifier, one or more finite sets can be associated, each made of premisses strictly needed for the inference to be accepted (the references one \emph{must} list following the keyword \verb|by|).

We adopt the term \label{RefDefAutomation} \emph{\A} to loosely indicate any device or mechanism enabling to reduce such a set, even if possibly only for some kinds of inferences.

First of all, it must be said that indeed \M{} does supply some \A{}s natively. 
However, they present several constraints: they are not strong enough to instruct the verifier to blindly accept \emph{any} already proved formula.
To be more precise, the \A{}s called \emph{requirements}, imported using the eponymous  keyword, are powerful enough to do exactly this, which is what we fancied of in item \eqref{RefItemHighLevel} of the above list.
The point is that requirements are out of reach of most users, because they are hard-coded in verifier's sources by developers (\cite{naumowicz2004improving}, \cite{naumowiczevaluating}).
The remaining \M{} provisions (see section \ref{RefSectMizarOverview}) to introduce \A{}s are less general, and mostly embedded in its type system; 
however, they are the building blocks of the methods we will see.  
%
%

\subsection{Type clustering to avoid redefinitions}

Let us return to the example \A in \eqref{RefEqSampleAutomation}: we would like to teach the verifier that 
\begin{align*}
X \cap Y \C X.
\end{align*}
A first naive way to do that would be to redefine the output type of the functor \verb|/\|. This can be done for whatever functor via the keyword \verb|redefine|, subject of course to the appropriate proof.
This process of `type recasting{}', however, is destructive: only the last (re)definition is retained by the verifier. 
And indeed, \MML already provides 
(in article
\verb|SUBSET_1|) yet another redefinition of \verb|/\|:

\begin{verbatim}
definition
  let E, X be set; let A be Subset of E;
  redefine func A /\ X -> Subset of E;
  coherence
  proof
    ...
  end;
end;
\end{verbatim}
which we do not want to lose.
The idea then is to combine the ability of \M{} to recognize \emph{one} type for a given term with the identification scheme seen at the end of section \ref{RefSectMizarLayers}, to `funnel{}' several recognized types into a single term as a result.
Following an example taken, as others in the sequel, from \cite{fomodel0} we introduce a dummy functor symbol, a `shadow{}' of the main functor symbol \verb|/\|, let us call it \verb|typed/\|:

\begin{verbatim}
definition
  let X,Y be set;
  func X typed/\ Y -> Subset of X equals X /\ Y;
  coherence;
end;
\end{verbatim}
Now, if we make \M{} identify (see section \ref{RefSectMizarLayers}) \verb|X typed/\ Y| with \verb|X /\ Y|:

\begin{verbatim}
registration
  let X,Y be set;
  identify X /\ Y with X typed/\ Y;
  compatibility;
  identify X typed/\ Y with X /\ Y;  
  compatibility;
end;
\end{verbatim}
then the two distinct typing we wanted do simultaneously co-exist:

\begin{verbatim}
now
  let Z be set; let X, Y be Subset of Z;
  X/\Y is Subset of Z; :: thanks to redefinition in article SUBSET_1
  X/\Y is Subset of X; :: thanks to registration above
end;
\end{verbatim}
The verifier accepts both the  formulas above without justification. 
What happens is clear: the term \verb|X/\Y| occurring in last formula is identified with \verb|X typed/\ Y|, which has the right type, convincing the verifier.
A couple of musings:
\begin{itemize}
\item
Generally, when employing the \verb|identify| registration, we always do it in both verses, as above. This is to be on the safe side, as \verb|identify| works in a not completely symmetrical manner (\cite{grabowski2010mizar}, section 2.7).
As observed in practice, the second identification in such cases always comes for free; that is, once the \verb|compatibility| condition for the first one is secured, the second \verb|compatibility| statement is validated without proof, even without starting a new \verb|registration|~$\ldots$~\verb|end;| block. 
Hence, not requiring much additional time, it is useful to do double identification each time. 
In subsequent examples we sometimes will omit transcribing the second identification, though.
\item
There is already an \A granting \verb|X /\ Y = Y /\ X| without justification (this is achieved via so-called \emph{properties}, more on which can be found in \cite{grabowski2010mizar}, section 2.5). 
Thence, one could expect he has obtained for free also the \A{} \verb|X /\ Y is Subset of Y|, via the ideal chain: 
\begin{verbatim}
X /\ Y = Y /\ X = Y typed/\ X.
\end{verbatim}
This will not work straightaway, however. There are two possibilities:
\begin{enumerate}
\item
\label{RefItemSymmetry}
Introduce a further identification between \verb|X typed/\ Y| and \verb|Y typed/\ X|.
\item
Introduce a further functor \verb|/\typed| working symmetrically with respect to \verb|typed/\|:
\begin{verbatim}
definition
  let X,Y be set;
  func X /\typed Y -> Subset of Y equals X/\Y;
  coherence;
end;
\end{verbatim}
and then proceed with the suitable registrations.
\end{enumerate}
Both approaches solve the problem providing the automation
\begin{verbatim}
X /\ Y is Subset of Y;
\end{verbatim}
As a passing note, method \eqref{RefItemSymmetry} above suggests that identifications
may replace 
properties in some circumstances: \M{} can be made aware of the commutativity of a given functor either via properties (as done in \MML for \verb|/\|) or by identifying a functor application with the application obtained by swapping its arguments.
It would be interesting to know to what extent these two approaches are equivalent.
One simple remark is that the latter has wider applicability: upon establishing commutativity property when defining \verb|typed/\|, one gets the error:\\
\verb|The result type is not invariant under swapping the arguments|,\\
while an identification does the job.
\end{itemize}

\subsection{Type clustering with dummy arguments: combining type clustering with notations}
\label{RefSectTypeClustering}
We would like to repeat the scheme above for the (trivial) set-theoretical property
\begin{align*}
Y \C X \Rightarrow X \cap Y = Y.
\end{align*}
Here, however, we face a limitation of the \verb|identify| construct we have not mentioned yet: there are formal restrictions on the functors being identified.
In particular, they must have the same number of arguments, so we cannot just write:
\begin{verbatim}
registration
  let X be set, Y be Subset of X;
  identify X /\ Y with Y;
\end{verbatim}
We just introduce a functor \verb|null| whose only (for the time being) utility is formally to take a second argument for the mere sake of balancing things:
\begin{verbatim}
definition
  let X,Y be set;
  func X null Y equals X;
  coherence;
end;

registration
  let X be set; let Y be Subset of X;
  identify X /\ Y with Y null X;
  compatibility by XBOOLE_1:28;
  identify Y null X with X /\ Y;
  compatibility;
end;
\end{verbatim}

The final effect is not as neat as that of section \ref{RefSectTypeClustering}, in that we cannot submit the verifier simply
\begin{verbatim}
let X be set, Y be Subset of X;
X/\Y = Y;
\end{verbatim}
This is because the verifier of course cannot guess that writing \verb|Y| we mean \verb|Y null X|: although the argument \verb|X| is semantically thrown away by \verb|null|, its presence supplies information.
Indeed, \M{} can understand things the other way round:

\begin{verbatim}
let X be set, Y be Subset of X;
X /\ Y = Y null X; then
X /\ Y = Y;
\end{verbatim}
This works.%
\footnote{\texttt{then} can replace \texttt{by} when referring to the immediately preceding formula.}
Again, we have some remarks:
\begin{itemize}
\item
The last inference works because the definition of \verb|null| is done via \verb|equals| rather than via \verb|means| (see item \eqref{RefItemEquals} on page \pageref{RefItemEquals}): the corresponding definition being a macro permits to take advantage of \M{}'s \emph{equals expansion}, 
see section 2.3.4 of \cite{grabowski2010mizar}.
Note that, in order to take advantage of equals expansion for a given functor outside the file in which it is defined, that file must be imported via the \verb|definitions| directive.
\item
As we said before, the aim of \A{}s is to reduce the time devoted to searching \MML, rather than to save keypresses. 
So this scheme is still arguably worth being applied: no \verb|by| is needed.
\end{itemize}
The following sort of a dual of the previous registration:

\begin{verbatim}
registration
  let X be set; let Y be Subset of X;
  identify X \/ Y with X null Y;
  compatibility by XBOOLE_1:12;
  identify X null Y with X \/ Y;
  compatibility;
end;
\end{verbatim}
permits

\begin{verbatim}
let X; let Y be Subset of X;
X \/ Y = X null Y; then X \/ Y = X;
\end{verbatim}

\subsection{Combining dummy arguments and type clustering}

The dummy argument of the functor \verb|null| can be more than a placeholder to satisfy \verb|identify|'s requirements. 
It can be used to control the desired type of a term. 
For example, we could redefine \verb|X null Y| to be a \verb|Subset of X\/Y|, and then be able to automate properties like:

\begin{verbatim}
let X, Y be set;
X null Y is Subset of X \/ Y; then X is Subset of X \/ Y;
\end{verbatim}
However, one can do better: recall that type redefinitions are destructive, while we might want in the future \verb|null| not to have that type.
It is natural then to resort to type clustering, just seen in section \ref{RefSectTypeClustering}; for example:

\begin{verbatim}
definition
  let X, Y be set;
  func X \typed/ Y -> Subset of X \/ Y equals X;
  coherence by XBOOLE_1:7;
end;
\end{verbatim}

\begin{verbatim}
registration
  let X, Y be set;
  identify X \typed/ Y with X null Y;
  compatibility;
  identify X null Y with X \typed/ Y;
  compatibility;
end;
\end{verbatim}
and the wanted automation is in charge.

\subsection{Reference redirection via functorial registrations}
\label{RefSectReferenceRedir}
Since functorial registration, seen in section \ref{RefSectFunctorialRegistrations}, are so powerful, the idea is to reduce the most used first-order relation symbols to attributes in order to save lookups into \MML{}.
\subsubsection{Translating set-theoretical equality, \texttt{=}, via attribute \texttt{empty}}
Let us start with the \M{} equality symbol, \verb|=|. 
It can be rendered via the functor \verb|\+\|%
\footnote{\texttt{\textbackslash+\textbackslash} is the set-theoretical symmetric difference, commonly denoted as $\Delta$: $X \Delta Y = X \sdiff Y \cup (Y \sdiff X)$. See also appendix \ref{RefSectNotations}.}
and the attribute \verb|empty| via the result (FOMODEL0:29):

\begin{verbatim}
for X, Y being set holds X \+\ Y is empty iff X=Y;
\end{verbatim}
This means that for every theorem in \MML{} whose statement has the form

\begin{equation}
\label{RefEqEqualForm}
\verb|B1: term1 = term 2;|
\end{equation}
%
%
one can produce a translation like

\begin{equation}
\label{RefEqAttributeForm}
\verb|term1 \+\ term2 is empty by B1, FOMODEL0:29;|
\end{equation}
This latter version has the advantage of being applicable as a functorial registration, which allows to use it without justification in subsequent proofs. 
Even if one needs the original version of the theorem, one can get it by referring back to \verb|FOMODEL0:29|.
This gives the possibility of remembering just one reference (\verb|FOMODEL0:29|) in place of several references, one for each needed theorem: of course, the more theorems are translated in registrable form \eqref{RefEqAttributeForm}, the more convenient this scheme gets.
As an example, \verb|XBOOLE_1:4| states associativity of \verb|\/|. 
We then register the following:
\begin{verbatim}
registration
  let X, Y, Z be set;
  cluster ((X \/ Y) \/ Z) \+\ ( X \/ (Y \/ Z) ) -> empty for set;
  coherence by XBOOLE_1:4, FOMODEL0:29;
end;
\end{verbatim}
Now, when we need this theorem we write:
\begin{verbatim}
let X,Y,Z be set; ((X\/Y)\/Z) \+\ (X\/(Y\/Z)) is empty; then 
(X\/Y)\/Z = X\/(Y\/Z) by FOMODEL0:29; 
\end{verbatim}
\verb|XBOOLE_1| contains many such elementary results, frequently employed and having form \eqref{RefEqEqualForm}, so it is arguably convenient to turn them into registrations. 
After doing that, each time the user invokes one of them, he will only need to remember at most \verb|FOMODEL0:29|.
Here is a list of some registrations of this kind introduced and deployed in \M{} articles \verb|FOMODEL0-4| (to save space, environments and type declarations are mostly omitted):
\begin{flushleft}
\verb|cluster ([x,y]`1) \+\ x -> empty for set;|\\
\verb|cluster ([x,y]`2) \+\ y -> empty for set;|\\
\verb|cluster (id {x}) \+\ {[x,x]} -> empty for set;|\\
\verb|cluster (x.-->y) \+\ {[x,y]} -> empty for set;|\\
\verb|cluster (id {x}) \+\ (x.-->x) -> empty for set;|\\
\verb|cluster <*x*> \+\ {[1,x]} -> empty for set;|\\
\verb|let p be FinSequence; cluster (<*x*>^p).1 \+\ x -> empty for set;|\\
\verb!let m be Nat;!\\ 
\verb!cluster m-tuples_on X \+\ Funcs(Seg m,X) -> empty for set;!\\
\verb|let f,g be Function;|\\ 
\verb|cluster (f+*g) \+\ (f \ [:dom g, rng f:] \/ g) -> empty for set;|\\
\label{RefPasteCluster}%
\verb!cluster (f+*g) \+\ f|(dom f \ dom g) \/ g -> empty for set;!\\
\verb!cluster (f+*g) \+\ ((f|(dom f) \ (f|(dom g))) \/ g) -> empty for set;!\\
\end{flushleft}

\subsubsection{Translating set-theoretical inclusion, \texttt{c=}, via attribute \texttt{empty}}
A similar translation can be done for the inclusion symbol \verb|c=| into the functor \verb|\| and the attribute \verb|empty| via \verb|XBOOLE_1:37|:
\begin{verbatim}
X \ Y = {} iff X c= Y;
\end{verbatim}
Here are some examples of registrations for this case:
\begin{flushleft}
\verb!cluster {x}\{x,y} -> empty for set;!\\
\verb!cluster NAT\INT -> empty for set;!\\
\verb!let X be set; let F be Subset of bool X;!\\ 
\verb!cluster union F \ X -> empty for set; !\\
\verb!let X,Y be set; let x be Subset of X, y be Subset of Y;!\\ 
\verb!cluster x\Y \ (X\y) -> empty for set;!\\
\verb!let m be Nat; cluster (m-tuples_on X) \ (X*) -> empty for set;!\\
\end{flushleft}

\subsubsection{Translating set-theoretical membership, \texttt{in}, via attribute \texttt{empty}}
The same goes with the rendering of relation symbol \verb|in| via functors \verb|{ }|, \verb|\| and again attribute \verb|empty|, thanks to:
\begin{verbatim}
for x, X being set holds x in X iff {x} \ X is empty;
\end{verbatim}
Also for this scheme we give some examples of registrations:
\begin{flushleft}
\verb!let U be non empty set, u be Element of U;!\\ 
\verb!cluster {(id U).u} \ U -> empty set;!\\
\verb!let m,n be Nat; let p be (m+1+n)-long Element of U*;!\\
\verb!cluster {p.(m+1)} \ U -> empty set;!\\
\end{flushleft}

\subsubsection{Translating basic arithmetics into attributes}
The same idea can be adapted to a broad scope of contexts.
Here, it was exploited when needing some very basic arithmetical identities, like:

\begin{flushleft}
\verb!let z be zero (integer number);!\\ 
\verb!cluster abs(z) -> zero (integer number);!\\
\verb!let z1 be non zero (complex number);!\\
\verb!cluster abs(z1) -> positive (real number);!\\
\verb!let x,y be real number;!\\
\verb!cluster max(x,y)-x -> non negative (real number);!\\
\end{flushleft}

As another application, request \ref{RefEq30} in definition \ref{RefDefLanguage} was translated as follows for easier reference:

\begin{flushleft}
\verb!let S be Language; cluster ar(TheEqSymbOf S) + 2 -> zero number;!\\
\verb!cluster abs(ar(TheEqSymbOf S)) - 2 -> zero number;!\\
\end{flushleft}

Similarly, other trivial arithmetical facts were rendered thus:
\begin{flushleft}
\verb!let v be literal Element of S; cluster ar(v) -> zero number;!\\
\verb!let m0 be zero number; let t be m0-termal string of S;!\\
\verb!cluster Depth t -> zero number;!\\
\verb!let phi0 be m0-wff string of S;!\\
\verb!cluster Depth phi0 -> zero number;!\\
\verb!let m be Nat; let phi be m-wff string of S;!\\
\verb!cluster m - (Depth phi) -> non negative (real number);!\\
\verb!let phi1 be non 0wff (wff string of S);!\\
\verb!cluster Depth phi1 -> non zero Nat;!\\
\end{flushleft}

We omit any further detail; some more examples are in articles \verb|FOMODEL0-4|.

\subsection{Definiens clustering: combining identification and equals expansion}

Consider the last three registrations of section \ref{RefPasteCluster} involving the functor \verb|+*|: recalling the idea of that section, they express three set-theoretical equalities which, as all other equalities of this form, can be used remembering just one \MML reference, \verb|FOMODEL0:29|, once registered.
There is also a way to avoid even the need to refer to this single theorem, and make \M{} accept the corresponding equalities:
\begin{flushleft}
\verb!f \ [:dom g, rng f:] \/ g) = (f +* g);!\\
\verb!f|(dom f \ dom g) \/ g = (f +* g);!\\
\verb!((f|(dom f) \ (f|(dom g))) \/ g) = (f +* g);!
\end{flushleft}
straightaway.
Note that \MML{}'s original definition of \verb|+*| is done via \verb|means|, so equals expansion cannot be used. 
One could redefine \verb|+*| with one of the equalities above, but this would exclude the other two from \A.
Instead, it is possible to keep the original definition and proceed as follows:
\begin{verbatim}
definition
  let P,Q be Relation;
  func P +*1 Q equals P \ [:dom Q, rng P:] \/ Q;
  coherence;
  func P +*2 Q equals P|(dom P \ dom Q) \/ Q;
  coherence;
  func P +*3 Q equals ((P|(dom P) \ (P|(dom Q))) \/ Q);
  coherence;
end;
\end{verbatim}
Note that the shadow functors \verb|+*1|, \verb|+*2|, \verb|+*3| all accept more general arguments than its forefront functor \verb|+*|: every \verb|Function| is a \verb|Relation|, but the opposite does not hold. 
For this reason we first proceed with the mutual identification of the functors defined above:

\begin{verbatim}
registration
  let P, Q be Relation;
  identify P +*1 Q with P +*2 Q;
  compatibility
    proof
    ...
    end; 
  identify P +*2 Q with P +*3 Q;
  compatibility by RELAT_1:109;
end;  
\end{verbatim}
Having done so, \M{} now accepts equalities like:

\begin{verbatim}
let P, Q be Relation; P +*3 Q = P \[:dom Q, rng P:] \/ Q;
\end{verbatim}
This means, in particular, that identifications work transitively: we have identified \verb|+*1| with \verb|+*2| and \verb|+*2| with \verb|+*3|, but not \verb|+*1| with \verb|+*3|.
Finally, we can bind all these identifications with the forefront functor \verb|+*|, and then forget about the others:
\begin{verbatim}
registration
  let f, g be Function;
  identify f +*1 g with f+*g;
  compatibility
    proof
    ...
    end;
  identify f+*g with f +*1 g;
  compatibility;
end;
\end{verbatim}
Now the following works without justifications:

\begin{verbatim}
let f, g be Function;
f+*g = f\[:dom g, rng f:] \/ g;
f+*g = f|(dom f \ dom g) \/ g;
\end{verbatim}
We have thus `clustered' several definientia into the single functor \verb|+*|.

\section{Considerations on some formalization design issues}
\label{RefConsiderations}
Awareness that thoroughly calibrating types when spelling out definitions is a key factor for a well-structured proof grew steadily during the work. 
If one goes too strong, by being too fussy in specifying what type of arguments a functor takes, and at some point faces the need, for example, to apply the same functor to two arguments which differ little, but do not have the same type, in this case he is forced to do double work;
moreover, sometimes a job can be made lighter by adapting an existing type to an affine situation, and base on ready-made formalizations, instead of creating a brand new world of types and having to re-invent the wheel.
On the other hand, being too light with typing one loses the advantages of a tidy formalization given by \M{}. 
As an example, compare the definitions of atomic wff in \cite{QC_LANG1} and in the present work:

\begin{tabular}{p{6cm}|p{6cm}}
\begin{verbatim}
definition
  let F be Element of QC-WFF;
  attr F is atomic means
\end{verbatim}
\vdots
&
\begin{verbatim}
definition
let S be Language; 
let phi be string of S;
attr phi is 0wff means 
\end{verbatim}
\vdots
\end{tabular}

The definition on the right applies to any string, and not to anything less only because inside the body of the definition there are functors requiring a string (a \verb|FinSequence|) as arguments; 
on the other hand the left definition restricts the objects to which atomic attribute can be applied. 
This is likely to complicate forthcoming treatments.
One could object that the first solution has the strength of ensuring that `atomic{}' implies `wff{}'.
But this can be attained also in the second case by clustering (see section \ref{RefSectMizarOverview}), which is indeed done in the formalization:

\begin{verbatim}
registration
let S be Language;
cluster 0-wff -> atomic string of S;
cluster atomic -> 0-wff string of S;
let m be Nat;
cluster m-wff -> wff string of S;
let n be Nat;
cluster (m+0*n)-wff -> (m+n)-wff (string of S);
end;
\end{verbatim}

The heavy adoption of attributes and clusters is a trait of the present formalization%
\footnote{FOMODEL0 is the single registration-richest article in the whole \MML{}, as checked at
\urlMmlquery{} on~\printdate{31.3.2011}}%
.
Their use has a few advantages: first, a technical one, for they permit to automatically and implicitly reach conclusions which otherwise should be made explicit with a \verb|by| statement; this also brings an advantage in terms of terseness and legibility; finally, they make type-trimming easier, allowing rich typing with relative ease. 
 
In the present case, this is especially true for the classification of the various types of alphabet symbols: literal, compounder, relational, etc\ldots (see \ref{RefAttributes}), and for the classification of well-formed tuples, as in the example above.

\label{RefDefinitions}
A further character of this formalization is the effort to find definitions based on \verb|equals| and \verb|is|, avoiding those based on \verb|means| when possible. 
It seems that the former encourage the reusing of pre-existing objects (functors, modes or attributes), at the price of doing the preparatory work of translating the definition to be expressed in terms of those other objects. 
Definitions thus obtained are arguably more neat and readable, although sometimes less immediate. 
For sure ``\verb|equals|'' definitions have a technical advantage resembling that of attributes: they are grasped automatically by \M{} if included in the \verb|definitions| directive, again making life easier and code terser. 
See \cite{kornilowicz2009define}, section 3.
Good examples of this method could be the definitions of the functors 
\verb|===| (not reviewed here, needed in construction of \verb|-TruthEval|), 
\verb|X-freeInterpreter| (see \ref{RefFree2}), \verb|(I,m)-TruthEval| (see \ref{RefTruthEval2}), and \verb|ReassignIn| (see sections \ref{RefSemantics} and \ref{RefSectRuleSep}).

The last example is interesting because it also honors the ideas introduced in section \ref{RefSectDefense}: indeed, besides having a clean, \verb|equals|-based definition, it is first introduced for arguments of more general types than we need for our particular case:

\begin{verbatim}
definition
let x,y be set, f be Function;
func (x,y) ReassignIn f -> Function equals
f +* (x .--> ({} .--> y));
end;
\end{verbatim}

Recalling the action of \verb|+*| functor and how we encoded the interpretation of a literal symbol (section \ref{RefSemantics}), its way of working should be clear. 
We are leaning of course on a definition (\verb|+*|) given elsewhere, but this permits to use more general tools, avoid restating things, reduce the length of the definition, and, above all, reuse possible results already proven about \verb|+*|. 
Even if these results were not already available in \MML{}, proving them for a more general, pre-defined object is always better than providing a specialized result framed in a narrower context: somebody else could take advantage of them for developing possibly different areas of \MML{}.
Again, as in the first example of this section, we adapt this general definition to our needs by showing this functor returns the expected type when applied to the types we will feed it, using the powerful tool of functorial clustering (section  \ref{RefSectMizarOverview}):
\begin{verbatim}
registration
  let S be Language,U be non empty set,
  I be (S,U)-interpreter-like Function;
  let x be literal Element of S, u be Element of U;
  cluster (x,u) ReassignIn I -> (S,U)-interpreter-like;
end;
\end{verbatim}

Indeed, as noted in section \ref{RefSectDefense}, some developments needed in the present work produced results regarding only pre-existing, more general objects: as examples, one could consider the introduction of the \verb|-unambiguous| attribute for generic binary operations, and the related results for the generic monoids, sketched in section \ref{RefSectSubTerms}.
Here, two more examples, taken again from \verb|FOMODEL0| and which were missing from \MML{}, are exhibited in view of their concise and general statement; they both derived from investigations on how to formalize sequent calculus. 

The first regards the transitive closure \verb|R[*]| of a relation \verb|R| and states that it is both transitive and reflexive:

\begin{verbatim}
registration
  let R be Relation;
  cluster R[*] -> transitive Relation;
  cluster R[*] -> reflexive Relation;
end;
\end{verbatim}

The second binds together the transitive closure and the iteration of a function:
\label{RefIterResult}
\begin{verbatim}
for f being Function st rng f c= dom f holds f[*] = union 
  {iter(f,mm) where mm is Element of NAT: not contradiction};
\end{verbatim}

\section[About duplications in \MML{}]{About the specialization of existing results}
\label{RefDuplicates}
In proving \ref{RefThmLindenbaum}, we implicitly employed the following intuitive fact:
\begin{align*}
\left.
\begin{aligned}  
Y \text{ finite } 
\\ 
\forall n \in \N \ X_n \C X_{n+1}
\\
Y \C \bigcup_{n \in \N} X_n
\end{aligned}
\right\}
&& \Rightarrow 
&& \exists 
\overline{n} \in \N \ \st Y \C X_{\overline{n}}   
\end{align*}
Initially, we relied on \verb|HENMODEL:3|, which in turn employs the ad-hoc results  \verb|HENMODEL:1| and \verb|HENMODEL:2|, for a total of more than 250 lines of dedicated \M{} code.
Actually, such specific propositions could have not been written at all, for they are predated by the more general result \verb|COHSP_1:13|:
\begin{verbatim}
for X being non empty set, Y being set st 
X is c=directed & Y c= union X & Y is finite 
ex Z being set st Z in X & Y c= Z;
\end{verbatim}
where \verb|c=directed| substantially means somehow closed with respect to finite union, as from definition \verb|COHSP_1:def 3|:
\begin{verbatim}
definition
  let X be set;
  attr X is c=directed means
  for Y being finite Subset of X ex a being set st 
  union Y c= a & a in X;
end;
\end{verbatim}
Now consider the theorem \verb|COHSP_1:6| coupled with \verb|COHSP_1:13| reported above:

\begin{verbatim}
for X being non empty set st 
(for a,b being set st a in X & b in X
ex c being set st a \/ b c= c & c in X) holds X is c=directed;
\end{verbatim}
Clearly these two results generalize \verb|HENMODEL:3|, which runs like:
\begin{verbatim}
for f being Function of NAT,C, X being finite set st 
(for n,m st m in dom f & n in dom f & n < m holds 
f.n c= f.m) & X c= union rng f 
ex k st X c= f.k,
\end{verbatim}
and whose authors could have saved a fair amount of work by leveraging \verb|COHSP_1:13| and \verb|COHSP_1:6|.
Other instances of duplicated work inside \MML{} were noticed during the work, with this being probably the most blatant.
What is more, the excessive specialization of duplicate results in \verb|HENMODEL| makes their statement inelegant, e.g.,~obfuscating the simple meaning expressed by \verb|COHSP_1:13| with unnecessary objects like \verb|f|, \verb|m|, \verb|n|, \verb|k| appearing in \verb|HENMODEL:3|.
Duplication is a serious issue, because it bloats \MML{}, creates confusion in it, dissipates people's work, while often, like in this case, reusing existing code as much as possible results in more elegant and general formalizations (if the pre-existing code is already elegant and general enough). 
A major cause of this issue is the problematic browsing and mastering of such an extensive corpus like \MML{}.
Various attempts at delivering tools to assist \M{} authors in browsing it have been made (\cite{urban2006momm}, \cite{bancerek2004integrated} and \cite{bancerek2003information}).
Let us note that, in turn, \verb|COHSP_1:13| itself is susceptible of what, in the writer's opinion, are improvements: indeed, in \verb|FOMODEL0|, that same result, indeed stated in a slightly more general form

\begin{verbatim}
for Y being set st Y is c=directed holds 
for X being finite Subset of union Y 
ex y being set st y in Y & X c= y;
\end{verbatim}

is proved by slicing it into six small and general propositions, for an amount of $66$ lines of \M{} code versus the $68$ lines of the original proof.
Obviously the only purpose of this computation is to show that the two proofs are comparably long, what actually matters is the bunch of auxiliary results obtained `for free{}':

\begin{verbatim}
Th60: for X, Y being set st union X c= Y holds X c= bool Y;
\end{verbatim}

\begin{verbatim}
Th61: for X being set holds 
A is_finer_than B & X is_finer_than Y implies 
  A\/X is_finer_than B\/Y;
\end{verbatim}

\begin{verbatim}
Th62: for A, B being set st A is_finer_than B holds 
  A\/B is_finer_than B;
\end{verbatim}

\begin{verbatim}
Th63: for A, B being set st 
B is c=directed & A is_finer_than B holds 
  A\/B is c=directed;
\end{verbatim}

\begin{verbatim}
Th64: for X, Y being set holds 
  INTERSECTION(X,Y) is_finer_than X,
\end{verbatim}
also reverberating on other, even more general, \M{} articles. 
Indeed, \verb|INTERSECTION| and \verb|is_finer_than| are introduced in \verb|SETFAM_1|:

\begin{verbatim}
definition
  let SFX,SFY be set;
  pred SFX is_finer_than SFY means
  for X being set st X in SFX ex Y being set st 
    Y in SFY & X c= Y;
end;
\end{verbatim}

\begin{verbatim}
definition
  let SFX,SFY be set;
  func INTERSECTION (SFX,SFY) means
  for Z being set holds 
  (Z in it iff 
    ex X,Y being set st X in SFX & Y in SFY & Z = X /\ Y);
  existence;
  uniqueness;
end;
\end{verbatim}

This kind of trimming is here regarded as important for \MML{}, for reasons previously discussed in similar cases in which the proof of a given fact led to a string of by-products of independent interest.

\section{Numerically characterizing the formalization}
\label{RefNum}

We want to estimate formalization cost and de~Bruijn factor  (\cite{wiedijkbruijn, asperti2010some, naumowicz2006example}).
\\
There are huge spaces of discretionality, which will be discussed below, in both calculations, so we will make some arbitrary choices, hoping they will result sensible and acceptable.

Two figures are to be estimated in order to trigger calculations: the amount of man hours devoted to formalization and a number measuring the size of a non-formal, human-targeted mathematical text carrying information grossly equivalent to the one formalized.

\subsection{Estimating formalizing time}
\label{RefTime}
A significant amount of work regarded preliminary reformulation (\cite{caminati2009yet}) rather than \M{} formalization, as seen in chapter \ref{RefSectFormulation}.
This portion of work was carried on largely before \M{} formalization even started, however its results were revised `dynamically{}' during the formalization as a result of the `feedback' cited in section \ref{RefFeedback}, and as confirmed by the differences noticeable between \M{} code and \cite{caminati2009yet}.
Thus, formalization time assessment will be affected by some excess due to this auxiliary work subtracting time to effective coding, and to the fact that the workflow was rather irregular and interleaved with idle periods due to extraneous activities; this last issue is probably common to most formalization time estimations. 

With the foregoing cautionary remarks, evolution of the codebase is as follow, using \M{} public repository on author's homepage as a development history record.\\ 
The first \M{} file ever written by the author dates back to~\printdate{24.01.2010}, and, since then, formalization and \M{} learning efforts went on concurrently; the first codebase including G\"odel's completeness theorem was successfully checked on~\printdate{12.10.2010}. 
 
\L{}-Skolem theorem was first successfully compiled on~\printdate{05.11.2010}.
As a conclusion, formalizing time can be estimated in $284$ days.

\subsection[Establishing an equivalent source text]{Establishing a non-formal, equivalent mathematical source text}
\label{RefSpace}
For the reasons exposed in section~\ref{RefTime}, choosing a denominator to compute de~Bruijn factor is not so straightforward in this case. The nearest treatment would obviously be~\cite{caminati2009yet}, which, however, merely highlights the points in the proof which are novel and less trivial, and silently assumes a lot of prerequisites. 
Instead, the low starting point of this formalization demands we choose a more thorough treatment as a fairer reference, with an exposition starting from scratch (alphabets, strings, etc...) as this formalization does, and not omitting the tedious and `trivial{}' details.
Since~\cite{0387908951}, being an undergraduate text book, arguably satisfies these requirements and was the original source of inspiration, it seems a good candidate.
Specifically, we OCRed\footnote{%
Optical character recognition, usually abbreviated to OCR, is the mechanical or electronic translation of scanned images of handwritten, typewritten or printed text into machine-encoded text.
}
its scans and selected the excerpt going from section II.1 (`Alphabets', page 10) through section VI.1 (`The \L{}-Skolem Theorem', ending on page 89), taking the resulting ASCII text as our non-formal source text.
It is available on author's home page for reference.
We have not removed the dispensable bits occurring in this source (exercises, historical notes, examples); first, they can be considered quantitatively negligible for our purposes, especially if one consider how arbitrary the whole matter is; secondarily, if one regards de~Bruijn factor as a fundamental ratio between how much information is needed for a machine to accept statements and how much information is needed for a human to accept the same statements, rather than a totally empirical indicator to practically compare formalization verbosities, he could consider those bits as effectively useful for that human reader to accept (assimilate, he would say) those statements.

\subsection{Results}
The formalization cost is then calculated to be 
\begin{align*}
\frac{\frac{284}{7}}{89-10+1}=0.5 \text{ weeks per page}
\end{align*}
The de~Bruijn factor is shown below:
\begin{center}
\begin{tabular}{ccccc}
      
   &
      informal (bytes)
   &
      formal (bytes)
   &
   de~Bruijn factor
   &
\\ 
\cline{2-5}
      uncompressed
   &
   132495
   &
   710144
   &
   5.4
   &
   apparent
\\ 
      gzipped
   &
      46839 
   &
      153399
   &
      3.3
   &
      intrinsic
      \\
      \cline{2-5}
\end{tabular} 
\end{center}

\section{Formalization can bring insight}
\label{RefFormalizationAndKnowledge}
Various reasons supporting the endeavour of formalizing the body of known mathematics have been given in several expositions.
After doing such an extensive formalization, we would like to explicitly state an often overlooked, though merely potential, one: formalizing a proof can and should increase the amount of information the proof itself brings with it, with respect to the same proof in its `paper{}' version one has when starting mechanizing it.
\\
To elaborate on such a vague assertion, let us give specific cases, annotated with references to the present formalization: 
\begin{itemize}
\item
One is strongly encouraged to variously simplify things to make them digestible by a machine. This is likely to lead to a finer discern about what notions are really needed for a result to hold  or event to be stated.
For example, we note that the notion of consistency was not needed until \H{}'s theorem, \ref{RefThmHenkin2}.
\item
One is strongly encouraged to modularize and reuse. This can possibly bring to previously unknown, or at least not clearly stated, or maybe just obvious but useful in cutting down redundancies, relations between results. This is of particular relevance in case of community-developed, self-referencing repositories such as the \MML{}. See the discussion on page \pageref{RefModularization}. 
\item
Combining the two points above, one could, for example, obtain more, smaller propositions with less/weaker hypotheses, with the possible side effect of an escalation of their total number; as an example take what done in section~\ref{RefDuplicates}.
\item
As for other kinds of computation, a machine can help the human keeping track of a large amount of data, as could be a large number of hypotheses among which a minimal set is to be isolated to make a theorem hold; maybe this set of hypotheses has grown after some application of previous point. 
In our case, we had to filter out what derivation rules were needed corresponding to various lemmas, see section~\ref{RefSectRuleSep}.
\end{itemize}

Of course, the `final user{}' of a theorem is often little interested in this kind of internals; on the other hand, if a theorem is regarded as a particle of \rem{knowledge} information, this collateral, supplementary information \rem{knowledge} pursued in refining it can be deemed some value; 
which indeed happens when dealing with foundational issues, as in, e.g., reverse mathematics.

\appendix

\chapter{Proof of the Substitution Lemma}
\label{RefSectSubstLemma}
\begin{Prop}
\label{RefThmSubstLemmaTerm}
Given an interpretation $i$, a literal $v$ and a term $t$ of the language $S$, 
and given a set $X$, it holds:
\begin{align}
\label{RefEq44}
\restrict {\eval{i} \funccomp \eval{\reassign{v}{t}{\freeInt{X}}}} {{\terms{S}}_{, n}} 
=
\restrict {\eval {\reassign {v}{\eval{i} \left( t \right)}{i}}} {{\terms{S}}_{, n}}
\end{align}
for every $n \in \N$.
\end{Prop}

\begin{proof}
Let $U \neq \emp$ be the universe of $i$, and set $u := \eval{i} \left( t \right) \in U$,
$ I:= \reassign{v}{t}{\freeInt{X}}$.
The proof is by induction on $n$.
First, consider $t_0 \in {\terms{S}}_{, 0}$, and show that 
$ \eval i \left(  \eval I \left( t_0 \right)  \right) =
\eval{ \reassign v u i} \left( t_0 \right)
$ as follows.
Set 
$v_0 := t_0 \left( 0 \right) \in \im {\ari ^ {-1}} {\left\{ 0 \right\}} $ 
and proceed by cases.
\begin{description}
\item
[Case $v_0 = v$]
Then
\begin{align*}
\eval i \left( \eval I \left( t_0 \right) \right)
\overset{\text{\tiny{\ref{RefDefEvalAtomic}}}}{=}
\eval i \left( \left( I \left( v \right) \right) \left( 0 \right) \right)
\overset{\text{\tiny{\ref{RefNotationReassign}}}}{=}
\eval i \left( \left\{ \pair {0} {t} \right\} \left( 0 \right) \right)
=
u
\\
\overset{\text{\tiny{\ref{RefNotationReassign}}}}{=}
\left( \reassign v u i \left( \left\{ \pair{0}{v} \right\} \right)  \right) \left( 0 \right)
\overset{\text{\tiny{\ref{RefDefEvalAtomic}}}}{=}
\eval {\reassign v u i} \left( \left\{ \pair{0}{v} \right\} \right)
=
\eval {\reassign v u i} \left( t_0 \right).
\end{align*}
\item
[Case $v_0 \neq v$]
\begin{align*}
\eval i \left( \eval I \left( t_0 \right) \right)
\overset{\text{\tiny{\ref{RefDefEvalAtomic}}}}{=}
\eval i \left( \left( I \left( v_0 \right) \right) \left( 0 \right) \right)
\overset{\text{\tiny{\ref{RefNotationReassign}}}}{=}
\eval i \left(\left(  \freeInt{X} \left( v_0 \right)  \right)  \left( 0 \right)\right)
\overset{\text{\tiny{\ref{RefDefFreeInt}}}}{=}
\eval i \left(  t_0 \right)
\\
\overset{\text{\tiny{\ref{RefDefEvalAtomic}}}}{=}
\left( i \left( v_0 \right) \right) \left( 0 \right)
\overset{\text{\tiny{\ref{RefNotationReassign}}}}{=}
\left(  \left( \reassign v u i  \right) \left( v_0 \right) \right) \left( 0 \right)
\overset{\text{\tiny{\ref{RefDefEvalAtomic}}}}{=}
\eval {\reassign v u i } \left( t_0 \right).
\end{align*}
\end{description}

Now suppose \eqref{RefEq44} is verified for every $n \leq m$.
Consider $t' \in {\terms{S}}_{, m+1}$.
It will suffice to show
\begin{align}
\label{RefEq42}
\eval{i} \left( \eval{I} \left( t' \right) \right) =
\eval {\reassign {v}{u}{i}} \left( t' \right).
\end{align}
Set $s := t' \left( 0 \right)$.

Left hand side of \eqref{RefEq42} can be rewritten thus by \ref{RefDefEvalAtomic}:
\begin{align*}
\eval{i} \left( 
\left( I \left( s \right)  \right)
\left( \eval I \funccomp \subterms{t'} \right)
\right)
=
\eval{i} \left(
\left(
\freeInt{X} \left( s \right)
\right)
\left( \eval I \funccomp \subterms{t'} \right)
\right)
\overset{\text{\tiny{\ref{RefDefFreeInt}}}}{=}
\eval{i} \left(
\left\{ \pair{0} {s} \right\}
\conc
\left( \fconc \left( \eval I \funccomp \subterms{t'} \right) \right)
\right),
\end{align*}
where the first step is justified by $v \neq s$.
After setting $t'' := \left\{ \pair{0} {s} \right\}
\conc
\left( \fconc \left( \eval I \funccomp \subterms{t'} \right) \right) \in \terms{S}$,
we notice that $\subterms{t''} = \eval I \funccomp \subterms{t'}$ by definition 
\ref{RefDefSubtermsOfTerm}, so that left side of \eqref{RefEq42} becomes, recalling \ref{RefDefEvalAtomic},
\begin{align}
\label{RefEq43}
\left( i \left( s \right)  \right) \left( \eval{i} \funccomp \subterms{t''} \right)
=
\left( i \left( s \right)  \right) \left( \eval{i} \funccomp \left( 
\eval I \funccomp \subterms{t'}
\right) \right).
\end{align}

We now perform calculations on right hand of \eqref{RefEq42} as well:
\begin{align*}
\eval {\reassign {v}{u}{i}} \left( t' \right) =
\left( \reassign v u i \left( s \right)  \right)
\left(
\eval {\reassign v u i} \funccomp \subterms{t'} 
\right)
\\
\overset{\text{\tiny{ \eqref{RefEq44} }}}{=}
\eval {\reassign {v}{u}{i}} \left( t' \right) =
\left( \reassign v u i \left( s \right)  \right)
\left(
\eval i \funccomp \eval I \funccomp \subterms{t'}
\right) 
=
\left( i \left( s \right)  \right)
\left(
\eval i \funccomp \eval I \funccomp \subterms{t'}
\right),
\end{align*}
with last equality justified again by $v \neq s$.
Comparing this with \eqref{RefEq43} yields the thesis.
\end{proof}

\begin{Prop}
\label{RefThmSubstLemma0}
Given an interpretation $i$, a literal $v$ and a term $t$ of the language $S$
\begin{enumerate}
\item
For any formula $\psi$, $\depth{\subst{v}{t}{\psi}} = 0$ if and only if  $\depth {\psi}=0$.
\\
\item
$
\restrict{\eval{i} \funccomp \substf {v}{t}} { {\wffs{S}}_{, 0} } 
=
\restrict{ \eval{ \reassign{v}{\eval{i} \left( t \right)}{i} }}
{ {\wffs{S}}_{, 0} }
$.
\end{enumerate}
\end{Prop}

\begin{proof}
First thesis descends immediately from \ref{RefDefTermSubst}.
Consider $\psi_0 \in \wffs{S}$, $\depth {\psi_0} = 0$.
We have to show
$
\eval{i} \funccomp \substf {v}{t} \left( \psi_0 \right)
=
\eval{ \reassign{v}{\eval{i} \left( t \right)}{i} } \left( \psi_0 \right)
$.
Set $r := \psi_0 \left( 0 \right)$ and go by cases.
\begin{description}
\item[$r \neq \equiv$]\hfill\\
Then 
\begin{align*}
\left( \eval{i} \funccomp \substf {v}{t}  \right) \left( \psi_0 \right) = 
\eval{i} \left( \subst {v} {t} {\psi_0} \right)
\overset{\text{\tiny{\ref {RefDefTermSubst}, \ref{RefDefEvalAtomic}}}}{=}
\left( i \left( r \right)  \right) \left( 
\eval{i} \funccomp \left( \eval {\reassign {v}{t}{\freeInt{\emp}} } \right)
\funccomp \substrings {\psi_0} \right)
\\
\overset{\text{\tiny{\ref{RefThmSubstLemmaTerm}}}}{=}
\left( i \left( r \right) \right)
\left( 
\eval{\reassign{v}{\eval{i} \left( t \right)}{i}} \funccomp \subterms{\psi_0}
\right)
=
\left( \reassign {v}{\eval{i} \left( t \right)}{i} \left( r \right) \right)
\left( 
\eval{\reassign{v}{\eval{i} \left( t \right)}{i}} \funccomp \subterms{\psi_0}
\right)
\overset{\text{\tiny{\ref{RefDefEvalAtomic}}}}{=}
\eval{\reassign{v}{\eval{i} \left( t \right)}{i}} \left( \psi_0 \right),
\end{align*}
where the second last step took into account that $v \neq r$ (this is because $\ari \left( v \right) = 0$ while $ \ari {r} < 0$).
\item
[$r = \equiv$]
This case is similar to the one above.
It can be retrieved inside \verb|FOMODEL3:8|.
\end{description}
\end{proof}

\begin{Prop}
\label{RefThmDepthSubst}
$ \depth{\subst{v}{t}{\psi}} = \depth{ \psi }$.
\end{Prop}

\begin{proof}
It is an easy induction exploiting \ref{RefThmSubstLemma0} and \ref{RefDefTermSubst}.
\end{proof}

\begin{Lm}
Given $n \in \N$, a set $U \neq \emp$, a language $S$, a literal $v$ and a term $t$ of $S$:
\begin{align}
\label{RefEq40}
\text{for every interpretation } i \text{ of } S \text{ having } U \text{ as universe, it holds }
\notag
&& 
\\
\restrict{\eval{i} \funccomp \substf {v}{t}} { {\wffs{S}}_{, n} } 
=&
\restrict{ \eval{ \reassign{v}{\eval{i} \left( t \right)}{i} }}
{ {\wffs{S}}_{, n} }.
\end{align}
\end{Lm}

\begin{proof}
Set $f := \substf{v}{t}$ (see definition \ref{RefDefTermSubst}).
By induction on $n$.
The base case $n = 0$ is given by \ref{RefThmSubstLemma0}.
Assume \eqref{RefEq40} holds for any $n \leq m$, then consider $\psi \in {\wffs{S}}_{, m+1}$ and an interpretation $i$ of $S$ having universe $U$.
It suffices to show 
$ \eval {i} \left( f \left( \psi \right) \right) = 
\eval {\reassign {v}{\eval{i} \left( t \right)} {i}} \left( \psi \right)$.
Set $s:= \psi \left( 0 \right)$.
We can assume $\depth{\psi} > 0$, and proceed by cases.
\\
\\
\textbf{ Case 1):} $ s \neq \nor $.
\\
Then $ s = v_1 \in \im {\ari^{-1}} {\left\{ 0 \right\}}$, and $\psi = \left\{ \pair{0} {v_1} \right\} \conc \varphi$ for some 
$\varphi \in { \wffs{S}}_{, m}$.
By \ref{RefDefTermSubst},
$
f \left( \psi \right) = \left\{ \pair{0}{v_2} \right\} \conc f \left( \symbsubst {v_1} {v_2}
{\varphi} \right) 
$, with
\begin{align}
\label{RefEq41}
v_2 \notin \left\{ v \right\} \cup \symbof {t, \varphi}.
\end{align}
Assume $ \eval {i} \left( f \left( \psi \right) \right) = 1$. 
Then, by \ref{RefDefEvalCompound}, consider $u_2 \in U$ such that 
\begin{align*}
1 = \eval {\reassign{v_2}{u_2}{i}} \left( f \left( \symbsubst{v_1}{v_2} {\varphi} \right) \right) 
\overset{\text{\tiny{\ref{RefThmDepthSubst}}}}{=}
\eval{\reassign{v}{\eval{i_2} \left( t \right)}{ \reassign{v_2}{u_2}{i} }} 
\left( \symbsubst{v_1}{v_2}{\varphi} \right)
\\
\overset{\text{\tiny{\eqref{RefEq41}}}}{=}
\eval{\reassign{v_2}{u_2}{ \reassign{v}{\eval{i_2} \left( t \right)}{i} }}
\left( \symbsubst{v_1}{v_2}{\varphi} \right)
\overset{\text{\tiny{
\ref{RefThmSubstLemma1}, \eqref{RefEq41}
}}}{=}
\eval{\reassign{v_1}{u_2}{ \reassign{v}{\eval{i_2} \left( t \right)}{i} }}
\left( \varphi \right),
\end{align*}
where we set $i_2 := \reassign{v_2}{u_2}{i}$, and \ref{RefThmDepthSubst} is invoked to trigger induction.
Hence, by \ref{RefDefEvalCompound}
\begin{align*}
1 = \eval {\reassign{v}{\eval{i_2} \left( t \right)}{i}} \left( \left\{ \pair{0}{v_1} \right\}
\conc \varphi \right)
=
\eval {\reassign{v}{\eval{i} \left( t \right)}{i}} \left( \psi \right),
\end{align*}
where last step is due to $v_2 \notin \rng t$.
The proof of 
$ 
\eval {\reassign{v}{\eval{i} \left( t \right)}{i}} \left( \psi \right) = 1
\ra{}
\eval {i} \left( f \left( \psi \right) \right) = 1
$
is very similar.
\\
\\
\textbf{Case 2):} $ s = \nor $.
\\
Then consider $\psi_1, \psi_2 \in { \wffs{S}}_{, m}$ such that 
$\psi = \left\{ \pair{0}{\nor} \right\} \conc \psi_1 \conc \psi_2$.
\begin{align*}
\eval{i} \left( f (\psi) \right)
\overset{\text{\tiny{\ref{RefDefTermSubst}}}}{=}
\eval{i} \left( \pair{0}{\nor} \conc f\left( \psi_1 \right) \conc f\left( \psi_2 \right) \right)
\overset{\text{\tiny{\ref{RefDefEvalCompound}}}}{=}
N \left( \pair 
{\eval {i} \left( f \left( \psi_1 \right) \right)} 
{\eval {i} \left( f \left( \psi_2 \right) \right)} 
\right)
\\
\overset{\text{\tiny{\ref{RefThmDepthSubst}}}}{=}
N \left( 
\pair
{ \eval{ \reassign {v}{ \eval{i} \left( t \right) }{i} } \left( \psi_1 \right) }
{ \eval{ \reassign {v}{ \eval{i} \left( t \right) }{i} } \left( \psi_2 \right) }
\right)
\overset{\text{\tiny{\ref{RefDefEvalCompound}}}}{=}
\eval{ \reassign {v}{ \eval{i} \left( t \right) }{i}} \left( 
\left\{ \pair {0} {\nor} \right\}
\conc \psi_1 \conc \psi_2
\right).
\end{align*}
Again, \ref{RefThmDepthSubst} is needed to deploy induction, and $N$ is a shorthand for the map 
$ \indicator_{\left\{ \pair {0} {0} \right\}}^{2 \cartprod 2}$.
\end{proof}

\chapter{\M{} functors used in the text}
\label{RefSectNotations}
\begin{tabularx}{\linewidth}{ 
>{\hsize=.8\hsize}X>{\hsize=1.4\hsize}X>{\hsize=.8\hsize}X
}
\hline
\verb|f"X| & preimage of the set $X$ through $f$ & $\im{f^{-1}}{X}$
\\
\verb|X/\Y| & set-theoretical intersection & $X \cap Y$
\\
\verb|X\/Y| & set-theoretical union &  $X \cup Y$
\\
\verb|X\Y|  & set-theoretical difference & $X \sdiff Y$
\\
\verb|X\+\Y|& symmetric difference 
& $ \left( A \sdiff B \right) \cup \left( B \sdiff A  \right)$
\\
\verb|[x,y]| & Kuratowski ordered pair & $\left( x,y \right)$
\\
\verb|[:X,Y:]| & cartesian product of  sets  & $X \times Y$
\\
\verb|NAT, INT| & natural numbers and integers & $\mathbb{N, Z}$
\\
\verb|X*| & tuples on $X$ & $X^*$
\\
\verb|n-tuples_on X| & tuples of $n$ letters in $X$ & $X^n$
\\
\verb|Seg n| & &
$\left\{ 1, \ldots, n \right\} $
\\
\verb|<*s*>| & the tuple made of the 
char $s$ & $ \left\{ \pair {0} {s} \right\}$
\\
\verb|p^q| & concatenation of tuples $p$ and $q$ & $p \conc q$
\\
\verb|dom R, rng R| & domain, range of  relation $R$ & 
\\
\verb|p/^n| & 
the tuple $p$ with 
the first $n$ chars removed
& 
\\
\verb|bool X| & the power set of X & $2^X$
\\
\verb|f.x| & the value of the function $f$ in $x$ & $f \left( x \right)$
\\
\verb|id X| & the identity function on $X$ & $ \bigcup_{x \in X} \left\{ x \right\} \cartprod \left\{ x \right\}$
\\
\verb|f +* g| & the pasting of functions $f,g$ & $f \paste {g}$
\\
\verb|curry| & currying & 
$ x \mapsto \lambda x . f\left( x,y \right)$
\\
\verb|f * g| & functional composition & $f \circ g$
\\
\verb|f.:X| & image of the set $X$ through $f$ & $ \im{f}{X} $
\\
\verb|[x,y]`1| \quad \verb|[x,y]`2| & projectors for Kuratowski pairs & 
\begin{varwidth}{.8\hsize}
\begin{align*}
\pair{x}{y} \mapsto x 
\\
\pair{x}{y} \mapsto y
\end{align*}
\end{varwidth}
\\
\verb|Funcs(X,Y)| & the set of functions from $X$ to $Y$ & $\mapsFromTo{X}{Y}$
\\
\verb|PFuncs(X,Y)| & the set of partial functions from $X$ to $Y$ &
\begin{varwidth}{.7\hsize}
\begin{align*}
\bigcup_{x \subseteq X} Y^x
\end{align*}
\end{varwidth}
\\
\verb|iter(f,n)| &  
$n$-th iteration of a function $f$
& $\iter{f}{n}$
\\
\verb|R[*]| &  
transitive closure of $R$
&
\\
\verb|X --> y| & the $y$-constant function on $X$ & 
$ X \to \left\{ y \right\}$
\\
\verb|x .--> y | & function between two singletons & $ \left\{ \pair{x}{y} \right\}$ 
\\
\verb|chi(Y,X)| & characteristic function of $Y \subseteq X$ &
$ \indicator_Y^X$
\end{tabularx}

\backmatter

\end{document}